\documentclass[12pt,reqno]{amsart}
\usepackage{amsmath}
\usepackage{amssymb}
\usepackage{verbatim}
\usepackage{mathrsfs}
\usepackage{graphicx}
\usepackage[font=footnotesize]{caption}%
\usepackage{subcaption}
\usepackage[capitalise]{cleveref}
\usepackage[noadjust]{cite}

\newcommand{\mtz}{\mathcal{T}}
\newcommand{\C}{\mathbb{C}}    
\newcommand{\N}{\mathbb{N}}    
\newcommand{\NN}{\mathbb{N}_0} 
\newcommand{\R}{\mathbb{R}}    
\newcommand{\Z}{\mathbb{Z}}    
\newcommand{\dm}{\mathsf{M}}
\newcommand{\bs}{\backslash}
\newcommand{\cd}{\mathcal{R}}  
\newcommand{\sd}{\mathcal{S}}  
\newcommand{\tz}{\mathcal{T}}  
\newcommand{\sr}{\operatorname{sr}}  
\newcommand{\sm}{\operatorname{sm}}  
\newcommand{\fs}{\operatorname{fsupp}}  
\newcommand{\setsp}{\;:\;}     
\newcommand{\Lp}[1]{L_{#1}(\mathbb{R})}
\newcommand{\gl}{\lambda}

\newcommand{\td}{\boldsymbol{\delta}}  
\newcommand{\CH}[1]{\mathscr{C}^{#1}(\R)} 
\newcommand{\lp}[1]{l_{#1}(\mathbb{Z})}
\newcommand{\pa}{\mathsf{a}}
\newcommand{\pb}{\mathsf{b}}
\newcommand{\pc}{\mathsf{c}}
\newcommand{\wh}{\widehat}
\newcommand{\pp}{\mathsf{p}}

\newcommand{\pu}{\mathsf{u}}
\newcommand{\pv}{\mathsf{v}}
\newcommand{\bo}{\mathscr{O}} 

\newcommand{\be}{ \begin{equation} }
\newcommand{\ee}{ \end{equation} }
\newcommand{\bp}{ \begin{proof} }
\newcommand{\ep}{\hfill \end{proof} }
\newcommand{\imply}{ \Longrightarrow }

\usepackage[top=0.5in,bottom=0.4in,left=0.7in,right=0.7in]{geometry}

\newtheorem{lemma}{Lemma}
\newtheorem{prop}[lemma]{Proposition}
\newtheorem{cor}[lemma]{Corollary}
\newtheorem{theorem}[lemma]{Theorem}

\newtheorem{example}{Example}

\newtheorem{definition}{Definition}

\numberwithin{equation}{section}

\begin{document}

\title[Interpolating Refinable Functions and Subdivision Schemes]
{Interpolating Refinable Functions and $n_s$-step Interpolatory Subdivision Schemes}

\author{Bin Han}

\address{Department of Mathematical and Statistical Sciences,
University of Alberta, Edmonton,\quad Alberta, Canada T6G 2G1.
\quad {\tt bhan@ualberta.ca}\quad {\tt http://www.ualberta.ca/$\sim$bhan}
}

\thanks{Research was supported in part by the Natural Sciences and Engineering Research Council of Canada (NSERC) under grant RGPIN 2024-04991.
}

\thanks{Contact information of the corresponding author Bin Han: E-mail: bhan@ualberta.ca, Phone: 1-587-8828375, Fax: 1-780-4926826,  Web: http://www.ualberta.ca/$\sim$bhan}

\makeatletter \@addtoreset{equation}{section} \makeatother

\begin{abstract}
Standard interpolatory subdivision schemes and their underlying interpolating refinable functions are of interest in CAGD, numerical PDEs, and approximation theory. Generalizing these notions, we introduce and study $n_s$-step interpolatory $\dm$-subdivision schemes and their interpolating $\dm$-refinable functions with
$n_s\in \N \cup\{\infty\}$ and
a dilation factor $\dm\in \N\bs\{1\}$.
We completely characterize $\mathscr{C}^m$-convergence and smoothness of $n_s$-step interpolatory subdivision schemes and their interpolating $\dm$-refinable functions in terms of their masks. Inspired by $n_s$-step interpolatory stationary subdivision schemes, we further introduce the notion of $r$-mask quasi-stationary subdivision schemes, and then we characterize their $\mathscr{C}^m$-convergence and smoothness properties using only their masks. Moreover, combining $n_s$-step interpolatory subdivision schemes with $r$-mask quasi-stationary subdivision schemes, we can obtain $r n_s$-step interpolatory subdivision schemes.
Examples and construction procedures of convergent $n_s$-step interpolatory $\dm$-subdivision schemes are provided to illustrate our results with dilation factors $\dm=2,3,4$. In addition, for the dyadic dilation $\dm=2$ and $r=2,3$, using $r$ masks with only two-ring stencils, we provide examples of $\mathscr{C}^r$-convergent $r$-step interpolatory $r$-mask quasi-stationary dyadic subdivision schemes.
\end{abstract}


\maketitle

\pagenumbering{arabic}

\section{Introduction and Main Results}
\label{sec:intro}

In this paper we are interested in interpolatory subdivision schemes and their interpolating refinable functions, because such functions are the backbone for building wavelets for image processing and numerical PDEs, and in computer aided geometric design (CAGD) for developing fast computational algorithms. Throughout this paper, a positive integer $\dm\in \N\bs\{1\}$ is called \emph{a dilation factor}.
By $\lp{0}$ we denote the space of all finitely supported sequences $a=\{a(k)\}_{k\in \Z}: \Z \rightarrow \C$. For any finitely supported sequence $a\in \lp{0}$, its symbol $\tilde{\pa}(z)$ is a Laurent polynomial defined by
%
\[
\tilde{\pa}(z):=\sum_{k\in \Z} a(k) z^k,\qquad z\in \C\bs\{0\}.
\]
For $a\in \lp{0}$ satisfying $\sum_{k\in \Z} a(k)=1$ (i.e., $\tilde{\pa}(1)=1$), we can define a compactly supported distribution $\phi$ through the Fourier transform $\wh{\phi}(\xi):=\prod_{j=1}^\infty \tilde{\pa}(e^{-i \dm^{-j}\xi})$ for $\xi \in \R$, where the Fourier transform here is defined to be $\wh{f}(\xi):=\int_\R f(x) e^{-ix\xi}dx, \xi\in \R$ for integrable functions $f$ and can be naturally extended to tempered distributions through duality. Note that $\wh{\phi}(0)=1$. It is well known and straightforward that $\phi$ is \emph{an $\dm$-refinable function} satisfying the following refinement equation:
\be \label{refeq}
\phi=\dm \sum_{k\in \Z} a(k)\phi(\dm \cdot-k),\quad \mbox{or equivalently},\quad
\wh{\phi}(\dm \xi)=\tilde{\pa}(e^{-i\xi}) \wh{\phi}(\xi),
\ee
where the sequence $a\in \lp{0}$ in \eqref{refeq} is often called the \emph{mask} for the $\dm$-refinable function $\phi$.

An interpolating function $\phi$ is
a continuous function on the real line $\R$ such that $\phi(k)=\td(k)$ for all $k\in \Z$, where $\td$ is the \emph{Dirac sequence} such that $\td(0)=1$ and $\td(k)=0$ for all $k\in \Z\bs \{0\}$.  The simplest example of compactly supported interpolating functions is probably the hat function
\be \label{hatfunc}
\phi(x):=\max(1-|x|,0), \qquad x\in \R,
\ee
which is used in numerical PDEs and approximation theory. Note that the hat function  $\phi$ in \eqref{hatfunc} is $\dm$-refinable with a mask $a\in \lp{0}$ given by $\tilde{\pa}(z)=\dm^{-2}z^{1-\dm}(1+z+\cdots+z^{\dm-1})^2$.
Therefore, the hat function $\phi$ in \eqref{hatfunc} is also used to build wavelets  for their applications to image processing and computational mathematics.
For $s_a\in \R$, generalizing standard interpolating functions and motivated by \cite{grv22,rv20,rom19}, in this paper we consider a more general class of interpolating functions $\phi$ satisfying
\be \label{intphi:sa}
\phi(s_a+k)=\td(k),\qquad \forall\; k\in \Z.
\ee
For simplicity, we call $\phi$ \emph{an $s_a$-interpolating function} if it is continuous and satisfies \eqref{intphi:sa}. For a given function $f$ on $\R$ and a mesh size $h>0$,
the interpolation property in \eqref{intphi:sa} guarantees that $g(x):=\sum_{k\in \Z} f(hk)\phi(h^{-1} x+s_a-k)$ interpolates $f$ in the sense that $g(hk)=f(hk)$ for all $k\in \Z$.

In this paper, we are interested in $s_a$-interpolating $\dm$-refinable functions and their intrinsic connections to interpolatory subdivision schemes.
Except spline refinable functions such as the hat function in \eqref{hatfunc}, an $\dm$-refinable function $\phi$ with a mask $a\in \lp{0}$ generally cannot have any analytic expression (e.g., see \cite[Section~6.1]{hanbook}).
Consequently, a subdivision scheme is often employed to approximate a refinable function $\phi$ using its mask $a\in \lp{0}$. By $\lp{}$ we denote the space of all sequences $v=\{v(k)\}_{k\in \Z}:\Z\rightarrow \C$.
The \emph{$\dm$-subdivision operator} $\sd_{a,\dm}: \lp{}\rightarrow \lp{}$ is defined to be
\be \label{sd}
[\sd_{a,\dm} v](j):=\dm \sum_{k\in \Z} v(k) a(j-\dm k),\quad j\in \Z, v\in \lp{}.
\ee
For many applications such as CAGD and numerical algorithms, the subdivision operator in \eqref{sd} is often implemented using convolution and coset masks.
For $u,v\in \lp{0}$, their convolution is defined to be $[u*v](j)=\sum_{k\in \Z} u(k) v(j-k)$ for $j\in \Z$. Note that the symbol of $u*v$ is just $\tilde{\pu}(z) \tilde{\pv}(z)$.
For a mask $a\in \lp{0}$ and $\gamma\in \Z$, its $\gamma$-coset mask $a^{[\gamma:\dm]}$ is defined to be
\be \label{coset}
a^{[\gamma:\dm]}(k):=a(\gamma+\dm k),\qquad k\in \Z.
\ee
Then the definition of the $\dm$-subdivision operator $\sd_{a,\dm}$ in \eqref{sd} can be equivalently expressed as
\[
[\sd_{a,\dm} v]^{[\gamma:\dm]}(j)=
[\sd_{a,\dm} v](\gamma+\dm j)=\dm \sum_{k\in \Z} v(k) a(\gamma+\dm (j-k))=
\dm [v* a^{[\gamma: \dm]}](j),\qquad j\in \Z.
\]
Hence, for $\gamma=0,\ldots,\dm-1$, each $\dm a^{[\gamma:\dm]}$ is called a stencil in CAGD and is an $n$-ring stencil if $a^{[\gamma:\dm]}$ is supported inside $[-n,n+\td(\gamma)-1]$. Note that a mask $a\in \lp{0}$ has at most $n$-ring stencils if and only if the mask $a$ is supported inside $[-\dm n, \dm n]$. In CAGD and other applications, it is highly desired to have subdivision schemes with small $n$-ring  stencils for fast implementation and for reducing the number of special  subdivision rules near extraordinary vertices of subdivision surfaces. However, this greatly restricts the choices of desired subdivision schemes and smooth refinable functions. Consequently, new settings and ideas are needed to circumvent this obstacle.

Starting from an initial sequence $v\in \lp{}$, an $\dm$-subdivision scheme iteratively computes a sequence $\{\sd_{a,\dm}^n v\}_{n=1}^\infty$ of subdivision data.
The backward difference operator $\nabla: \lp{}\rightarrow \lp{}$ is defined to be
\[
[\nabla v](k):=v(k)-v(k-1),\qquad k\in \Z, v\in \lp{}
\quad \mbox{with the convention}\quad \nabla^0 v:=v.
\]
We now recall the definition of the $\mathscr{C}^m$-convergence of a (stationary) $\dm$-subdivision scheme below (e.g., see \cite[Theorem~2.1]{hj06}) and discuss the notion of $\infty$-step interpolatory subdivision schemes:

\begin{definition}\label{def:sd}
{\rm Let $\dm\in \N\bs\{1\}$ be a dilation factor. Let $m\in \NN:=\N\cup\{0\}$ and $a\in \lp{0}$ be a finitely supported mask satisfying $\sum_{k\in \Z} a(k)=1$. We say that \emph{the $\dm$-subdivision scheme with mask $a\in \lp{0}$ is $\mathscr{C}^m$-convergent} if for every initial input sequence $v\in \lp{}$, there exists a continuous function $\eta_v\in \CH{m}$ such that for every constant $K>0$,
\be \label{converg}
\lim_{n\to\infty} \max_{k\in \Z\cap [-\dm^n K, \dm^n K]} |\dm^{jn} [\nabla^j \sd_{a,\dm}^n v](k)-\eta_v^{(j)}(\dm^{-n} k)|=0, \qquad \mbox{for all } j=0,\ldots,m,
\ee
where $\eta_v^{(j)}$ stands for the $j$th derivative of the function $\eta_v$. In addition, we say that a $\mathscr{C}^0$-convergent $\dm$-subdivision scheme with mask $a\in \lp{0}$ is \emph{$\infty$-step interpolatory with center $s_a\in \R$} if
\be \label{inftystep}
\eta_v(s_a+k)=v(k),\qquad \forall\; k\in \Z, v\in \lp{}.
\ee
}\end{definition}

For a convergent subdivision scheme with mask $a\in \lp{0}$, the limit function $\eta_{v}$ in \cref{def:sd} with the initial input sequence $v=\td$ is called its \emph{basis function}, which must be the $\dm$-refinable function $\phi$ with the mask $a$ (e.g., see \cref{sec:proof} for details).
For every $v\in \lp{}$, noting that $v=\sum_{k\in \Z} v(k) \td(\cdot-k)$ and the $\dm$-subdivision scheme is linear, we have
$\eta_v=\sum_{k\in \Z} v(k) \phi(\cdot-k)$.
Now it is evident that $\phi$ is $s_a$-interpolating (i.e., $\phi(s_a+k)=\td(k)$ for all $k\in \Z$) if and only if \eqref{inftystep} holds, i.e., its convergent subdivision scheme must be $\infty$-step interpolatory with the center $s_a$.
Therefore, to study the convergence of a subdivision scheme, it is critical to investigate its $\dm$-refinable function $\phi$ with a mask $a\in \lp{0}$. If $\phi$ is a standard interpolating function, i.e., $\phi$ is $0$-interpolating, then $\eta_v(k)=v(k)$ for all $k\in \Z$ and $v\in \lp{}$.
Such a subdivision scheme is called \emph{a standard interpolatory $\dm$-subdivision scheme},
whose mask $a$ must be $\dm$-interpolatory satisfying the condition $a(\dm k)=\dm^{-1}\td(k)$ for all $k\in \Z$.
Standard interpolatory subdivision schemes have been extensively studied and constructed in the literature, for example, see \cite{cdm91,cr13,dd89,dvms00,dl02,dl87,hanbook,hj99} and references therein.

Masks having the symmetry property are of particular interest in CAGD and wavelet analysis (e.g., see \cite{dl02,han98,hanbook,hm21}). For a mask $a\in \lp{0}$, we say that $a$ is \emph{symmetric} about the point $c_a/2$ if
\be \label{symmask}
a(c_a-k)=a(k) \qquad \forall\; k\in \Z \quad \mbox{with}\quad c_a\in \Z.
\ee
A subdivision scheme with a symmetric mask $a$ in \eqref{symmask} for an odd (or even) integer $c_a$
is called a dual (or primal) subdivision scheme in CAGD. As pointed out in \cite{grv22},
an open question was asked by M. Sabin: Does there exist an interpolatory dual subdivision scheme which is similar to interpolatory primal subdivision schemes? This question has been recently answered by L.~Romani and her collaborators in the interesting papers \cite{grv22,rv20,rom19}, showing that this is only possible for $\dm>2$. Moreover, for dilation factors $\dm>2$, interesting results and several examples are presented in  \cite{grv22,rv20,rom19,v23}, which have greatly motivated this paper. In particular, for $\dm=2$ (which is the most common choice in CAGD and wavelet analysis), dropping the symmetry property in \eqref{symmask}, we are interested in whether there exists an $s_a$-interpolating $2$-refinable function with $s_a\not \in \Z$.
This further motivates us to characterize all $s_a$-interpolating $\dm$-refinable functions and their $\infty$-step interpolatory $\dm$-subdivision schemes in terms of their masks. Indeed, we show in \cref{ex:M2,ex:M2a} that there are $s_a$-interpolating $2$-refinable functions with $s_a\in \{\frac{1}{3},\frac{1}{7}\}$ and their dyadic subdivision schemes are $2$-step or $3$-step interpolatory.

To present our main results in this paper on $s_a$-interpolating $\dm$-refinable functions and their associated $\dm$-subdivision schemes, we recall some necessary definitions.
The convergence and smoothness of a subdivision scheme are linked with the sum rules of a mask $a\in \lp{0}$. For $J\in \NN$, we say that a mask $a$ has \emph{order $J$ sum rules with respect to a dilation factor $\dm$} if
\be \label{sr:2}
\sum_{k\in \Z} \pp(\gamma+\dm k) a(\gamma+\dm k)=\dm^{-1}
\sum_{k\in \Z} \pp(k) a(k),\qquad \forall\; \pp\in \Pi_{J-1}, \gamma\in \Z,
\ee
where $\Pi_{J-1}$ is the space of all polynomials of degree less than $J$.
For convenience, we define $\sr(a,\dm):=J$ with $J$ in \eqref{sr:2} being the largest such an integer.
Note that a polynomial sequence $\{\pp(k)\}_{k\in \Z}$ on $\Z$ can be uniquely identified with its underlying polynomial $\pp$ on the real line $\R$.

The convergence of a subdivision scheme can be characterized by a technical quantity $\sm_p(a,\dm)$, which is introduced in \cite{han03}. For a mask $a\in \lp{0}$ and $1\le p\le \infty$, we define (e.g., see \cite{han98,han03,han03smaa,hanbook})
\be \label{sm}
\sm_p(a,\dm):=\tfrac{1}{p}-\log_{\dm} \rho_J(a,\dm)_p
\quad \mbox{with}\quad \rho_J(a,\dm)_p:=\limsup_{n\to\infty} \| \nabla^{J}\sd_{a,\dm}^n \td\|_{\lp{p}}^{1/n}, \; J:=\sr(a,\dm).
\ee
It is known that an $\dm$-subdivision scheme with mask $a\in \lp{0}$ is $\mathscr{C}^m$-convergent if and only if $\sm_\infty(a,\dm)>m$ (e.g., see \cite[Theorem~4.3]{han03} or \cite[Theorem~2.1]{hj06}).
We shall discuss how to effectively compute and estimate the smoothness exponents $\sm_2(a,\dm)$ and $\sm_\infty(a,\dm)$ in Subsection~\ref{subsec:sm}.

As we shall see in \cref{thm:int}, an $\dm$-refinable function $\phi$ with a mask $a\in \lp{0}$ is $s_a$-interpolating if and only if its $\dm$-subdivision scheme is $\mathscr{C}^0$-convergent and $\infty$-step interpolatory. However, for special centers $s_a$, its subdivision scheme can be $n_s$-step (instead of $\infty$-step) interpolatory for some finite integer $n_s\in \N$ in the following sense:

\begin{definition}\label{def:nsstep}
{\rm For $n_s\in \N$,
we say that an $\dm$-subdivision scheme with a mask $a\in \lp{0}$ is \emph{$n_s$-step interpolatory} if
\be \label{cond:intsubdiv}
[\sd_{a,\dm}^{n_s} v](s+\dm^{n_s} k)=v(k),\qquad \forall\; k\in \Z, v\in \lp{}
\ee
for some shift $s\in \Z$. We often take $n_s\in \N$ to be the smallest integer such that \eqref{cond:intsubdiv} holds.
}\end{definition}

Using the definition of the subdivision operator in \eqref{sd}, we can directly deduce from \eqref{cond:intsubdiv} that
\[
[\sd_{a,\dm}^{q n_s} v]((I+\dm^{n_s}+
\cdots+\dm^{(q-1) n_s})s+\dm^{qn_s} k)=v(k),\qquad \forall\; k\in \Z, q\in \N, v\in \lp{}.
\]
Hence, the subdivision scheme in \cref{def:nsstep} interpolates the data after every $n_s$-step subdivision and the same subdivision scheme is obviously $qn_s$-step interpolatory with the shift
$(I+\dm^{n_s}+
\cdots+\dm^{(q-1) n_s})s$.
Note that a standard interpolatory subdivision scheme is simply $1$-step interpolatory with the shift $s=0$ and $n_s=1$ in \cref{def:nsstep}.
Moreover, if an $n_s$-step interpolatory $\dm$-subdivision scheme with a mask $a\in \lp{0}$ in \cref{def:nsstep} is $\mathscr{C}^0$-convergent, then by \eqref{converg} and \eqref{cond:intsubdiv}, the $\dm$-subdivision scheme must be also $\infty$-step interpolatory with the center $s_a:=(\dm^{n_s}-1)^{-1} s$ and its $\dm$-refinable function $\phi$ must be $s_a$-interpolating.

For special choices of $s_a\in \R$, the following result, whose proof is given in \cref{sec:proof},
characterizes all
$s_a$-interpolating $\dm$-refinable functions and their $\mathscr{C}^m$-convergent $\infty$-step interpolatory $\dm$-subdivision schemes in terms of their masks.

\begin{theorem}\label{thm:int}
Let $\dm\in \N\bs\{1\}$ be a dilation factor. Let $m\in \NN$ and $a\in \lp{0}$ be a finitely supported mask with $\sum_{k\in \Z} a(k)=1$. Define a compactly supported distribution $\phi$ by $\wh{\phi}(\xi):=\prod_{j=1}^\infty \tilde{\pa}(e^{-i\dm^{-j}\xi})$ for $\xi \in \R$.
For a real number $s_a\in \R$ satisfying
\be \label{cond:sa}
\dm^{m_s}(\dm^{n_s}-1)s_a\in \Z \quad\mbox{for some $m_s\in \NN$ and $n_s\in \N$},
\ee
the following statements are equivalent to each other:
\begin{enumerate}
\item[(1)] The $\dm$-refinable function $\phi$ with mask $a$ belongs to $\CH{m}$ and is $s_a$-interpolating as in \eqref{intphi:sa}.
\item[(2)] $\sm_\infty(a,\dm)>m$ and there is a finitely supported sequence $w\in \lp{0}$ such that
\begin{align}
&[A_{m_s}*w](\dm^{m_s} k)=\dm^{-m_s} \td(k)\qquad \forall\; k\in \Z, \label{cond:ms}\\
&[A_{n_s}*w](\dm^{m_s}(\dm^{n_s}-1)s_a+\dm^{n_s}k)=\dm^{-n_s} w(k),\qquad \forall\; k\in \Z,\label{cond:ns}
\end{align}
where the finitely supported masks $A_n\in \lp{0}$ are defined to be
\be \label{An}
A_n:=\dm^{-n} \sd_{a,\dm}^n \td, \quad \mbox{or equivalently},\quad \widetilde{\mathbf{A}_n}(z):=\tilde{\pa}(z^{\dm^{n-1}})
\tilde{\pa}(z^{\dm^{n-2}})\cdots \tilde{\pa}(z^\dm) \tilde{\pa}(z).
\ee
For the particular case $m_s=0$, the conditions in \eqref{cond:ms} and \eqref{cond:ns} together are equivalent to
\be \label{cond:ms=0}
A_{n_s}((\dm^{n_s}-1)s_a+\dm^{n_s}k)=\dm^{-n_s} \td(k)\qquad \forall\; k\in \Z,
\ee
because $w=\td$ is the unique solution to \eqref{cond:ms} with $m_s=0$ due to $A_0=\td$ and $\td*w=w$.

\item[(3)] The $\dm$-subdivision scheme with mask $a$ is $\mathscr{C}^m$-convergent and $\infty$-step interpolatory with the center $s_a$ as in the sense of \cref{def:sd}. For the particular case $m_s=0$, the $\dm$-subdivision scheme with mask $a$ is further $n_s$-step interpolatory with the integer shift $(\dm^{n_s}-1)s_a$ as in the sense of \cref{def:nsstep}.
\end{enumerate}
Moreover, any of the above items (1)--(3) implies that the $\dm$-subdivision scheme with mask $a$ has the following polynomial-interpolation property:
\be \label{poly:int}
\sd_{a,\dm}^n \pp=\pp(\dm^{-n}(s_a+\cdot)-s_a),\qquad \forall\; n\in \N, \pp\in \Pi_{\sr(a,\dm)-1}.
\ee
\end{theorem}

The set of all $s_a\in \R$ satisfying \eqref{cond:sa} is $\cup_{m_s=0}^\infty \cup_{n_s=1}^\infty [\dm^{-m_s}(\dm^{n_s}-1)^{-1}\Z]$, which is dense in $\R$. Moreover, $s_a \in \R$ satisfies \eqref{cond:sa} if and only if $[0,1)\cap (\cup_{j=0}^\infty [\dm^j s_a+\Z])$ is a finite set. We shall explain in Subsection~\ref{subsec:sa} the condition \eqref{cond:sa} on $s_a$ in details, which is rooted in the fundamental problem of how to determine the exact (not approximated) value $\phi(s_a)$ of a continuous $\dm$-refinable function $\phi$ (not necessarily interpolating) within finitely many steps using only its mask $a\in \lp{0}$.

For many applications such as curve/surface generation in CAGD and wavelet methods for numerical PDEs and image processing, $d$-dimensional refinable functions $\phi$ with the dilation matrix $2I_d$ are highly desired to possess high smoothness (e.g., $\phi\in \mathscr{C}^2(\mathbb{R}^d)$ in CAGD for continuity of the curvature of subdivision curves or surfaces), interpolation property (e.g., interpolating curves/functions in CAGD and numerical PDEs), and masks of small supports (for fast implementation and boundary treatment in applications, e.g., see \cite{hanbook,hm21}).
However, these highly desired properties of $\phi$ are mutually conflicting to each other. For example,
\cite[Theorem~3.5 and Corollary~4.3]{han99} shows that there are no standard interpolating $2I_d$-refinable functions $\phi\in \mathscr{C}^2(\R^d)$ whose masks can be supported inside $[-3,3]^d$. Consequently, it is impossible to have $\mathscr{C}^2$-convergent (dyadic) $2I_d$-subdivision schemes with two-ring stencils.
Motivated by the $n_s$-step interpolatory stationary subdivision schemes in \cref{thm:int} and \cite{han99},
we shall show that this can be remedied by introducing the notion of $r$-mask quasi-stationary subdivision schemes.

Let $r\in \N$ and $a_1,\ldots, a_{r}\in \lp{0}$ be finitely supported masks. For $n\in \N$, we define
\be \label{sd:new}
\sd_{a_1,\ldots, a_r,\dm}^{n, r}:=
\begin{cases} [\sd_{a_r,\dm}\cdots \sd_{a_1,\dm}]^{\lfloor n/r\rfloor}, &\text{if $n\in r\N$,}\\
\sd_{a_{\{n\}},\dm}\sd_{a_{\{n\}-1},\dm}\cdots \sd_{a_1,\dm}
[\sd_{a_r,\dm}\cdots \sd_{a_1,\dm}]^{\lfloor n/r\rfloor},
&\text{if $n\not\in r\N$,}
\end{cases}
\ee
where $\lfloor x\rfloor$ is the largest integer not greater than $x$ and $\{n\}:=n-r \lfloor n/r\rfloor\in \{0,\ldots,r-1\}$.
For any initial input sequence $v\in \lp{}$, we obtain a sequence $\{\sd_{a_1,\ldots, a_r,\dm}^{n, r} v\}_{n=1}^\infty$ of $\dm$-subdivision data. In other words, we apply the $\dm$-subdivision operators on the initial data $v\in \lp{}$ using the masks $\{a_1,\ldots,a_r\}$ in the $r$-periodic ordering fashion $a_1,\ldots,a_r, a_1,\ldots,a_r,\ldots$.
Therefore, such a subdivision scheme using masks $\{a_1,\ldots,a_r\}$ will be called an \emph{$r$-mask quasi-stationary} subdivision scheme.

Similar to \cref{def:sd}, we have

\begin{definition}\label{def:sd:qs}
{\rm Let $\dm\in \N\bs\{1\}$ be a dilation factor and $r\in \N$. Let $m\in \NN$ and $a_1,\ldots,a_r\in \lp{0}$ be finitely supported masks satisfying $\sum_{k\in \Z} a_\ell(k)=1$ for $\ell=1,\ldots,r$. We say that \emph{the $r$-mask quasi-stationary $\dm$-subdivision scheme with masks $\{a_1,\ldots,a_r\}$ is $\mathscr{C}^m$-convergent} if for every initial input sequence $v\in \lp{}$, there exists a function $\eta_v\in \CH{m}$ such that for every constant $K>0$,
\be \label{converg:qs}
\lim_{n\to\infty} \max_{k\in \Z\cap [-\dm^n K, \dm^n K]} |\dm^{jn} [\nabla^j \sd_{a_1,\ldots, a_r,\dm}^{n, r} v](k)-\eta_v^{(j)}(\dm^{-n} k)|=0, \qquad \mbox{for all } j=0,\ldots,m.
\ee
%
}\end{definition}

Obviously, \cref{def:sd:qs} with $r=1$ becomes \cref{def:sd}. We now characterize the $\mathscr{C}^m$-convergent quasi-stationary subdivision schemes in the following result, whose proof is presented in \cref{sec:qs}.

\begin{theorem}\label{thm:qs}
Let $\dm\in \N\bs\{1\}$ be a dilation factor and $r\in \N$. Let $m\in \NN$ and $a_1,\ldots,a_r\in \lp{0}$ be finitely supported masks with $\sum_{k\in \Z} a_\ell(k)=1$ for $\ell=1,\ldots,r$.
Define a mask $a\in \lp{0}$ by
\be \label{mask:a1r}
a:=\dm^{-r} \sd_{a_r,\dm}\cdots \sd_{a_1,\dm}\td,\quad \mbox{that is},\quad
\tilde{\pa}(z):=\widetilde{\pa_1}(z^{\dm^{r-1}})\cdots \widetilde{\pa_{r-1}}(z^\dm)\widetilde{\pa_r}(z).
\ee
Define the compactly supported $\dm^r$-refinable function/distribution $\phi$ via the Fourier transform $\wh{\phi}(\xi):=\prod_{j=1}^\infty \tilde{\pa}(e^{-i \dm^{-r j}\xi})$ for $\xi\in \R$. Note that $\wh{\phi}(0)=1$. Then
the $r$-mask quasi-stationary $\dm$-subdivision scheme with masks $\{a_1,\ldots,a_r\}$ is $\mathscr{C}^m$-convergent if and only if
\be \label{cond:qs}
\sm_\infty(a,\dm^r)>m
\quad \mbox{and}\quad
\sr(a_\ell,\dm)>m,\qquad \forall\; \ell=1,\ldots,r.
\ee
Moreover, for every $v\in \lp{}$, the limit function $\eta_v$ in \eqref{converg:qs} of \cref{def:sd:qs} must be given by $\eta_v=\sum_{k\in \Z} v(k)\phi(\cdot-k)$. 
\end{theorem}


The contributions and potential usefulness of the results in \cref{thm:int,thm:qs} are outlined below:

\begin{enumerate}
\item[(1)] We introduce the notion of $r$-mask quasi-stationary subdivision schemes and fully characterize them in \cref{thm:qs}.
    This notion and examples in Section~\ref{sec:alg} offer new $n_s$-step interpolatory $2$-subdivision schemes for CAGD and new interpolating refinable functions for numerical PDEs. In particular, for $r=2,3$, we obtain $\mathscr{C}^r$-convergent $r$-step interpolatory $r$-mask quasi-stationary dyadic subdivision schemes using only two-ring stencils in \cref{ex2,ex3}. Their tensor products obviously offer $d$-dimensional $\mathscr{C}^r$-convergent $r$-step interpolatory $r$-mask quasi-stationary dyadic $2I_d$-subdivision schemes using only two-ring stencils.

\item[(2)] We introduce the notion of $s_a$-interpolating refinable functions and characterize them in \cref{thm:int}, leading to $n_s$-step interpolatory subdivision schemes with $n_s\in \N\cup\{\infty\}$. \cref{ex:M2} shows the existence of $\frac{1}{3}$-interpolating $2$-refinable functions $\phi\in \mathscr{C}^1(\R)$ and their $\mathscr{C}^1$-convergent $2$-step interpolatory dyadic subdivision schemes. For $\dm=3,4$, \cref{ex:M3,ex:M4} obtain several interpolatory dual $\dm$-subdivision schemes with symmetric masks such that their $n_s$-interpolatory $\dm$-subdivision schemes are $\mathscr{C}^2$-convergent with $n_s\in \{2,\infty\}$.
\item[(3)] For masks symmetric about $c_a/2$, interpolatory dual $\dm$-subdivision schemes with $\dm>2$ have been studied in the interesting papers \cite{grv22,rv20,rom19,v23}. Because the basis functions there must be $s_a$-interpolating with $s_a=\frac{c_a}{2(\dm-1)}$ (see details after \cref{prop:lpm}),
    the necessary and sufficient \cref{thm:int} can be applied to this special case and only uses masks $a$ without requiring symmetry.
    While \cite[Theorems~3.4 and 3.5]{rv20} involves both masks $a$ and values of $\phi$ on $\frac{n}{T}+\Z$ in \cite[(11)]{rv20}, which is further addressed in
     \cite[Assumptions~1 and~2]{grv22}.
     Interestingly, we construct in \cref{ex:M2:M4} a $\mathscr{C}^2$-convergent $\infty$-interpolatory $2$-mask quasi-stationary dyadic $2$-subdivision scheme with masks $\{a_1,a_2\}$, which leads to a symmetric $\frac{1}{6}$-interpolating $\dm$-refinable function and a $\mathscr{C}^2$-convergent interpolatory dual $\dm$-subdivision scheme with $\dm=4$.

\item[(4)] \cref{thm:int,thm:qs}, which can be combined as in \cref{cor:qs}, offer new interpolating refinable functions (e.g., see \cref{ex3,ex:M2:M4}) and wavelets which are of interest in their applications to numerical PDEs and image processing. Interestingly, the notion of $r$-mask quasi-stationary subdivision schemes in \cref{thm:qs} offers a flexible framework for constructing non-traditional wavelets with added features, for which we shall leave it as a future research problem.
\end{enumerate}

The structure of the paper is as follows.
In \cref{sec:alg}, we provide several examples and construction procedures of quasi-stationary $2$-subdivision schemes and $n_s$-step interpolatory $\dm$-subdivision schemes. We also discuss how to estimate the smoothness exponent $\sm_\infty(a,\dm)$ and explain in details the condition \eqref{cond:sa} and roles on $s_a$.
In \cref{sec:proof}, we first develop some auxiliary results and then we prove \cref{thm:int}. In \cref{sec:qs}, we shall prove \cref{thm:qs} and then we shall present a result in \cref{cor:qs} by combining both \cref{thm:int,thm:qs} for $rn_s$-step interpolatory $r$-mask quasi-stationary $\dm$-subdivision schemes with masks $\{a_1,\ldots,a_r\}$.

\section{Examples of $n_s$-step Interpolatory (Quasi)-Stationary Subdivision Schemes}
\label{sec:alg}

Applying \cref{thm:int,thm:qs}, we discuss how to construct desired masks $a$ for $s_a$-interpolating refinable functions and their $n_s$-step interpolatory subdivision schemes. We first discuss how to estimate the smoothness exponent $\sm_\infty(a,\dm)$ and then explain the condition \eqref{cond:sa} on $s_a$.
Then we present some examples of convergent $r$-step interpolatory $r$-mask quasi-stationary subdivision schemes using \cref{thm:int,thm:qs} with the commonly used dilation factor $\dm=2$. Next, we provide
construction procedures and
several examples of masks for $s_a$-interpolating $\dm$-refinable functions using \cref{thm:int}. Finally, we apply our constructed $n_s$-step interpolatory subdivision schemes to CAGD for generating smooth subdivision curves and we explain the roles of $s_a$ in CAGD.

To present our examples and discuss their construction, for a mask $a\in \lp{0}$, we define the filter support $\fs(a)$ to be the smallest interval $[l_a,h_a]$ with $l_a,h_a\in \Z$ such that
$a(l_a)a(h_a)\ne 0$ and
$a(k)=0$ for all $k\in \Z\bs [l_a,h_a]$.
Then its $\dm$-refinable function $\phi$ must be supported inside $[\frac{l_a}{\dm-1},\frac{h_a}{\dm-1}]$.

\subsection{Estimate and optimize the smoothness quantity $\sm_\infty(a,\dm)$}
\label{subsec:sm}

To construct desired masks in \cref{thm:int,thm:qs} for smooth interpolating refinable functions, we first discuss how to calculate and estimate the smoothness exponent $\sm_\infty(a,\dm)$ in \eqref{sm}. Because masks $a$ constructed in \cref{thm:int,thm:qs} often have
some free parameters, we shall discuss how to search among these free parameters in the masks $a$ such that the smoothness exponent $\sm_\infty(a,\dm)$ is as large as possible.

The smoothness exponent $\sm_p(a,\dm)$ defined in \eqref{sm} for $1\le p\le \infty$ plays a critical role in studying subdivision schemes and wavelets. In CAGD, an $\dm$-subdivision scheme with mask $a\in \lp{0}$ is $\mathscr{C}^m$-convergent if and only if $\sm_\infty(a,\dm)>m$ (e.g., see \cite[Theorem~2.1]{hj06} or \cite[Theorem~7.3.1]{hanbook}).
Even for arbitrary matrix masks $a\in (\lp{0})^{r\times r}$, the vector $\dm$-subdivision scheme with mask $a$ is $\mathscr{C}^m$-convergent if and only if $\sm_\infty(a,\dm)>m$ (e.g., see \cite[Theorem~1]{han24}). Moreover, the convergence rate of the vector subdivision scheme is also determined by $\sm_\infty(a,\dm)$, e.g., see \cite[Theorem~2]{han24}.
To study refinable functions in wavelet analysis, recall that the cascade operator $\cd_{a,\dm}: L_p(\R)\rightarrow L_p(\R)$ is defined to be $\cd_{a,\dm} f:=\dm \sum_{k\in \Z} a(k)f(\dm\cdot-k)$. Let $\phi$ be the $\dm$-refinable function with a mask $a$. Then $\phi$ is a fixed point of $\cd_{a,\dm}$, i.e., $\cd_{a,\dm} \phi=\phi$.
From the refinement equation \eqref{refeq}, one can easily see that $\cd_{a,\dm}^n f=\sum_{k\in \Z} [\sd_{a,\dm}^n \td](k)f(\dm^n\cdot-k)$, i.e.,
a cascade algorithm is closely linked to a subdivision scheme for studying the convergence of the cascade algorithm $\{\cd_{a,\dm}^n f\}_{n=1}^\infty$ in the Sobolev space $W^m_p(\R)$ (e.g., see \cite{han03,hanbook,hj98}).
Then a cascade algorithm with mask $a$ converges in $W^m_p(\R)$ if and only if $\sm_p(a,\dm)>m$ (e.g., see \cite[Theorem~4.3]{han03} or \cite[Theorem~5.6.16]{hanbook}).  The $L_p$-smoothness exponent $\sm_p(\phi)$ is defined later in \eqref{sm:f}. Then $\sm_p(\phi)\ge \sm_p(a,\dm)$. Moreover, $\sm_p(\phi)=\sm_p(a,\dm)$ holds if the integer shifts of $\phi$ are stable, i.e., $\mbox{span}\{\wh{\phi}(\xi+2\pi k)\setsp k\in \Z\}=\C$ for every $\xi\in \R$.

Generally, computing $\sm_\infty(a,\dm)$ is not an easy task, but we can often estimate $\sm_\infty(a,\dm)$. 
Let $a\in \lp{0}$ be a finitely supported mask. Define $J:=\sr(a,\dm)$, the highest order sum rules of the mask $a$ with respect to a dilation factor $\dm$.
Then we can write
\be \label{sr}
\tilde{\pa}(z)=(1+z+\cdots+z^{\dm-1})^J \tilde{\pb}(z)
\qquad \mbox{for some sequence } b\in \lp{0}.
\ee
It is well known that a mask $a\in \lp{0}$ has order $J$ sum rules as defined in \eqref{sr:2} if and only if \eqref{sr} holds (e.g., see \cite[Theorem 3.5]{han13} or \cite[Theorem~1.2.5]{hanbook}). Recall that the quantity $\rho_J(a,\dm)_p$ is defined in \eqref{sm}. Then we must have
\be \label{smabJ}
\rho_J(a,\dm)_p=\rho_0(b,\dm)_p:=\limsup_{n\to\infty} \|\sd_{b,\dm}^n \td\|_{\lp{p}}^{1/n}
\ee
(e.g., see \cite[Theorem~2.1]{han98} and \cref{lem:eig}) and by \cite[Corollary~5.8.5]{hanbook} and \cite[Corollary~2.2]{han98},
\be \label{sminfty}
\sm_\infty(a,\dm)=-\log_\dm \rho_0(b,\dm)_\infty \quad \mbox{and}\quad
\rho_0(b,\dm)_\infty=\inf_{n\in \N} \sup_{\gamma=0,\ldots, \dm^n-1} \Big[\sum_{k\in \Z} |[\sd_{b,\dm}^n\td](\gamma+\dm^n k)|\Big]^{\frac{1}{n}}.
\ee
In particular, for every $n\in \N$, we obviously have the following lower bounds of $\sm_\infty(a,\dm)$:
\be \label{sm:est:n}
\sm_\infty(a,\dm)=-\log_{\dm} \rho_0(b,\dm)_\infty
\ge -\log_{\dm}
\Big(\sup_{\gamma=0,\ldots, \dm^n-1} \Big[\sum_{k\in \Z} |[\sd_{b,\dm}^n\td](\gamma+\dm^n k)|\Big]^{\frac{1}{n}}\Big).
\ee
If there exists $\gamma_0\in \Z$ such that
\be \label{sm:special}
b(\gamma_0+\dm k)=0 \quad \forall\; k\in \Z\bs\{0\}
\quad \mbox{and}\quad
\sum_{k\in \Z} |b(\gamma+\dm k)|\le |b(\gamma_0)|,\qquad \forall\; \gamma=0,\ldots,\dm-1,
\ee
then $\rho_0(b,\dm)_\infty=\dm |b(\gamma_0)|$ by \cite[Corollary~2.2]{han98}, and hence, $\sm_\infty(a,\dm)=-1-\log_\dm |b(\gamma_0)|$.
Otherwise, we often have to take large integers $n$ in \eqref{sm:est:n} to obtain accurate low bounds of $\sm_\infty(a,\dm)$.

Fortunately, for the special case $p=2$, the quantities $\sm_2(a,\dm)$ and $\rho_0(b,\dm)_2$ can be effectively computed by finding the spectral radius of some special finite matrix $B$.
Because $b$ is finitely supported, we define  $[l_b,h_b]:=\fs(b)$ to be the filter support of $b$.
Define a sequence $c\in \lp{0}$ by $c(j):=\sum_{k=l_b}^{h_b} b(j+k) \overline{b(k)}$ for $j\in \Z$. That is, $\tilde{\pc}(e^{-i\xi})=|\tilde{\pb}(e^{-i\xi})|^2$ for $\xi\in \R$.
Then $\fs(c)=[l_b-h_b,h_b-l_b]$. By \cite[Theorem~2.1]{han98} or \cite[Corollary~5.8.5]{hanbook}, we have $\rho_0(b,\dm)_2=\dm \sqrt{\rho(B)}$ and
\be \label{sm2}
\sm_2(a,\dm)=-\tfrac{1}{2}-\tfrac{1}{2}\log_{\dm} \rho (B) \quad
 \quad{with}\quad
B:=(c(\dm k-j))_{-\lfloor \frac{h_b-l_b}{\dm-1}\rfloor\le j,k\le \lfloor \frac{h_b-l_b}{\dm-1}\rfloor},
\ee
where $\rho(B)$ is the spectral radius of the finite matrix $B$ and $\lfloor \frac{h_b-l_b}{\dm-1}\rfloor$ is the largest integer $\le \frac{h_b-l_b}{\dm-1}$.
Note that the sequence $c$ must be symmetric about the origin. Therefore, taking advantages of symmetry of the sequence $c$, we can further speed up the calculation of $\sm_2(a,\dm)$ by computing the spectral radius of a smaller matrix (roughly speaking, half size of the matrix $B$ in \eqref{sm2}), see \cite[Algorithm~2.1]{han03smaa}.
Moreover, the quantity $\sm_\infty(a,\dm)$ can be estimated from $\sm_2(a,\dm)$ by
\be \label{sm2:sminfty}
\sm_2(a,\dm)-\tfrac{1}{2}\le \sm_\infty(a,\dm)\le \sm_2(a,\dm).
\ee
%
We also refer to \cite[Corollary~5.8.5]{hanbook} for other ways of estimating the smoothness exponent $\sm_\infty(a,\dm)$.

In all our examples constructed through \cref{thm:int,thm:qs}, the masks $a\in \lp{0}$ often have several free parameters. Because we often have to solve nonlinear equations in \cref{thm:int,thm:qs}, in fact we often obtain several families of masks with  free parameters and complicated expressions.
Consequently, to find special values of the parameters such that the smoothness exponent $\sm_\infty(a,\dm)$ is as large as possible, we simply use a brute force method by locally searching for the highest possible smoothness $\sm_2(a,\dm)$ among the parameters until $\sm_2(a,\dm)$ achieves a local maximum value among such parameters. Due to \eqref{sm2:sminfty}, such smoothness exponent $\sm_\infty(a,\dm)$ at the special parameter values is nearly the highest among all values of the parameters. Directly minimizing the spectral radius $\rho(B)$ among the parameters of masks $a$ is difficult, because the masks $a$ obtained by \cref{thm:int,thm:qs} often have complicated structure and many parameters. For relatively simple masks with free parameters, this issue has been addressed in \cite[Section~4]{hoy03} for constructing smooth bivariate Hermite subdivision schemes aided by spectral radius optimization.

\subsection{The condition \eqref{cond:sa} on $s_a$ in \cref{thm:int} for $s_a$-interpolating refinable functions}
\label{subsec:sa}


At first glance, the condition \eqref{cond:sa} in \cref{thm:int} may appear to be artificial and complicated to the readers. But \eqref{cond:sa} is in fact rooted in the fundamental problem of how to determine the exact (not approximated) value $\phi(s_a)$ of a general continuous $\dm$-refinable function $\phi$ (not necessarily interpolating) within finitely many steps using only its mask $a\in \lp{0}$. As we mentioned before, except spline refinable functions, an $\dm$-refinable function $\phi$ with a mask $a\in \lp{0}$ cannot have any analytic expression (e.g., see \cite[Chapter~6.1]{hanbook}) and is only theoretically defined through the Fourier transform by
$\wh{\phi}(\xi):=\prod_{j=1}^\infty \tilde{\pa}(e^{-i \dm^{-j}\xi})$ for $\xi \in \R$, or equivalently, $\phi$ is the unique latent solution to the refinement equation \eqref{refeq} under the normalization condition $\wh{\phi}(0)=1$.
Consequently, to satisfy the condition $\phi(s_a+k)=\td(k)$ for all $k\in \Z$  in \eqref{intphi:sa}, the exact values $\phi(s_a+k)$ for $k\in \Z$ must be able to be determined in finitely many steps from its mask $a\in \lp{0}$ through the refinement equation \eqref{refeq}.

Iterating the refinement equation \eqref{refeq}, one can easily deduce that
\be \label{phi:An}
\phi(x)=\sum_{k\in \Z} [\sd_{a,\dm}^n \td](k)\phi(\dm^n x-k)=\dm^n \sum_{k\in \Z} A_n(k)\phi(\dm^n x-k),\qquad n\in \N, x\in \R,
\ee
where $A_n$ is defined in \eqref{An}, i.e., $A_n:=\dm^{-n} \sd_{a,\dm}^n \td$.
Therefore, an $\dm$-refinable function $\phi$ with mask $a$ is also an $\dm^{n}$-refinable function with the mask $A_{n}$ for every $n\in \N$.

For any given $s_a\in \R$, without assuming that the continuous $\dm$-refinable function $\phi$ is interpolating, now we discuss how to determine the exact value $\phi(s_a)$ in finitely many steps by using the mask $a$.
Let $m_s\in \NN$.
If the exact values $\phi(\dm^{m_s}s_a+k)$ for all $k\in \Z$ are known, then \eqref{phi:An} with $n=m_s$ and $x=s_a$
uniquely determines $\phi(s_a)$ in finitely many steps through the mask $a$ (more precisely, $A_{m_s}$). For simplicity, we define $s:=\dm^{m_s} s_a$ with $m_s\in \NN$ and rewrite the refinement equation \eqref{refeq} as
\be \label{refeq:2}
\phi(\dm^{-1}x)=\dm \sum_{k\in \Z} a(k) \phi(x-k),\qquad x\in \R.
\ee
If the exact values of $\phi$ on $s+\Z$ are known, then the refinement equation \eqref{refeq:2} uniquely determines the values of $\phi$ on $\dm^{-1} s+\dm^{-1} \Z$. Repeating the same argument, we deduce that
the exact values of $\phi$ on
$\dm^{-n}s+\dm^{-n}\Z$
are uniquely determined by \eqref{refeq:2} for all $n\in \Z$.
Hence, we have two cases:

Case 1: $s \in \dm^{-n}s+\dm^{-n}\Z$ with $s:=\dm^{m_s}s_a$ for some $n=n_s\in \N$ and some $m_s\in \NN$. Then we must have $(\dm^{n_s}-1)s\in \Z$. Consequently, we obtain the condition \eqref{cond:sa}:
\[
\dm^{m_s} (\dm^{n_s}-1)s_a= (\dm^{n_s}-1)s\in \Z.
\]
In this case, we have $[s+\Z] \subseteq [\dm^{-n_s} s+\dm^{-n_s}\Z]$. Consequently, the exact values of $\phi$ on $s+\Z$
are determined by the values of $\phi$ on $\dm^{-n_s} s+\dm^{-n_s}\Z$, which are determined in turn by the values of $\phi$ on $s+\Z$ through \eqref{phi:An} with $n=n_s$ and $x\in \dm^{-n_s} s+\dm^{-n_s}\Z$.
Therefore, if the finitely supported mask $a$ is known, then the exact values of $\phi$ on $s+\Z$ can be uniquely determined by finitely many equations plus the normalization condition $\sum_{k\in \Z} \phi(s+k)=1$. Consequently, because $s=\dm^{m_s} s_a$, the values $\phi(s_a+k)$ for all $k\in \Z$ are uniquely determined by \eqref{phi:An} with $n=m_s$ and $x\in s_a+\Z$. To have a necessary condition for $\phi$ to be $s_a$-interpolating,
the above argument and the condition \eqref{cond:sa} eventually lead to the key nonlinear equations \eqref{cond:ms} and \eqref{cond:ns} in \cref{thm:int}.
Note that $s_a\in \R$ satisfies \eqref{cond:sa} if and only if $s_a\in \cup_{m_s=0}^\infty \cup_{n_s=1}^\infty [\dm^{-m_s}(\dm^{n_s}-1)^{-1}\Z]$, which is dense in $\R$.

Next we claim that $s_a \in \R$ satisfies \eqref{cond:sa} if and only if $[0,1)\cap (\cup_{j=0}^\infty [\dm^j s_a+\Z])$ is a finite set. If $s_a$ satisfies \eqref{cond:sa}, then $s_a\in \dm^{-m_s} (\dm^{n_s}-1)^{-1}\Z$ for some $m_s\in \NN$ and $n_s\in \N$. Obviously, $[\dm^j s_a+\Z] \subseteq \dm^{-m_s} (\dm^{n_s}-1)^{-1}\Z$ for all $j\in \NN$. Because $[0,1)\cap \dm^{-m_s} (\dm^{n_s}-1)^{-1}\Z$ is obviously a finite set, we conclude that $[0,1)\cap (\cup_{j=0}^\infty [\dm^j s_a+\Z])$ must be a finite set.
Conversely, suppose that $T:=[0,1)\cap (\cup_{j=0}^\infty [\dm^j s_a+\Z])$ is a finite set. Then for each $j\in \NN$, there must exist unique $k_j\in \Z$ and $t_j\in T$ such that $\dm^j s_a+k_j=t_j\in T$.
Because $T$ is a finite set, there must exist $0\le j<\ell<\infty$ such that $t_j=t_\ell$. Hence $\dm^j s_a+k_j=t_j=t_\ell=\dm^\ell s_a+k_\ell$, from which we have $\dm^j (\dm^{\ell-j}-1)s_a=\dm^\ell s_a-\dm^j s_a=k_j-k_\ell\in \Z$. Therefore, \eqref{cond:sa} is satisfied with $m_s=j\in \NN$ and $n_s=\ell-j \in \N$.
Thanks to the fact that $[0,1)\cap (\cup_{j=1}^\infty [\dm^j s_a+\Z])$ is a finite set, our argument for Case 1 shows that the exact value $\phi(s_a)$ can be obtained in finitely many steps by only using the mask $a$.

Case 2: $s\not \in \dm^{-n}s+\dm^{-n}\Z$ with $s:=\dm^{m_s} s_a$ for all $n\in \N$ and $m_s\in \NN$. Then
\eqref{cond:sa} on $s_a$ fails and the set $[0,1)\cap (\cup_{j=0}^\infty [\dm^j s_a+\Z])$ must be infinite. For this case, we are not aware of any known method for computing the exact value $\phi(s_a)$ of a continuous $\dm$-refinable function $\phi$ within finitely many steps from its mask or the existence of any $s_a$-interpolating refinable function when \eqref{cond:sa} fails.

Let $\phi$ be a continuous $\dm$-refinable function with a mask $a$ such that $\sm_\infty(a,\dm)>0$.
Without assuming that $\phi$ is interpolating,
we shall further discuss how to effectively compute $\phi(s_a)$ in finitely many steps by only using its mask $a$ at the end of \cref{sec:proof} for any $s_a\in \R$ satisfying \eqref{cond:sa}.
We shall also explain the rule of $s_a\in \R$ from the perspective of subdivision curves in CAGD in Subsection~\ref{subsec:app}.

\subsection{Examples of $r$-step interpolatory $r$-mask quasi-stationary dyadic subdivision schemes with symmetry}

The dilation factor $\dm=2$ is the most widely studied case in the literature.
Though it is highly desired to have $\mathscr{C}^2$-convergent dyadic subdivision schemes with masks having two-ring stencils, as discussed in \cref{sec:intro},
there are no standard interpolating $2I_d$-refinable functions $\phi\in \mathscr{C}^2(\R^d)$ and no $\mathscr{C}^2$-convergent interpolatory $2I_d$-subdivision schemes with masks having two-ring stencils (\cite[Corollary~4.3]{han99}).
Applying \cref{thm:int,thm:qs}, we now present examples to show that this shortcoming can be remedied by using $r$-step interpolatory $r$-mask quasi-stationary $2$-subdivision schemes with $r\in \{2,3\}$ and all symmetric masks $\{a_1,\ldots, a_r\}$ having at most two-ring stencils.

\begin{example}\label{ex1}
\rm Let $\dm=2$ and $r=2$.
Let $a_1, a_2\in \lp{0}$ be symmetric masks supported inside $[-2,2]$  with $c_a=0$ in \eqref{symmask} and $\sr(a_1,\dm)=\sr(a_2,\dm)=2$ as follows:
\begin{align*}
&\widetilde{\pa_1}(z)=\tfrac{1}{4}z^{-1}(1+z)^2(
t_1 z^{-1}+1-2t_1+t_1z),\\
&\widetilde{\pa_2}(z)=\tfrac{1}{4}z^{-1}(1+z)^2(t_2 z^{-1}+1-2t_2+t_2z),
\end{align*}
with $t_1,t_2\in \R$. Note that both masks $a_1$ and $a_2$ have only one-ring stencils: the even stencil $\{2a_\ell(-2), 2a_\ell(0), 2a_\ell(2)\}$, and the odd stencil $\{2a_\ell(-1), 2a_\ell(1)\}$ for $\ell=1,2$.
Define a new mask $a\in \lp{0}$ by $\tilde{\pa}(z):=\widetilde{\pa_1}(z^2)\widetilde{\pa_2}(z)$.
Solving the interpolation condition $a(4k)=0$ for all $k\in \Z\bs\{0\}$ (i.e., $a(-4)=a(4)=0$), we obtain $t_1=\frac{t_2}{2(t_2-1)}$ with $t_2\in \R\bs \{1\}$.
Optimizing the smoothness quantity $\sm_2(a,4)$ as described in Subsection~\ref{subsec:sm} among choices of the parameter $t_2$, we have $t_2=\frac{11}{32}$ (and hence $t_1=-\frac{11}{42}$) with $\sm_2(a,4)\approx 1.709055$. Explicitly,
\be\label{a1a2:r1}
a_1=\{-\tfrac{11}{168}, \tfrac{1}{4}, \tfrac{53}{84}, \tfrac{1}{4}, -\tfrac{11}{168}\}_{[-2,2]},\qquad
a_2=\{\tfrac{11}{128}, \tfrac{1}{4}, \tfrac{21}{64}, \tfrac{1}{4}, \tfrac{11}{128}\}_{[-2,2]}.
\ee
Note that the mask $a$ is supported inside $[-6,6]$ and its $4$-refinable function $\phi$ is supported inside $[-2,2]$. Because $\sm_\infty(a,4)\ge 1.512277$ using \eqref{sminfty} with $n=3$,
by \cref{thm:int,thm:qs} (also see \cref{cor:qs}), the $4$-refinable function $\phi\in \CH{1}$ is $0$-interpolating with $\phi(k)=\td(k)$ for all $k\in \Z$. Moreover, the  $2$-step interpolatory $2$-mask quasi-stationary $2$-subdivision scheme with masks $\{a_1,a_2\}$ is $\mathscr{C}^1$-convergent.
See \cref{fig:ex1} for the graph of the $0$-interpolating $4$-refinable function $\phi\in \CH{1}$ for the masks $\{a_1, a_2\}$ in \eqref{a1a2:r1}. Moreover, $\sm_2(a_1,2)\approx 0.860944$ and $\sm_2(a_2,2)\approx 1.989281$.
\end{example}

\begin{example}\label{ex2}
\rm Let $\dm=2$ and $r=2$.
Let $a_1, a_2\in \lp{0}$ be symmetric masks supported inside $[-4,4]$
with $c_a=0$ in \eqref{symmask}  and $\sr(a_1,\dm)=\sr(a_2,\dm)=4$ as follows:
\be \label{a1a2}
\begin{split}
&\widetilde{\pa_1}(z)=\tfrac{1}{16}z^{-2}(1+z)^4(t_2z^{-2}+t_1 z^{-1}+1-2t_1-2t_2+t_1z+t_2 z^2),\\
&\widetilde{\pa_2}(z)=\tfrac{1}{16}z^{-2}(1+z)^4(t_4z^{-2}+t_3 z^{-1}+1-2t_3-2t_4+t_3z+t_4 z^2),
\end{split}
\ee
with $t_1,\ldots,t_4\in \R$. Note that both masks $a_1$ and $a_2$ have two-ring stencils: the even stencil $\{2a_\ell(-4), 2a_\ell(-2), 2a_\ell(0), 2a_\ell(2), 2a_\ell(4)\}$, and the odd stencil $\{2a_\ell(-3), 2a_\ell(-1), 2a_\ell(1), 2a_\ell(3)\}$ for $\ell=1,2$.
Define a new mask $a\in \lp{0}$ by $\tilde{\pa}(z):=\widetilde{\pa_1}(z^2)\widetilde{\pa_2}(z)$.
Solving the interpolation condition $a(4k)=0$ for all $k\in \Z\bs\{0\}$ (i.e., $a(4)=a(8)=a(12)=0$ by using symmetry), we obtain four solution families of masks $a$ below:
\begin{align*}
&\{t_1=-\tfrac{1}{4}-3 t_2-s_2, t_3=-\tfrac{3}{2}-4t_2+4s_2, t_4=0\},\quad
\{t_1=-\tfrac{1}{4}-3 t_2+s_2, t_3=-\tfrac{3}{2}-4t_2-4s_2, t_4=0\},\\
&\{ t_1=-\tfrac{1}{4}-\tfrac{t_4}{8}+s_4,
t_2=0, t_3=-\tfrac{3}{2}-\tfrac{7}{2} t_4-4s_4\},\quad
\{t_1=-\tfrac{1}{4}-\tfrac{t_4}{8}-s_4, t_2=0, t_3=-\tfrac{3}{2}-\tfrac{7}{2} t_4+4s_4\},
\end{align*}
where $s_2:=\frac{1}{4} \sqrt{16 t_2^2+24 t_2+1}$ and $s_4:=\frac{1}{8}
\sqrt{t_4^2+12 t_4+4}$. Note that $t_2$ is a free parameter in the first two solutions while $t_4$ is a free parameter in the third and fourth solutions.
Optimizing the smoothness quantity
$\sm_2(a,4)$ as described in Subsection~\ref{subsec:sm}
in the third and fourth solution families with
the parameter $t_4$,
we find $t_4=-\frac{9}{32}$ in the fourth solution leading to
\be \label{tval1}
t_1=-\tfrac{\sqrt{721}+55}{256},\quad t_2=0,\quad t_3=\tfrac{\sqrt{721}-33}{64},\quad t_4=-\tfrac{9}{32}\quad \mbox{with}\quad
\sm_2(a,4)\approx 2.62522.
\ee
Note that the mask $a$ is supported inside $[-10,10]$ and its $4$-refinable function $\phi$ is supported inside $[-\frac{10}{3},\frac{10}{3}]$. Because $\sm_\infty(a,4)\ge \sm_2(a,4)-0.5\approx 2.12522>2$,
by \cref{thm:int,thm:qs} (also see \cref{cor:qs}),
the $4$-refinable function $\phi\in \CH{2}$ is $0$-interpolating with $\phi(k)=\td(k)$ for all $k\in \Z$. The  $2$-step interpolatory $2$-mask quasi-stationary $2$-subdivision scheme with masks $\{a_1,a_2\}$ is $\mathscr{C}^2$-convergent. Moreover, $\sm_2(a_1,2)\approx 2.747783$ and $\sm_2(a_2,2)\approx 2.623172$.

Optimizing the smoothness quantity
$\sm_2(a,4)$ as described in Subsection~\ref{subsec:sm}
in the first and second solution families with
the free parameter $t_2$,
we find $t_2=\frac{5}{16}$ in the first solution leading to
\be \label{tval2}
t_1=-\tfrac{\sqrt{161}+19}{16},\quad
t_2=\tfrac{5}{16},\quad t_3=\tfrac{\sqrt{161}-11}{4},\quad t_4=0\quad \mbox{with}\quad
\quad \sm_2(a,4)\approx 3.073353.
\ee
Note that the mask $a$ is supported inside $[-11,11]$ and its $4$-refinable function $\phi$ is supported inside $[-\frac{11}{3},\frac{11}{3}]$.
Because $\sm_\infty(a,4)\ge \sm_2(a,4)-0.5\approx2.573353$ (or $\sm_\infty(a,4)\ge 2.806997$ using \eqref{sminfty} with $n=3$),
by \cref{thm:int,thm:qs},
the $4$-refinable function $\phi\in \CH{2}$ is $0$-interpolating with $\phi(k)=\td(k)$ for all $k\in \Z$. The $2$-step interpolatory $2$-mask quasi-stationary $2$-subdivision scheme with masks $\{a_1,a_2\}$ is $\mathscr{C}^2$-convergent.
Moreover, $\sm_2(a_1,2)\approx 1.3074664$ and $\sm_2(a_2)\approx 3.991650$.
See \cref{fig:ex1} for the graph of the interpolating $4$-refinable function $\phi\in \CH{2}$ with the parameters $t_1,\ldots, t_4$ in \eqref{tval2}.
For both cases with the dyadic dilation factor $\dm=2$, the $2$-step interpolatory $2$-mask quasi-stationary $2$-subdivision schemes with masks $\{a_1,a_2\}$ have the $2$-step interpolation property:
\[
[(\sd_{a_2,\dm}\sd_{a_1,\dm})^n v](\dm^{2n}k)=
[\sd_{a_1,a_2, \dm}^{2n, 2} v](\dm^{2n}k)
=[\sd_{a,\dm^2}^n v](\dm^{2n}k)=
v(k),\qquad \forall\; k\in \Z, n\in \N,v\in \lp{}.
\]
\end{example}

\begin{example}\label{ex3}
\rm Let $\dm=2$ and $r=3$.
Let $a_1, a_2, a_3\in \lp{0}$ be symmetric masks supported inside $[-4,4]$ with $c_a=0$
in \eqref{symmask} and $\sr(a_1,\dm)=\sr(a_2,\dm)=\sr(a_3,\dm)=4$ such that masks $a_1, a_2$ are given in \eqref{a1a2} and the mask $a_3$ is parameterized as follows:
\[
\widetilde{\pa_3}(z)=\tfrac{1}{16}z^{-2}(1+z)^4(t_5 z^{-1}+1-2t_5+t_5z)
\]
with $t_1,\ldots,t_5\in \R$.
Define a mask $a\in \lp{0}$ by $\tilde{\pa}(z):=\widetilde{\pa_1}(z^4)\widetilde{\pa_2}(z^2)\widetilde{\pa_3}(z)$.
Solving the interpolation condition $a(8k)=0$ for all $k\in \Z\bs\{0\}$ (i.e., $a(8)=a(16)=a(24)=0$ by using symmetry), we obtain nine solution families of masks $a$.
Optimizing the smoothness quantity $\sm_2(a,8)$
for each solution family of masks among their free parameters, we have
\[
t_1=-\tfrac{\sqrt{713}+41}{32},\quad t_2=\tfrac{11}{32},\quad t_3=\tfrac{179\sqrt{713}}{616}-\tfrac{140873}{19712},\quad t_4=\tfrac{40137}{39424}-\tfrac{51\sqrt{713}}{1232},\quad t_5=\tfrac{19}{64}
\]
with $\sm_2(a,8)\approx 3.4519942$. That is, we set $t_2=\frac{11}{32}$ and $t_5=\frac{19}{64}$ in the particular solution family
\[
\{t_1=-\tfrac{1}{4}-3t_2-\tfrac{1}{4}s_2,\quad t_3=\tfrac{2t_5+5}{8(2t_5-3)}(13+32 t_2+2t_5-8s_2), \quad t_4=-\tfrac{2t_5+1}{16(2t_5-3)} (13+32 t_2+2t_5-8s_2)\}
\]
with $s_2:=\sqrt{16 t_2^2+24 t_2+1}$.
Note that the mask $a$ is supported inside $[-27,27]$ and its $8$-refinable function $\phi$ is supported inside $[-\frac{27}{7},\frac{27}{7}]$. Because $\sm_\infty(a,8)\ge 3.216038$ using \eqref{sminfty} with $n=4$,
by \cref{thm:int,thm:qs},
the $8$-refinable function $\phi\in \CH{3}$ is $0$-interpolating with $\phi(k)=\td(k)$ for all $k\in \Z$. The  $3$-step interpolatory $3$-mask quasi-stationary $2$-subdivision scheme with masks $\{a_1,a_2, a_3\}$ is $\mathscr{C}^3$-convergent. Moreover, $\sm_2(a_1,2)\approx 1.239518$, $\sm_2(a_2)\approx 3.955358$, and $\sm_2(a_3,2)\approx 3.995045$.
See \cref{fig:ex1} for the graph of the interpolating $8$-refinable function $\phi\in \CH{3}$. Finally, we point out that solution families of masks $a$ with simple expressions may not lead to large $\sm_2(a,\dm)$, for example, the solution family
$\{ t_1=-\tfrac{9}{4}, t_2=\tfrac{3}{8}, t_3=-4t_4, t_5=\tfrac{3}{2}\}$
can achieve the almost highest $\sm_2(a,\dm)\approx 2.405870$ at $t_4=-\frac{3}{32}$ and $\sm_\infty(a,\dm)\ge 2.160594$ using \eqref{sminfty} with $n=2$.
\end{example}

\begin{figure}[htbp]
	\centering
\begin{subfigure}[b]{0.30\textwidth} \includegraphics[width=\textwidth,height=0.4\textwidth]{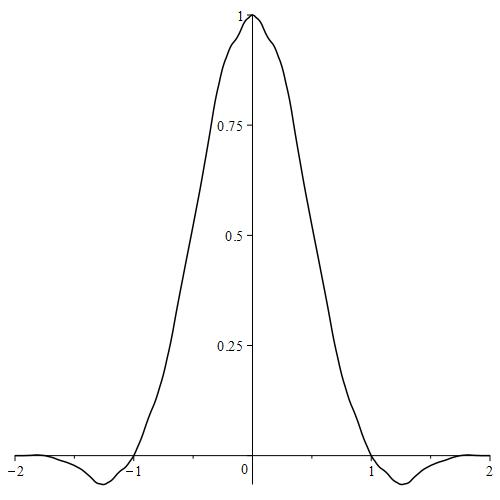} \caption{$\phi\in \mathscr{C}^1(\R)$}
	\end{subfigure}
	 \begin{subfigure}[b]{0.30\textwidth} \includegraphics[width=\textwidth,height=0.4\textwidth]{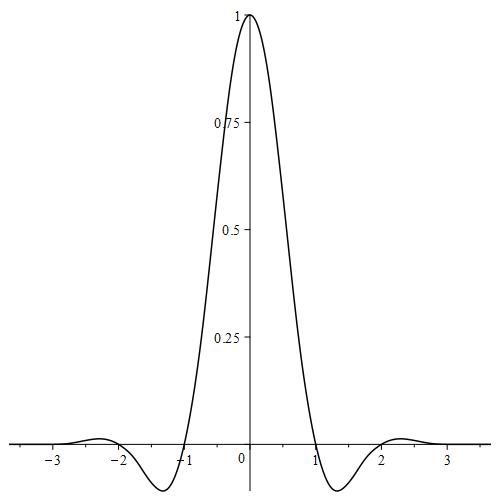} \caption{$\phi\in \mathscr{C}^2(\R)$}
	\end{subfigure}
	 \begin{subfigure}[b]{0.30\textwidth} \includegraphics[width=\textwidth,height=0.4\textwidth]{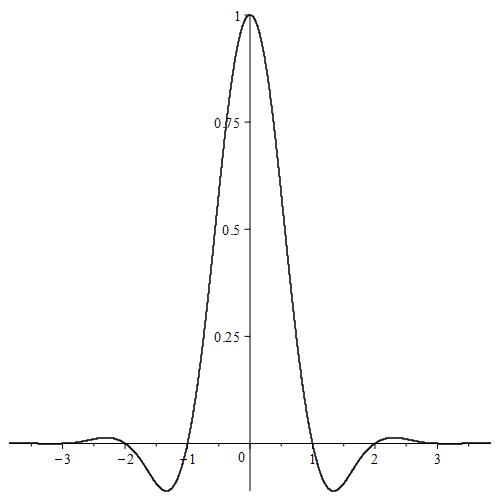} \caption{$\phi\in \mathscr{C}^3(\R)$}
	\end{subfigure}\\ \vskip0.1in
\begin{subfigure}[b]{0.30\textwidth} \includegraphics[width=\textwidth,height=0.4\textwidth]{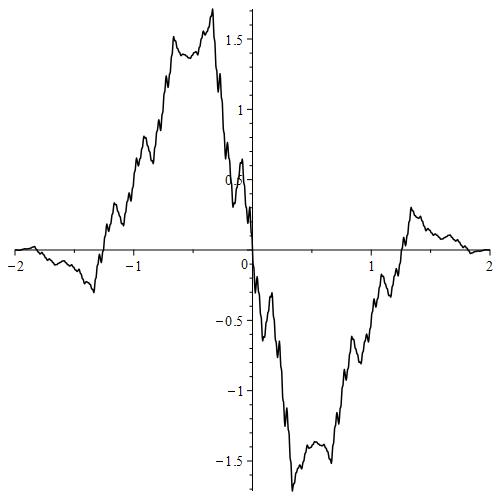} \caption{$\phi'$}
	\end{subfigure}
	 \begin{subfigure}[b]{0.30\textwidth} \includegraphics[width=\textwidth,height=0.4\textwidth]{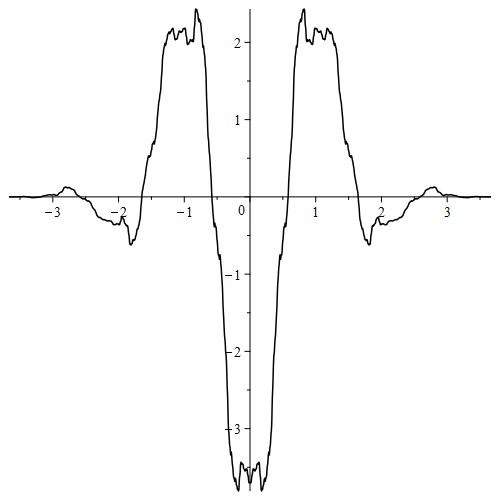} \caption{$\phi''$}
	\end{subfigure}
	 \begin{subfigure}[b]{0.30\textwidth} \includegraphics[width=\textwidth,height=0.4\textwidth]{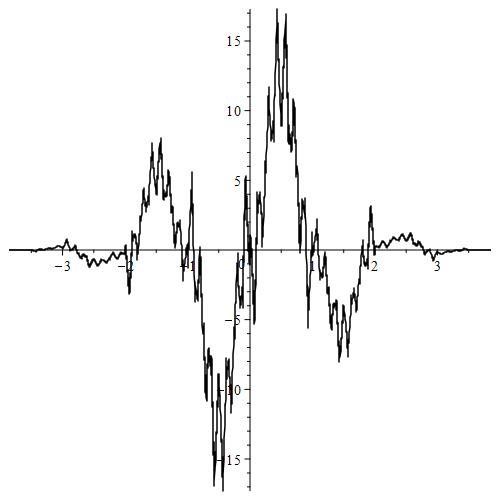} \caption{$\phi'''$}
	\end{subfigure}
\caption{
(a) is the graph of the interpolating $4$-refinable function $\phi\in \CH{1}$ in \cref{ex1} and (d) is its first-order derivative $\phi'$.
(b) is the graph of the interpolating $4$-refinable function $\phi\in \CH{2}$ in \cref{ex2} with parameters in \eqref{tval2} and (e) is
its second-order derivative $\phi''$.
(c) is the graph of the interpolating $8$-refinable function $\phi\in \CH{3}$ in \cref{ex3} and (f) is its third-order derivative $\phi'''$.
}\label{fig:ex1}
\end{figure}

\subsection{Construction procedure of all desired masks in \cref{thm:int}}\label{subsec:construct}

In order to provide some examples using \cref{thm:int}, we now discuss how to construct $s_a$-interpolating refinable functions and their $n_s$-step interpolatory subdivision schemes.

Except the special case $m_s=0$ and $n_s=1$ for standard interpolatory subdivision schemes, the conditions in \eqref{cond:ms}, \eqref{cond:ns} and \eqref{cond:ms=0} of \cref{thm:int} involve nonlinear equations, which are computationally challenging. Therefore, it is helpful to obtain further necessary conditions to facilitate the construction through \cref{thm:int}. We shall take advantages of linear-phase moments in \cite{han03,han10,han13} to facilitate the construction of $s_a$-interpolating $\dm$-refinable functions.
For convenience, throughout the paper we shall adopt the following big $\bo$ notion: For $J\in \NN$ and smooth functions $f$ and $g$,
\[
f(\xi)=g(\xi)+\bo(|\xi|^J),\quad \xi\to 0
\quad \mbox{ stands for }\quad f^{(j)}(0)=g^{(j)}(0),\quad \forall\; j=0,\ldots, J-1.
\]
Now we have the following result about necessary conditions on $s_a$-interpolating $\dm$-refinable functions.

\begin{prop}\label{prop:lpm}
Let $\dm\in \N\bs\{1\}$ be a dilation factor.
Let $\phi$ be a compactly supported $s_a$-interpolating $\dm$-refinable function normalized by $\wh{\phi}(0)=1$ with a finitely supported mask $a\in \lp{0}$. Define $J:=\sr(a,\dm)$ (i.e., the mask $a$ has order $J$ sum rules with respect to $\dm$ as in \eqref{sr:2} or \eqref{sr}). Then
\begin{align}
&\sum_{k\in \Z} k^j \phi(x+k)=(s_a-x)^j,\quad \forall\; j=0,\ldots,J-1 \mbox{ and } x\in \R,\label{phi:lpm}\\
&\sum_{k\in \Z} k^j a(k)=m_a^j,\qquad \forall\, j=0,\ldots, J-1 \quad \mbox{with}\quad m_a:=(\dm-1)s_a,
\label{a:lpm}
\end{align}
and both the function $\phi$ and mask $a$ must have order $J$ linear-phase moments as follows:
\be \label{lpm}
\wh{\phi}(\xi)=e^{-i s_a\xi}+\bo(|\xi|^J)
\quad \mbox{and}\quad \tilde{\pa}(e^{-i\xi})=e^{-im_a \xi}+\bo(|\xi|^J),\qquad \xi\to 0.
\ee
Consequently, if $\sr(a,\dm)\ge 2$, then the real numbers $s_a$ and $m_a$ must be given by
\be \label{phi:sa}
s_a=i(\wh{\phi})'(0)=\int_{\R} x\phi(x) dx=\frac{m_a}{\dm-1} \quad \mbox{with}\quad m_a:=(\dm-1)s_a=\sum_{k\in \Z} k a(k).
\ee
\end{prop}

\bp To prove the claims, the key ingredient of the proof is to show that $\sr(a,\dm)=J$ implies
\be \label{phi:poly}
\sum_{k\in \Z} \pp(k)\phi(x-k)=
\sum_{j=0}^{J-1} \frac{(-i)^j}{j!} \pp^{(j)}(x) \wh{\phi}^{(j)}(0),\qquad \forall\; \pp\in \Pi_{J-1}.
\ee
For convenience of discussion, we define
$\wh{a}(\xi):=\tilde{\pa}(e^{-i\xi})=\sum_{k\in \Z} a(k) e^{-ik\xi}$.
Then $\wh{\phi}(\dm\xi) = \wh{a}(\xi)\wh{\phi}(\xi)$. Note that $\sr(a,\dm)\ge J$  if and only if (e.g., see \cite[Theorem 3.5]{han13} or \cite[Theorem~1.2.5]{hanbook})
\be \label{sr:1}
\wh{a}(\xi+2\pi \dm^{-1}\gamma)=\bo(|\xi|^J),\qquad \xi\to 0 \quad \mbox{for all } \gamma\in \Z\bs [\dm \Z].
\ee
For $k\in \Z\bs\{0\}$, we can uniquely write $k=\dm^{n} \gamma$ with $n\in \NN$ and $\gamma\in \Z\bs [\dm \Z]$. Recursively applying $\wh{\phi}(\xi)=\wh{a}(\dm^{-1}\xi)\wh{\phi}(\dm^{-1}\xi)$ and noting that $\wh{a}$ is $2\pi$-periodic, we derive from \eqref{sr:1} that
\[
\wh{\phi}(\xi+2\pi k)=\wh{\phi}(\xi+2\pi \dm^{n} \gamma)=
\left[\prod_{j=1}^{n} \wh{a}(\dm^{-j}\xi)\right] \wh{a}(\dm^{-n-1}\xi+2\pi \dm^{-1} \gamma) \wh{\phi}(\dm^{-n-1}\xi+2\pi \dm^{-1}\gamma)
=\bo(|\xi|^J),
\]
as $\xi \to 0$. Hence, we proved
\be \label{phi:moment}
\wh{\phi}(0)=1 \quad \mbox{and}\quad \wh{\phi}(\xi+2\pi k)=\bo(|\xi|^J),\qquad \xi \to 0\; \mbox{ for all } k\in \Z\bs\{0\}.
\ee
For the $1$-periodic function $f(x):=\sum_{k\in \Z} (x-k)^j \phi(x-k)$, we observe that its Fourier coefficient $\int_0^1 f(x) e^{-i 2\pi kx}dx=
\int_{\R} x^j \phi(x) e^{-i2\pi kx} dx=
i^j \wh{\phi}^{(j)}(2\pi k)$ for $k\in \Z$. Hence, using the Fourier series of the $1$-periodic function $f$,
we easily deduce from \eqref{phi:moment} that
\be \label{phi:j:poly}
\sum_{k\in \Z} (x-k)^j \phi(x-k)=i^j \wh{\phi}^{(j)}(0),\qquad j=0,\ldots, J-1.
\ee
Using the Taylor expansion of $\pp(k)=\pp(x-(x-k))=\sum_{j=0}^\infty \frac{(-1)^j}{j!} \pp^{(j)}(x) (x-k)^j$ at the base point $x$, we conclude from \eqref{phi:j:poly}  (also see \cite[Theorem~5.5.1]{hanbook}) that \eqref{phi:poly} holds by noting
\[
\sum_{k\in \Z} \pp(k)\phi(x-k)=
\sum_{j=0}^\infty \frac{(-1)^j}{j!} \pp^{(j)}(x) \sum_{k\in \Z} (x-k)^j\phi(x-k)
=\sum_{j=0}^{J-1} \frac{(-i)^j}{j!} \pp^{(j)}(x) \wh{\phi}^{(j)}(0),\qquad
\pp\in \Pi_{J-1}.
\]
Because $\phi(s_a+n)=\td(n)$ for all $n\in \Z$, plugging $x=s_a+n$ into \eqref{phi:poly} and using the Taylor expansion of $\pp$ at the base point $s_a+n$, we observe
\[
\sum_{j=0}^{J-1} \frac{(-i)^j}{j!} \pp^{(j)}(s_a+n) \wh{\phi}^{(j)}(0)
=\pp(n)=\sum_{j=0}^{\infty} \frac{1}{j!} \pp^{(j)}(s_a+n) (-s_a)^j,\qquad \forall\; \pp\in \Pi_{J-1}, n\in \Z.
\]
For $m=0,\ldots,J-1$, we deduce from the above identity using $\pp(x)=(x-s_a-n)^m$ that $\wh{\phi}^{(m)}(0)=(-is_a)^m$. This proves
$\wh{\phi}(\xi)=e^{-is_a\xi}+\bo(|\xi|^J)$ as $\xi\to 0$, i.e., the first identity in \eqref{lpm} holds.
Using the refinement equation $\wh{\phi}(\dm \xi)=\wh{a}(\xi)\wh{\phi}(\xi)$, we have
$e^{-i\dm s_a\xi}=\wh{a}(\xi)e^{-is_a\xi}+\bo(|\xi|^J)$ as $\xi\to 0$, from which we have the second identity in \eqref{lpm} and consequently, \eqref{a:lpm} holds. Using \eqref{lpm} and \eqref{phi:poly} with $\pp(x)=x^j$, we have \eqref{phi:lpm}.
If $\sr(a,\dm)\ge 2$ (i.e., $J\ge 2$), then we obtain from \eqref{lpm} that $\wh{\phi}'(0)=-is_a$ and $\wh{a}'(0)=-im_a$, i.e., $s_a=i\wh{\phi}'(0)$ and $m_a=i \wh{a}'(0)=\sum_{k\in \Z} k a(k)$. This proves \eqref{phi:sa}.
\ep

Under the condition $\sr(a,\dm)\ge 2$,
from \eqref{phi:sa} of \cref{prop:lpm}, we must have $s_a=\frac{m_a}{\dm-1}$ with $m_a=\sum_{k\in \Z} k a(k)$, that is, the real number $s_a$ is uniquely determined by the mask $a$ of an $s_a$-interpolating $\dm$-refinable function $\phi$.
Note that if a mask $a\in \lp{0}$ has symmetry in \eqref{symmask} and $\sum_{k\in \Z} a(k)=1$, then \eqref{phi:sa} of \cref{prop:lpm} tells us
\[
m_a=\sum_{k\in \Z} ka(k)=\sum_{k\in \Z} k a(c_a-k)=\sum_{k\in \Z} (c_a-k)a(k)=c_a-m_a,
\]
from which we must have $m_a=c_a/2$, the symmetry center of the symmetric mask $a$. Hence, for symmetric masks $a$ satisfying the symmetry property in \eqref{symmask}, it follows from \cref{prop:lpm} that
\be \label{ca:sa}
s_a=\frac{c_a}{2(\dm-1)}.
\ee
Moreover, we deduce from the refinement equation that its $\dm$-refinable function $\phi$ must be supported inside $\frac{1}{\dm-1}\fs(a)$ and
have the symmetry $\phi(2s_a-\cdot)=\phi$. Consequently, the interpolating refinable functions in convergent interpolatory dual subdivision schemes considered in \cite{grv22,rv20,rom19,v23} are $s_a$-interpolating $\dm$-refinable functions with the particular choice $s_a=\frac{c_a}{2(\dm-1)}$ for an odd integer $c_a$.

Let $s_a\in \R$ satisfy \eqref{cond:sa} with $m_s\in \NN$ and $n_s\in \N$, i.e.,
$s_a=\dm^{-m_s} (\dm^{n_s}-1)^{-1}k$ for some integer $k$. Note that $\sr(a,\dm)\ge \sm_\infty(a,\dm)$ by \eqref{srsm}.
We now discuss and outline how to construct all desired masks $a\in \lp{0}$ in \cref{thm:int} aided by \cref{prop:lpm} for $s_a$-interpolating $\dm$-refinable functions.

\noindent \textbf{Construction Procedure:}
{\it
Let $m\in \NN$ and a positive integer $J>m$. Take $s_a\in \R$ satisfying \eqref{cond:sa} and select $l_a, h_a\in \Z$ with $h_a\ge l_a+(\dm-1)J$.
Then all possible desired masks $a\in \lp{0}$ in \cref{thm:int} satisfying $\fs(a)\subseteq [l_a,h_a]$ and $\sr(a,\dm)\ge J$ are given by the following procedure:
\begin{enumerate}
\item[(S1)] Parameterize masks $a$ by $\tilde{\pa}(z)=(1+z+\cdots+z^{\dm-1})^J \tilde{\pb}(z)$ with unknown $b=\{b(l_b),\ldots, b(h_b)\}_{[l_b,h_b]}$, where $l_b:=l_a$ and $h_b:=h_a-(\dm-1)J$.
    If the mask $a$ is required to have symmetry in \eqref{symmask} (i.e., $a(c_a-k)=a(k)$ for all $k\in \Z$ and this is only possible for $s_a$ in \eqref{ca:sa} with $c_a=l_a+h_a$), then we further require $b(k)=b(h_b+l_b-k)$ for all $k=l_b,\ldots, h_b$.

\item[(S2)] Solve the linear equations \eqref{a:lpm}, i.e., more precisely,
\[
\sum_{k=l_a}^{h_a} k^j a(k)=(m_a)^j,\qquad  \mbox{for all } j=0,\ldots, J-1
\]
with $m_a:=(\dm-1)s_a$ for the unknowns $b(l_b),\ldots, b(h_b)$.

\item[(S3)] Case 1: $m_s=0$. Then we solve the nonlinear equations \eqref{cond:ms=0}, i.e., more explicitly,
$A_{n_s}((\dm^{n_s}-1)s_a+\dm^{n_s}k)=\td(k)$
for all $k=\lceil \frac{(1-\dm^{-n_s})(l_a-(\dm-1)s_a)}{\dm-1}\rceil,\ldots,
\lfloor \frac{(1-\dm^{-n_s})(h_a-(\dm-1)s_a)}{\dm-1}\rfloor$,
for the remaining free parameters among $b(l_b),\ldots, b(h_b)$ after (S2).
\item[(S3')] Case 2: $m_s>0$. Then we parameterize a sequence $w\in \lp{0}$ with filter support $[l_w,h_w]:=\Z\cap (\frac{l_a}{\dm-1}-\dm^{m_s} s_a, \frac{h_a}{\dm-1}-\dm^{m_s} s_a)$. First, we solve the linear equations
\be \label{w}
\sum_{k=l_w}^{h_w} k^j w(k)=(s_a-\dm^{m_s} s_a)^j,\qquad j=0,\ldots,J-1
\ee
for the unknowns $w(l_w),\ldots, w(h_w)$.
Then we solve the nonlinear equations \eqref{cond:ms} and \eqref{cond:ns} of \cref{thm:int} for the remaining unknowns after (S2).

\item[(S4)]  Compute and optimize $\sm_2(a,\dm)$ as described in Subsection~\ref{subsec:sm} for selecting special parameter values among all remaining free parameters such that $\sm_2(a,\dm)$ is as large as possible.
\end{enumerate}
If $\sm_2(a,\dm)>m+\frac{1}{2}$ for the selected values of parameters in (S4), then $\sm_\infty(a,\dm)>m$ and item (2) of \cref{thm:int} is satisfied. Hence, all the claims in items (1)--(3) and \eqref{poly:int} of \cref{thm:int} hold.
} 

As we shall see in the proof of \cref{thm:int} in \cref{sec:proof}, the sequence $w$ in \eqref{w} must be given in \eqref{w:phi}, that is, $w(k)=\phi(\dm^{m_s} s_a+k)$ for all $k\in \Z$. Hence, the linear equations in \eqref{w} is equivalent to those in \eqref{phi:lpm}.
Note that the function $\phi(\dm^{m_s}s_a+\cdot)$ is supported inside $[\frac{l_a}{\dm-1}-\dm^{m_s} s_a, \frac{h_a}{\dm-1}-\dm^{m_s}s_a]$. Because we are constructing an $s_a$-interpolating $\dm$-refinable function $\phi$ through \cref{thm:int}, the function $\phi$ is required to be continuous and hence, it is necessary that $\phi(\frac{l_a}{\dm-1}-\dm^{m_s} s_a)=0$ and $\phi(\frac{h_a}{\dm-1}-\dm^{m_s}s_a)=0$. Consequently, the sequence $w$ must be supported inside $[l_w,h_w]:=\Z\cap (\frac{l_a}{\dm-1}-\dm^{m_s} s_a, \frac{h_a}{\dm-1}-\dm^{m_s} s_a)$.
Therefore, all the desired masks in \cref{thm:int} for $s_a$-interpolating $\dm$-refinable functions can indeed be constructed by the above Construction Procedure.


Let $\phi$ be the $\dm$-refinable function with mask $a\in \lp{0}$. For any $\gamma\in \Z$, the function $\phi(\cdot+\frac{\gamma}{\dm-1})$ is the $\dm$-refinable function with mask $a(\cdot+\gamma)$, while $\phi(-\cdot)$ is the $\dm$-refinable function with mask $a(-\cdot)$. Therefore, it is sufficient for us to consider $s_a\in [0,\frac{1}{2(\dm-1)}]$ for $s_a$-interpolating $\dm$-refinable functions.

\subsection{Special case $m_s=0$ for $s_a$-interpolating refinable functions and $n_s$-step interpolatory subdivision schemes}
\label{subsec:ms0}

We are interested in the special cases of $s_a$ satisfying \eqref{cond:sa} with $m_s=0$, (i.e., $s_a=\frac{k}{\dm^{n_s}-1}$ for some $k\in \Z$ and $n_s \in \N$), due to their special properties and structures.

For $m_s=0$, it is crucial to observe that the equations in \eqref{cond:ms} simply become $w=\td$ due to $A_{m_s}=A_0=\td$ and hence \eqref{cond:ns} is reduced to \eqref{cond:ms=0}.
Because $(\dm^{n_s}-1)s_a\in \Z$ by \eqref{cond:sa} with $m_s=0$, we can define a shifted mask $A(k):=A_{n_s}((\dm^{n_s}-1)s_a+k)$ for $k\in \Z$ and define a function $\Phi:=\phi(s_a+\cdot)$. Then $\Phi$ is obviously a $0$-interpolating (i.e., standard interpolating) $\dm^{n_s}$-refinable function with an interpolatory mask $A$ with respect to the dilation factor $\dm^{n_s}$ satisfying
\[
\Phi=\dm^{n_s}\sum_{k\in \Z} A(k) \Phi(\dm^{n_s}\cdot-k)
\quad \mbox{and}\quad
\Phi(k)=\td(k),\quad
A(\dm^{n_s} k)=\dm^{-n_s}\td(k),\qquad \forall\; k\in \Z.
\]
Hence, $\Phi$ is just a standard interpolating $\dm^{n_s}$-refinable function and its mask $A$ is a standard interpolatory mask with respect to $\dm^{n_s}$. Thus, the $\dm$-refinable function $\phi$ is just a shifted version (precisely, $\phi=\Phi(\cdot-s_a)$) of the standard interpolating $\dm^{n_s}$-refinable function $\Phi$.
In particular, we have $A=a$ and $s_a\in (\dm-1)^{-1}\Z$ for standard interpolatory $\dm$-subdivision schemes if $m_s=0, n_s=1$.

For symmetric masks $a$ in \eqref{symmask}, we must have \eqref{ca:sa}, i.e., $s_a=\frac{c_a}{2(\dm-1)}$, where $c_a/2$ is the symmetry center of the mask $a$.
Because $s_a=\frac{c_a}{2(\dm-1)}$, we have the following two cases for $m_s=0$ in \eqref{cond:sa}:

Case 1: $c_a$ is an even integer. Then
$s_a=\frac{c_a}{2(\dm-1)}=
\frac{c_a/2}{\dm-1}$ satisfies the condition \eqref{cond:sa} with $m_s=0$ and $n_s=1$, due to $c_a/2\in \Z$.
Hence, $\phi(s_a+\cdot)$ is a standard interpolating $\dm$-refinable function with the standard interpolatory mask $a(\frac{c_a}{2}+\cdot)$.
That is, the $s_a$-interpolating $\dm$-refinable function $\phi$ is just an integer-shifted version of a standard interpolating $\dm$-refinable function
and its subdivision scheme is $1$-step interpolatory.

Case 2: $c_a$ is an odd integer and $\dm$ is an odd dilation factor.
Then $s_a=\frac{c_a}{2(\dm-1)}=\frac{c_a (\dm+1)/2}{\dm^2-1}$ satisfies the condition \eqref{cond:sa} with $m_s=0$ and $n_s=2$, due to $(\dm+1)/2\in \Z$. Therefore, according to item (3) of \cref{thm:int}, its subdivision scheme is $2$-step interpolatory with the integer shift $(\dm^2-1)s_a$ (i.e., $c_a (\dm+1)/2$).
As we discussed above,
$\phi(s_a+\cdot)$ is
a standard interpolating $\dm^2$-refinable function with the interpolatory mask $A_2((\dm^2-1)s_a+\cdot)$, where the mask $A_2$ is defined in \eqref{An}.

For $s_a$ satisfying \eqref{cond:sa} with $m_s=0$ and $n_s=2$, or equivalently, $s_a=\frac{k}{\dm^2-1}$ for some $k\in \Z$,
we now show that Construction Procedure described in Subsection~\ref{subsec:construct} becomes much simpler.
Because $m_s=0$ and $n_s=2$, the equations in \eqref{cond:ms=0} of \cref{thm:int} can be equivalently expressed as
\be \label{ms=0:ns=2}
\sum_{j\in \Z} a(j)a((\dm^2-1)s_a+\dm^2 k-\dm j)=\dm^{-2} \td(k), \qquad k\in \Z.
\ee
We can easily observe that \eqref{a:lpm} in (S2) and $\sr(a,\dm) \ge J$ together are equivalent to
\be \label{ms=0:ns=2:sr=J}
\sum_{k\in \Z} k^j a(\gamma+\dm k)=\dm^{-1-j}(m_a-\gamma)^j,\qquad \mbox{ for all } j=0,\ldots,J-1
\ee
and for all $\gamma=0,\ldots,\dm-1$.
Recall that the $\gamma$-coset mask $a^{[\gamma:\dm]}(k):=a(\gamma+\dm k)$ for all $k\in \Z$ as defined in \eqref{coset}. If
\be \label{gammaa}
\#S_{\gamma_a}=J \quad \mbox{with}\quad
S_{\gamma_a}:=\fs(a^{[\gamma_a:\dm]}) \subseteq \Z,\quad
\gamma_a:=(\dm^2-1)s_ a,
\ee
where $\# S_{\gamma_a}$ is the cardinality of the set $S_{\gamma_a}$, then using the invertibility of a square Vandermonde matrix, one can easily conclude (e.g., see \cite[Theorem~2.1]{han99}) that the linear equations in \eqref{ms=0:ns=2:sr=J} for the particular $\gamma=\gamma_a$ must have a unique solution of $\{a(\gamma_a+\dm k)\}_{k\in S_{\gamma_a}}$, i.e., the linear equations
\be \label{ms=0:ns=2:gammaa}
\sum_{k\in S_{\gamma_a}} k^j a(\gamma_a+\dm k)=\dm^{-1-j}(m_a-\gamma_a)^j,\qquad \mbox{ for all } j=0,\ldots,J-1
\ee
must have a unique solution for $\{a(\gamma_a+\dm k)\}_{k\in S_{\gamma_a}}$.
Thus, because $a^{[\gamma_a:\dm]}$ on the set $S_{\gamma_a}$ (with the convention that
$a^{[\gamma_a:\dm]}(k)=0$ for all $k\in \Z\bs S_{\gamma_a}$) is uniquely determined and available now, the nonlinear equations in \eqref{ms=0:ns=2} simply become a system of linear equations, which can be easily solved.

Consequently, Construction Procedure in Subsection~\ref{subsec:construct} can be significantly reduced to

\noindent \textbf{Special Construction Procedure:}
{\it Suppose that $s_a$ satisfies \eqref{cond:sa} with $m_s=0, n_s=2$ (i.e., $s_a=\frac{k}{\dm^2-1}$ with $k\in \Z$). All the desired masks $a\in \lp{0}$ in \cref{thm:int} with $\sr(a,\dm)\ge J$ can be obtained by the following procedure:

\begin{enumerate}
\item[(S1)] Parameterize masks $a$ by $\tilde{\pa}(z):=(1+z+\cdots+z^{\dm-1})^J \tilde{\pb}(z)$ with unknown $b=\{b(l_b),\ldots, b(h_b)\}_{[l_b,h_b]}$. Note that $\fs(a)=[l_b, h_b+(\dm-1)J]$. To have symmetric masks $a$, we additionally require $b(k)=b(h_b+l_b-k)$ for $k=l_b,\ldots, h_b$.

\item[(S2)] Let $\gamma_a:=(\dm^2-1)s_a$ and obtain the coset mask $a^{[\gamma_a:\dm]}$ from the parameterized mask $a$. Then solve the system of linear equations \eqref{ms=0:ns=2:gammaa} for $\{a^{[\gamma_a:\dm]}(k)\}_{k\in S_{\gamma_a}}$ with $S_{\gamma_a}:=\fs(a^{[\gamma_a:\dm]})$.

\item[(S3)] Set $a^{[\gamma_a:\dm]}(k)=0$ for all $k\in \Z \bs S_{\gamma_a}$. Use the remaining freedoms in the mask $a$ to solve the nonlinear equations \eqref{ms=0:ns=2}, which become a system of linear equations if the solution $\{a^{[\gamma_a:\dm]}(k)\}_{k\in S_{\gamma_a}}$ in (S2) is one of the following cases:
\begin{enumerate}
\item[(1)] The solution $\{a^{[\gamma_a:\dm]}(k)\}_{k\in S_{\gamma_a}}$ in (S2) is unique, which is true if $\#S_{\gamma_a}=J$ or if one could increase the integer $J$ in \eqref{ms=0:ns=2:gammaa} until it has a unique solution $\{a^{[\gamma_a:\dm]}(k)\}_{k\in S_{\gamma_a}}$.

\item[(2)] The free parameters in solution $\{ a^{[\gamma_a:\dm]}(k)\}_{k\in S_{\gamma_a}}$ of (S2) are not treated as unknowns in (S3) or simply preassigned  parameter values in advance before solving \eqref{ms=0:ns=2} in (S3).
\end{enumerate}
\end{enumerate}
} 

As we shall see in the following example,
quite often the solution in (S2) is unique even if
$\#S_{\gamma_a}>J$. Consequently, we only need to solve linear equations in (S3). Note that the subdivision schemes are $2$-step interpolatory.
For an odd dilation factor $\dm$ and $s_a=\frac{1}{2(\dm-1)}$, our computation indicates that there often exist desired unique symmetric masks $a$ satisfying \eqref{ms=0:ns=2} with the highest possible order $J$ of sum rules with respect to a prescribed filter support $\fs(a)$. Here we provide an example of $2$-step interpolatory $\dm$-subdivision schemes with $\dm=3$ by using Special Construction Procedure.

\begin{example}\label{ex:M3}
\rm
Let $\dm=3$ and $s_a=\frac{c_a}{2(\dm-1)}$ with $c_a=1$. Note that $\gamma_a:=(\dm^2-1)s_a=2$.
We consider symmetric masks $a$ such that $\fs(a)=[-3,4]$ and $\sr(a,\dm)=J$ with $J=2$.
We parameterize masks $a$ in (S1) by $\tilde{\pa}(z)=(1+z+z^2)^J \tilde{\pb}(z)$ with $b=\{t_1, t_2, t_2, t_1\}_{[-3,0]}$. Note that $S_{\gamma_a}:=\fs(a^{[\gamma_a:\dm]})=\{-1,0\}$. and $\#S_{\gamma_a}=J$. Hence, the condition \eqref{gammaa} guarantees the unique solution $\{a^{[\gamma_a:\dm]}(k)\}_{k\in S_{\gamma_a}}$ to \eqref{ms=0:ns=2:gammaa} in (S2), which is given through the solution
$t_2=\tfrac{1}{18}-t_1$ with $t_1\in \R$. Explicitly,
\[
a^{[\gamma_a:\dm]}=\{\tfrac{1}{6}, \tfrac{1}{6}\}_{[-1,0]} \quad \mbox{and}\quad
b=\{t_1, \tfrac{1}{18}-t_1, \tfrac{1}{18}-t_1, t_1\}_{[-3,0]}.
\]
Now solving the linear equations \eqref{ms=0:ns=2} in (S3) with $\fs(a)=[-3,4]$, we obtain a unique solution $t_1=-\frac{1}{36}$ and hence we obtain a
symmetric mask $a\in \lp{0}$ with $\sr(a,\dm)=2$:
\[
a=\{ -\tfrac{1}{36}, \tfrac{1}{36}, \tfrac{1}{6}, \tfrac{1}{3},\tfrac{1}{3},
\tfrac{1}{6}, \tfrac{1}{36}, -\tfrac{1}{36}\}_{[-3,4]},\qquad
b=\{-\tfrac{1}{36}, \tfrac{1}{12}, \tfrac{1}{12}, -\tfrac{1}{36}\}_{[-3,0]}.
\]
By calculation, we have $\sm_2(a,\dm)\approx 1.393267$. Moreover, we conclude from \eqref{sm:est:n} and \eqref{sm:special} with $\gamma_0=-1$ that $\rho_0(b,\dm)_\infty=\dm |b(-1)|=\frac{1}{4}$ and hence
$\sm_\infty(a,\dm)=-\log_\dm \rho_0(b,\dm)_\infty=
\log_3 4\approx 1.261860$.
By \cref{thm:int}, its symmetric
$\dm$-refinable function $\phi\in \CH{1}$ must be $s_a$-interpolating and its $2$-step interpolatory $\dm$-subdivision scheme must be $\mathscr{C}^1$-convergent.

Next we consider symmetric masks $a$ such that $\fs(a)=[-6,7]$ and $\sr(a,\dm)=J$ with $J=3$.
We parameterize masks $a$ in (S1) by $\tilde{\pa}(z)=(1+z+z^2)^J\tilde{\pb}(z)$ with $b=\{t_1, t_2, t_3, t_4, t_4, t_3, t_2, t_1\}_{[-6,1]}$. Then $S_{\gamma_a}=\{-2,-1,0,1\}$ and hence, $\#S_{\gamma_a}=4>J=3$. Obtaining the coset mask $a^{[\gamma_a:\dm]}$ and solving
the linear system \eqref{ms=0:ns=2:gammaa} in (S2), we find that $a^{[\gamma_a:\dm]}$ is actually uniquely determined by \eqref{ms=0:ns=2:gammaa} with
a solution $t_3=-\frac{1}{48}-6t_1-3t_2$ and $t_4=\frac{17}{432}+5t_1+2t_2$ for free parameters $t_1, t_2$. Explicitly,
\[
a^{[\gamma_a:\dm]}=\{-\tfrac{1}{48}, \tfrac{3}{16}, \tfrac{3}{16},-\tfrac{1}{48}\}_{[-2,1]}.
\]
Consequently, solving the linear equations \eqref{ms=0:ns=2} in (S3),
we obtain a solution $t_2=-\frac{1}{432}-4t_1$ and the symmetric mask $a\in \lp{0}$ with the symmetry center $1/2$ and $\sr(a,\dm)=3$ is given by
\be \label{a:ex:M3:C2}
a=\{t_1, -\tfrac{1}{432}-t_1,
-\tfrac{1}{48}, -\tfrac{1}{48}-2t_1, \tfrac{17}{432}+2t_1, \tfrac{3}{16}, \tfrac{137}{432}, \tfrac{137}{432},
\tfrac{3}{16},
\tfrac{17}{432}+2t_1,
-\tfrac{1}{48}-2t_1,
-\tfrac{1}{48},
-\tfrac{1}{432}-t_1, t_1\}_{[-6,7]},
\ee
where $t_1\in \R$.
For $t_1=0$, we have $\sm_2(a,\dm)\approx 2.173176$ and $\fs(a)=[-5,6]$.
For $t_1=\frac{1}{432}$,
we have $\sm_2(a,\dm)\approx 2.458912$
and $\sm_\infty(a,\dm)\ge 2.136745>2$ by  using \eqref{sminfty} with $n=4$.
According to \cref{thm:int}, its symmetric $\dm$-refinable function $\phi\in \CH{2}$ must be $s_a$-interpolating and its $2$-step interpolatory $\dm$-subdivision scheme must be $\mathscr{C}^2$-convergent.

Next we consider symmetric masks $a$ such that $\fs(a)=[-11,12]$ and $\sr(a,\dm)=J$ with $J=5$.
We parameterize masks $a$ in (S1) by $\tilde{\pa}(z)=(1+z+z^2)^J\tilde{\pb}(z)$ with
\[
b=\{t_1, t_2, t_3, t_4, t_5, t_6, t_7, t_7, t_6, t_5, t_4, t_3, t_2, t_1\}_{[-11,2]}.
\]
Then we have $S_{\gamma_a}=\fs(a^{[\gamma_a:\dm]})
=[-4,3]\cap \Z$ and hence, $\#S_{\gamma_a}=8>J=5$. Obtaining the coset $a^{[\gamma_a:\dm]}$ and solving
the linear equations \eqref{ms=0:ns=2:gammaa} in (S2), we have a solution
\[
t_5=\tfrac{1}{256}-70t_1-35t_2-15t_3-5t_4,
\quad
t_6=-\tfrac{2875}{186624}+189t_1+90t_2+35t_3+9t_4,
\quad
t_7=\tfrac{1265}{93312}-120t_1-56t_2-5t_4
\]
with free parameters $t_1, t_2, t_3, t_4\in \R$.
Then we obtain $a^{[\gamma_a: \dm]}$ with only one free parameter below:
\[
a^{[\gamma_a:\dm]}=\{t, \tfrac{1}{256}-5t, -\tfrac{25}{768}+9t, \tfrac{25}{128}-5t,
\tfrac{25}{128}-5t, -\tfrac{25}{768}+9t, \tfrac{1}{256}-5t,t\}_{[-4,3]}
\]
with $t:=5t_1+t_2$. If we preset $t=0$ (i.e., set $t_2=-5t_1$), then we only need to solve linear equations in (S3) with three unknowns $\{t_1, t_3, t_4\}$ in the symmetric mask $b$. The solution is given by
$t_3=-40t_1, t_4=\frac{5305}{9144576}+291t_1$ and hence
\[
b|_{[-4,2]}=\{\tfrac{97445}{9144576}-455t_1, -\tfrac{46565}{4572288} + 958t_1, \tfrac{2299}{2286144} - 750t_1,
\tfrac{5305}{9144576} + 291t_1, -40t_1, -5t_1, t_1\}_{[-4,2]}
\]
with $t_1\in \R$. Optimizing the smoothness quantity $\sm_2(a,\dm)$ and choosing
$t_1=\frac{1}{150528}$, we obtain a symmetric mask $a\in \lp{0}$ with symmetry center $1/2$ and $\sr(a,\dm)=5$ such that
\[
a|_{[1,12]}=\{
\tfrac{11558345}{36578304},
\tfrac{25}{128},
\tfrac{921259}{18289152},
-\tfrac{59711}{2032128},
-\tfrac{25}{768},
-\tfrac{110615}{12192768},
\tfrac{178057}{36578304},
\tfrac{1}{256},
\tfrac{16199}{18289152},
-\tfrac{25}{75264},
0,
\tfrac{1}{150528}\}_{[1,12]}
\]
with $\sm_2(a,\dm)\approx 3.329871$ and
$\sm_\infty(a,\dm)\ge 3.136794$ by
using \eqref{sminfty} with $n=2$.
By \cref{thm:int}, its
$\dm$-refinable function $\phi\in \CH{3}$ must be $s_a$-interpolating and its $2$-step interpolatory subdivision scheme must be $\mathscr{C}^3$-convergent.
See \cref{fig:ex:M3} for graph of the $s_a$-interpolating $\dm$-refinable function $\phi$. We mention that an example of $\mathscr{C}^3$-convergence $2$-step interpolatory $3$-subdivision schemes is reported in \cite[(28)]{grv22} whose mask has support $[-14,15]$, which is longer than the support $[-11,12]$ of our $\mathscr{C}^3$ example here.
\end{example}

\begin{figure}[htbp]
	\centering
\begin{subfigure}[b]{0.24\textwidth} \includegraphics[width=\textwidth,height=0.4\textwidth]
{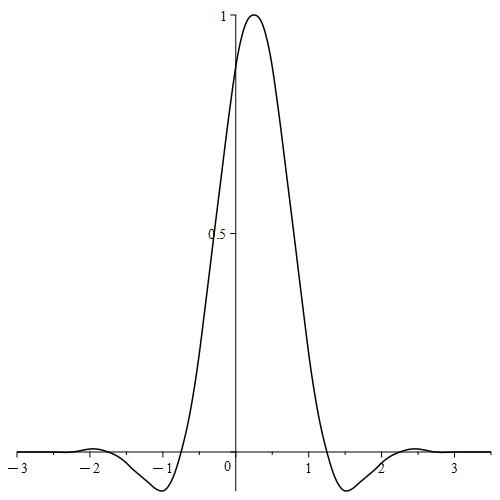} \caption{$\phi\in \CH{2}$}
	\end{subfigure}
	 \begin{subfigure}[b]{0.24\textwidth} \includegraphics[width=\textwidth,height=0.4\textwidth]
{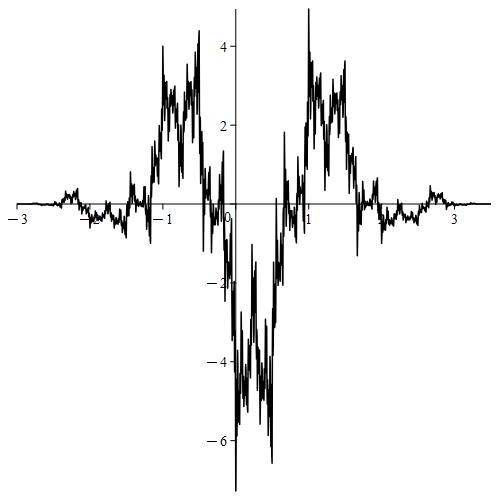} \caption{$\phi''$}
	\end{subfigure} \begin{subfigure}[b]{0.24\textwidth}	 \includegraphics[width=\textwidth,height=0.4\textwidth]
{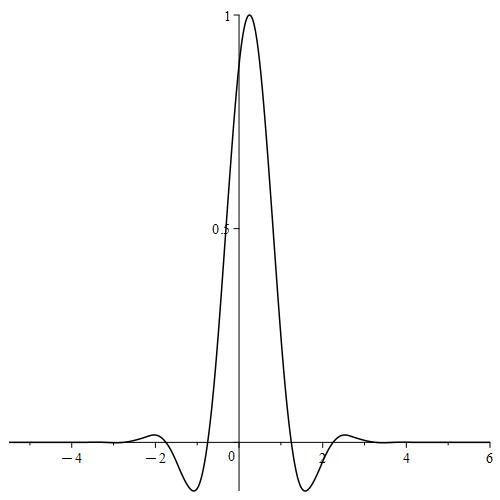}
		\caption{$\phi\in \CH{3}$}
\end{subfigure}
\begin{subfigure}[b]{0.24\textwidth}	 \includegraphics[width=\textwidth,height=0.4\textwidth]
{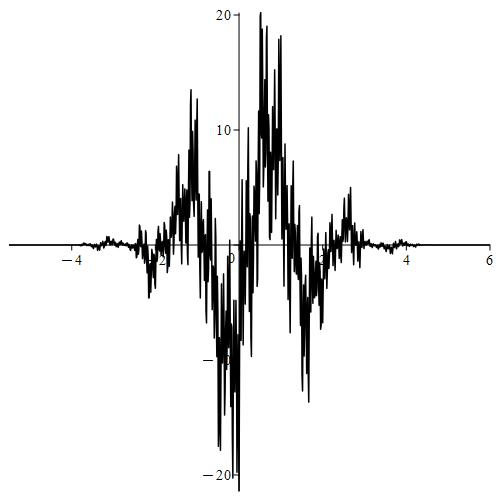}
		\caption{$\phi'''$}
	\end{subfigure}
\caption{
(a) is the graph of the $\frac{1}{4}$-interpolating $3$-refinable function $\phi\in \CH{2}$ in \cref{ex:M3} with the mask $a$ in \eqref{a:ex:M3:C2} with $t_1=\frac{1}{432}$, $\sr(a,3)=3$, $\fs(a)=[-6,7]$ and $\mbox{supp}(\phi)=[-3,\frac{7}{2}]$.
(b) is the graph of the second-order derivative $\phi''$ in (a).
(c) is the graph of the $\frac{1}{4}$-interpolating $3$-refinable function $\phi\in \CH{3}$ with the mask $a$ with $t_1=\frac{1}{150528}$, $\sr(a,3)=5$,
$\fs(a)=[-11,12]$ and $\mbox{supp}(\phi)=[-\frac{11}{2},6]$.
(d) is the graph of the third-order derivative $\phi'''$ in (c).
}\label{fig:ex:M3}
\end{figure}

\subsection{Examples of $n_s$-step interpolatory dyadic subdivision schemes with $\dm=2$}

In this subsection, we only consider $\dm=2$.
For $s_a$-interpolating $2$-refinable functions with symmetric masks,
we know from \eqref{ca:sa} that $s_a=\frac{c_a}{2(\dm-1)}=c_a/2\in [\frac{1}{2}+\Z]$ must hold for any odd integer $c_a$. Before presenting some examples, generalizing a result in \cite{grv22} for symmetric masks, we prove that even without symmetry, there are no $s_a$-interpolating $2$-refinable functions for $s_a\in [\frac{1}{2}+\Z]$.

\begin{lemma}\label{lem:M=2}
For $\dm=2$ and $s_a\in [\frac{1}{2}+\Z]$, there does not exist a compactly supported continuous $s_a$-interpolating $\dm$-refinable function with a finitely supported mask $a\in \lp{0}$.
\end{lemma}

\begin{proof} We use proof by contradiction.
Suppose not. Then we have an $s_a$-interpolating $\dm$-refinable function $\phi$ with a finitely supported mask $a\in \lp{0}$. As we discussed before, it suffices to consider $s_a=\frac{1}{2}$. Define $[l_a,h_a]:=\fs(a)$, the filter support of the mask $a$. Then $a(l_a)a(h_a)\ne 0$. Define a sequence $w$ by $w(k):=\phi(1+k)$ for all $k\in \Z$ and define $[l_w, h_w]:=\fs(w)$. Note that $w(l_w)w(h_w)\ne 0$.
From the refinement equation $\phi(x)=\sum_{k=l_a}^{h_a} a(k) \phi(2x-k)$ with $x=1+j$ and $x=\frac{1}{2}+j$ with $j\in \Z$, noting $\phi(1+j)=w(j)$ and $\phi(\frac{1}{2}+j)=\td(j)$ for all $j\in \Z$,
we have
\begin{align}
&\sum_{k=l_a}^{h_a}a(k) w(1+2j-k)=2^{-1}w(j),\qquad \forall\; j\in \Z, \label{M=2:1}\\
&\sum_{k=l_a}^{h_a} a(k)w(2j-k)=2^{-1}\td(j),\qquad \forall\; j\in \Z.\label{M=2:2}
\end{align}
Note that $\phi$ must be supported inside $[l_a,h_a]$ and $\phi(l_a)=\phi(h_a)=0$ because $\phi$ is continuous. Therefore, we must have
\be \label{la:lw}
l_a\le l_w\le h_w \le h_a-2.
\ee
If $l_a+l_w$ is an odd integer, then \eqref{M=2:1} with $j=\frac{l_a+l_w-1}{2}$ becomes $ a(l_a)w(l_w)=2^{-1}w(\frac{l_a+l_w-1}{2})$. Since $a(l_a)w(l_w)\ne 0$, we must have $\frac{l_a+l_w-1}{2}\ge l_w$, i.e., $l_w\le l_a-1$, contradicting \eqref{la:lw}. Hence, $l_a+l_w$ must be an even integer.
Now \eqref{M=2:2} with $j=\frac{l_a+l_w}{2}$ becomes $a(l_a)w(l_w)=2^{-1} \td(\frac{l_a+l_w}{2})$, which forces $l_w=-l_a$ because $a(l_a)w(l_w)\ne 0$.

If $h_a+h_w$ is an odd integer, then \eqref{M=2:1} with $j=\frac{h_a+h_w-1}{2}$ becomes $a(h_a) w(h_w)=2^{-1} w(\frac{h_a+h_w-1}{2})$, which forces $\frac{h_a+h_w-1}{2}\le h_w$, that is, $h_w\ge h_a-1$, contradicting \eqref{la:lw}. Therefore, $h_a+h_w$ must be an even integer. Then \eqref{M=2:2} with $j=\frac{h_a+h_w}{2}$ becomes $a(h_a) w(h_w)=2^{-1} \td(\frac{h_a+h_w}{2})$, which forces $h_a+h_w=0$, that is, $h_w=-h_a$, due to $a(h_a) w(h_w)\ne 0$.

Hence, we proved $l_w=-l_a$ and $h_w=-h_a$, from which we must have $l_a=h_a$ by $l_w\le h_w$ and $l_a\le h_a$. But $l_a=h_a$ contradicts $l_a\le h_a-2$ in \eqref{la:lw}.
This proves the nonexistence of continuous $s_a$-interpolating $2$-refinable functions with finitely supported masks $a\in \lp{0}$.
\end{proof}

We now present a few examples of $s_a$-interpolating $2$-refinable functions and their dyadic subdivision schemes using \cref{thm:int} and Special Construction Procedure in Subsection~\ref{subsec:ms0}.

\begin{example}\label{ex:M2}
\rm
Let $\dm=2$ and $s_a=\frac{1}{3}$ which satisfies \eqref{cond:sa} with $m_s=0$ and $n_s=2$. Note that $\gamma_a:=(\dm^2-1)s_a=1$.
We consider masks $a$ with $\fs(a)=[-2,4]$ and $\sr(a,\dm)=J$ with $J=2$. We parameterize masks $a$ in (S1) by $\tilde{\pa}(z)=(1+z)^J \tilde{\pb}(z)$ with $b=\{t_1, t_2, t_3, t_4, t_5\}_{[-2,2]}$. Then $S_{\gamma_a}:=\fs(a^{[\gamma_a:\dm]})=\{-1, 0, 1\}$ and $\#S_{\gamma_a}=3>J=2$.
Obtaining the coset $a^{[\gamma_a:\dm]}$ and solving the linear equations \eqref{ms=0:ns=2:gammaa} in (S2), we have a solution $t_4=\frac{2}{3}-4t_1-3t_2-2t_3, t_5=-\frac{5}{12}+3t_1+2t_2+t_3$ with the free parameters $t_1, t_2, t_3\in \R$. Now
the coset mask $a^{[\gamma_a:\dm]}$ is given by
\[
a^{[\gamma_a:\dm]}=\{\tfrac{1}{6}+t, \tfrac{1}{3}-2t, t\}_{[-1,1]},
\]
where $t:=2t_1+t_2-\frac{1}{6}$.
Not regarding $t$ as an unknown (i.e., set $t_2=\frac{1}{6}-2t_1+t$ and only solving for $\{t_1, t_3\}$ but not $t$), we see that
the nonlinear equations \eqref{ms=0:ns=2} of (S3) actually becomes linear equations, which have a unique solution $t_1=\frac{36t^2 + 12t + 1}{12(6t - 1)}, t_3 = \frac{t(12t - 7)}{2(1-6t)}$, which leads to
\be \label{exM2C1}
a=\left\{\tfrac{(1+6t)^2}{72t-12},
\tfrac{1}{6}+t, \tfrac{36t^2-6s+7}{12-72t},
\tfrac{1}{3}-2t, \tfrac{6s^2-3t}{2-12t}, t, \tfrac{3t^2}{6t-1} \right\}_{[-2,4]},
\ee
where $t\in \R\bs \{\frac{1}{6}\}$.
For $t=0$,
we have $\fs(a)=[-2,1]$ and $\sm_2(a,\dm)\approx 1.04123$;
Moreover, by $b=\{-\frac{1}{12}, \frac{1}{3}\}_{[-2,-1]}$, we conclude from \eqref{sm:est:n} and \eqref{sm:special} with $\gamma_0=-1$ (also see \cite[Theorem~2.1 and Corollary~2.2]{han98}) that
$\rho_0(b,\dm)_\infty=\dm |b(\gamma_0)|=\frac{2}{3}$ and hence
$\sm_\infty(a,2)=-\log_2 \tfrac{2}{3}\approx0.584962$. By \cref{thm:int}, its $\dm$-refinable function $\phi$ must be $s_a$-interpolating.
For $t=-\frac{1}{18}$,
we have $\sm_2(a,\dm)\approx 1.821703$ and hence, $\sm_\infty(a,\dm)\ge \sm_2(a,\dm)-0.5\ge 1.243484$. In fact, $\sm_\infty(a,\dm)\ge 1.305626$ by \eqref{sminfty} with $n=5$.
Hence, according to \cref{thm:int},
its $\dm$-refinable function $\phi\in \mathscr{C}^1(\R)$ must be $s_a$-interpolating and its $2$-step interpolatory dyadic subdivision scheme is $\mathscr{C}^1$-convergent.
See \cref{fig:ex:M2} for the graph of the $s_a$-interpolating $2$-refinable function $\phi$ with the mask $a$ in \eqref{exM2C1} and
$t=-\frac{1}{18}$.

For $\dm=2$ and $s_a=\frac{1}{3}$,
we consider masks $a$ with $\fs(a)=[-6,4]$ and $\sr(a,\dm)=J$ with $J=4$. We parameterize masks $a$ in (S1) by $\tilde{\pa}(z)=(1+z)^J \tilde{\pb}(z)$ with $b=\{t_1, t_2, t_3, t_4, t_5, t_6, t_7\}_{[-6,0]}$. Note that $\gamma_a=1$ and $S_{\gamma_a}:=\fs(a^{[\gamma_a:\dm]})=[-3,1]\cap \Z$. Hence, $\#S_{\gamma_a}=5>J=4$.
Obtaining the coset mask $a^{[\gamma_a:\dm]}$ and solving the linear equations \eqref{ms=0:ns=2:gammaa} in (S2), we obtain a solution
\begin{align*}
&t_1 = -\tfrac{91}{2592}+ t_5
+4t_6+ 10t_7,\quad
t_2 = \tfrac{37}{216} - 4t_5 - 15t_6 - 36t_7,\quad\\
&t_3 = 45t_7 + 20t_6 + 6t_5 - \tfrac{277}{864},\quad
t_4 = \tfrac{20}{81} - 4t_5 - 10t_6 - 20t_7
\end{align*}
with free parameters $t_5, t_6, t_7\in \R$.
Hence, we find that $a^{[\gamma_a:\dm]}$ has only one free parameter given by
\[
a^{[\gamma_a:\dm]}=\{
\tfrac{5}{162}+s, -\tfrac{4}{27}-4s, \tfrac{10}{27}+6s, \tfrac{20}{81}-4s,
s\}_{[-1,1]},
\]
where $s:=t_6+4t_7$.
Now solving the nonlinear equations \eqref{ms=0:ns=2} in (S3), we obtain four solution families with complicated expressions. One of the solutions is given by
\[
t_5=\tfrac{4}{81}t^3 + \tfrac{37}{81}t^2 - \tfrac{991}{5184}t + \tfrac{323}{10368},\quad
t_6:=\tfrac{5}{81}t,\quad
t_7=-\tfrac{4}{405}t^3 - \tfrac{13}{135}t^2 - \tfrac{121}{2880}t - \tfrac{323}{51840},
\]
where $t$ is a root of $512 t^4 + 5504 t^3 + 6370 t^2 + 2501t + 323=0$.
For the root $t\approx -0.319621$, we have $\sm_2(a,\dm)\approx 2.25960$. Hence, $\sm_\infty(a,\dm)\ge 1.75960$. By \cref{thm:int}, for the mask $a$ with $t\approx -0.319621$,
the $2$-refinable function $\phi\in \CH{1}$ is $s_a$-interpolating and its subdivision scheme is $\mathscr{C}^1$-convergent $2$-step interpolatory subdivision scheme.
See \cref{fig:ex:M2} for the graph of the $s_a$-interpolating $2$-refinable function $\phi$ with the above mask $a$ and $t\approx -0.319621$, which is approximately given by
\be \label{exM2sr4}
\begin{split}
a\approx \{
&0.00010829639, 0.0018661071,  0.010752801, -0.032155785, - 0.06531331, 0.19638181,\\
&0.51228456,
0.36290592, 0.044484714,
- 0.028998091, - 0.0023170997\}_{[-6,4]}.
\end{split}
\ee
See \cref{fig:ex:M2} for the graph of the $s_a$-interpolating $2$-refinable function $\phi$.
\end{example}

We now consider another example using $s_a=\frac{1}{7}$ which satisfies \eqref{cond:sa} with $m_s=0$ and $n_s=3$. Therefore, we have to use the general Construction Procedure in Subsection~\ref{subsec:construct}.

\begin{example}\label{ex:M2a}
\rm
Let $\dm=2$ and $s_a=\frac{1}{7}$ which satisfies \eqref{cond:sa} with $m_s=0$ and $n_s=3$. Note that \eqref{cond:ms} becomes $w=\td$ due to $m_s=0$.
We consider masks $a$ with $\fs(a)=[-2,1]$ and $\sr(a,\dm)=J$ with $J=2$. We parameterize masks $a$ in (S1) by $\tilde{\pa}(z)=(1+z)^J \tilde{\pb}(z)$ with $b=\{t_1, t_2\}_{[-2,-1]}$ for unknowns $t_1$ and $t_2$. Solving the linear equations \eqref{a:lpm} in (S2) of Construction Procedure, we have a unique solution $t_1=-\frac{1}{28}, t_2=\frac{2}{7}$. Because $m_s=0$, for (S3), we can directly check that the nonlinear equations \eqref{cond:ms=0} with $n_s=3$
are automatically satisfied. Hence, we obtain a unique solution:
\be \label{exM2sr2d7}
a=\{-\tfrac{1}{28},\tfrac{3}{14},\tfrac{15}{28},\tfrac{2}{7}\}_{[-2,1]},
\quad b=\{-\tfrac{1}{28},\tfrac{2}{7}\}_{[-2,-1]}
 \quad \mbox{with}\quad \sm_2(a,\dm)\approx 1.29617.
\ee
By \eqref{sm:est:n} and \eqref{sm:special} with $\gamma_0=-1$, we have $\rho_0(b,2)_\infty=\dm |b(\gamma_0)|=\frac{4}{7}$ and hence
$\sm_\infty(a,2)=-\log_2 \tfrac{4}{7}\approx0.80735$.
Hence, according to \cref{thm:int}, the $2$-refinable function $\phi$ must be $s_a$-interpolating and the dyadic subdivision scheme is $3$-step interpolatory.
See \cref{fig:ex:M2} for the graph of the $s_a$-interpolating $2$-refinable function $\phi$.
\end{example}

\begin{figure}[htbp]
	\centering
\begin{subfigure}[b]{0.24\textwidth} \includegraphics[width=\textwidth,height=0.4\textwidth]{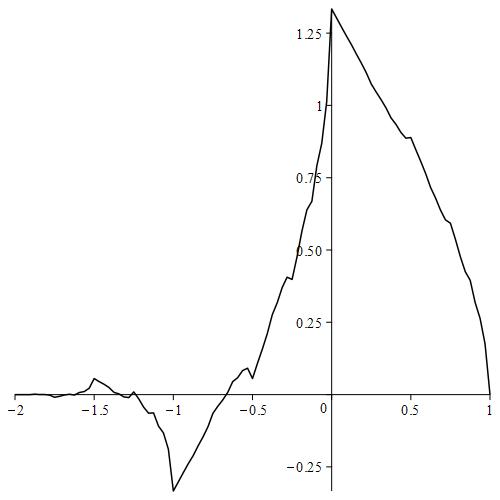} \caption{$\phi\in \CH{0}$}
	\end{subfigure}
	 \begin{subfigure}[b]{0.24\textwidth} \includegraphics[width=\textwidth,height=0.4\textwidth]
{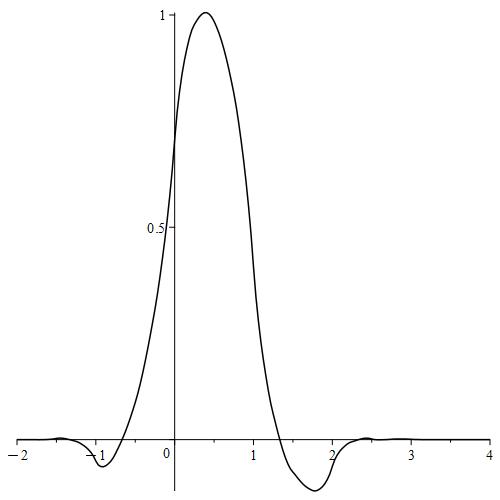} \caption{$\phi\in \CH{1}$}
	\end{subfigure}
	 \begin{subfigure}[b]{0.24\textwidth} \includegraphics[width=\textwidth,height=0.4\textwidth]{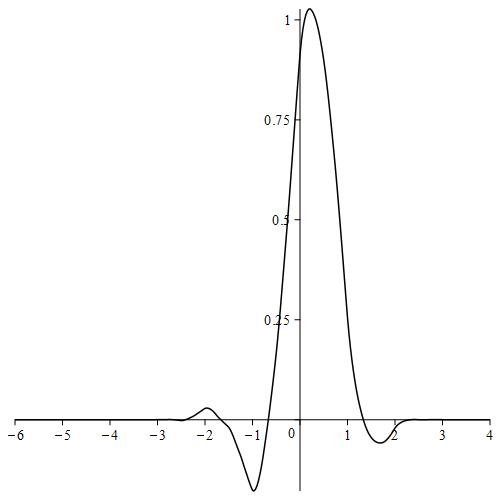} \caption{$\phi\in \CH{1}$}
	\end{subfigure} \begin{subfigure}[b]{0.24\textwidth}	 \includegraphics[width=\textwidth,height=0.4\textwidth]{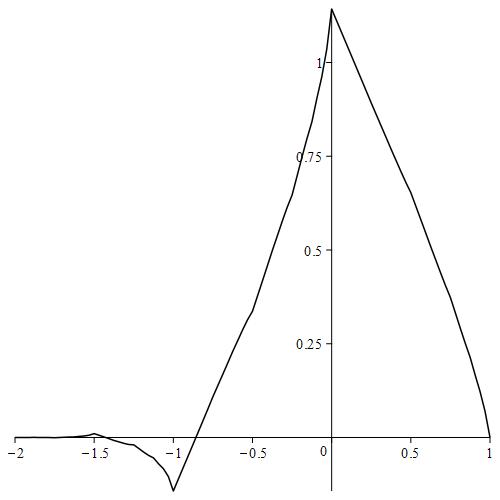}
		\caption{$\phi\in \CH{0}$}
	\end{subfigure}
\caption{
(a) is the graph of the $\frac{1}{3}$-interpolating $2$-refinable function $\phi\in \CH{0}$ in \cref{ex:M2} with the mask $a$ in \eqref{exM2C1} and $t=0$.
(b) is the graph of the $\frac{1}{3}$-interpolating $2$-refinable function $\phi\in \CH{1}$ with the mask $a$ in \eqref{exM2C1} and $t=-\frac{1}{18}$.
(c) is the graph of the $\frac{1}{3}$-interpolating $2$-refinable function $\phi\in \CH{1}$ with the mask $a$ in \eqref{exM2sr4}.
(d) is the graph of the $\frac{1}{7}$-interpolating $2$-refinable function $\phi\in \CH{0}$ in \cref{ex:M2a} with the mask $a$ in \eqref{exM2sr2d7}.
}\label{fig:ex:M2}
\end{figure}

\subsection{The case $m_s>0$ for $s_a$-interpolating $\dm$-refinable functions and $\infty$-step interpolatory subdivision schemes}

For $m_s>0$, we have to employ the general Construction Procedure in Subsection~\ref{subsec:construct} to construct $s_a$-interpolating refinable functions.
Their constructions are often much more complicated and we have to deal with nonlinear equations in \eqref{cond:ms} and \eqref{cond:ns}.

For symmetric masks $a$ satisfying \eqref{symmask}, we must have $s_a=\frac{c_a}{2(\dm-1)}$ in \eqref{ca:sa}.
If $c_a$ is an odd integer and $\dm$ is even,
then $s_a=\frac{c_a}{2(\dm-1)}=\frac{c_a \dm/2}{\dm(\dm-1)}$ satisfies
the condition \eqref{cond:sa} with $m_s=1$ and $n_s=1$. Therefore, according to item (3) of \cref{thm:int}, its subdivision scheme is only $\infty$-step (i.e., limit) interpolatory. We now particularly look at the special case $m_s=n_s=1$. Then the nonlinear equations \eqref{cond:ms} and \eqref{cond:ns} in Construction Procedure with $m_s=n_s=1$ become
\be \label{ms:ns:1}
[a*w](\dm k)=\dm^{-1} \td(k) \quad
\mbox{and}\quad [a*w](\dm(\dm-1)s_a+\dm k)=\dm^{-1} w(k),\qquad k\in \Z.
\ee
The above nonlinear equations in \eqref{ms:ns:1} become linear equations if the solution $w$ to \eqref{w} is unique (which can be always achieved by increasing $J$ in \eqref{w} until it has a unique solution) or the remaining free parameters in the solution $w$ of \eqref{w} are not regarded as unknowns or take preassigned values in advance. For simplicity of presentation, here we only consider $\dm=4$ and symmetric masks.

\begin{example}\label{ex:M4}
\rm
Let $\dm=4$ and $s_a=\frac{c_a}{2(\dm-1)}$ with $c_a=1$. Note that $s_a=\frac{c_a}{2(\dm-1)}=\frac{1}{6}$ satisfies \eqref{cond:sa} with $m_s=n_s=1$.
We consider symmetric masks $a$ with $\fs(a)=[-4,5]$ and $\sr(a,\dm)=J$ with $J=2$. We parameterize $a$ in (S1) by $\tilde{\pa}(z)=(1+z+z^2+z^3)^J \tilde{\pb}(z)$
such that $b=\{t_1, t_2, t_2, t_1\}_{[-4,-1]}$.
Solving the linear equations \eqref{a:lpm}
in (S2) of Construction Procedure, we have $t_1=\frac{1}{32}-t_2$ with the free parameter $t_2$. Because $m_s=1>0$, we have to use (S3') in Construction Procedure. Noting that $l_w=-1$ and $h_w=0$ due to $\mbox{supp}(\phi)=[-\frac{4}{3},\frac{5}{3}]$, we solve the linear equations \eqref{w} of (S3') and we obtain a unique solution $w=\{\frac{1}{2}, \frac{1}{2}\}_{[-1,0]}$. Now the nonlinear equations \eqref{cond:ms} and \eqref{cond:ns} in (S3') (i.e., \eqref{ms:ns:1}) become linear equations.
The linear equations \eqref{ms:ns:1} have a unique solution $t_2=\frac{3}{64}$ and
we obtain a symmetric mask $a\in \lp{0}$ with $\fs(a)=[-4,5]$ and $\sr(a,\dm)=2$:
\[
a=\{-\tfrac{1}{64}, \tfrac{1}{64}, \tfrac{3}{32}, \tfrac{5}{32},\tfrac{1}{4},\tfrac{1}{4},
\tfrac{5}{32},\tfrac{3}{32},\tfrac{1}{64},-\tfrac{1}{64}
\}_{[-4,5]},\qquad b=\{-\tfrac{1}{64}, \tfrac{3}{64}, \tfrac{3}{63}, -\tfrac{1}{64}\}_{[-4,-1]}.
\]
By calculation, we have $\sm_2(a,\dm)\approx 1.419518$.
Using \eqref{sm:est:n} and \eqref{sm:special} with $\gamma_0=-2$, we have $\rho_0(b,\dm)_\infty=\dm |b(\gamma_0)|=\frac{3}{16}$ and
hence, we have $\sm_\infty(a,\dm)=-\log_\dm \rho_0(b,\dm)_\infty=\log_4 \tfrac{16}{3}\approx 1.207519$.
By \cref{thm:int}, its symmetric
$\dm$-refinable function $\phi\in \CH{1}$ must be $\frac{1}{6}$-interpolating and its $\infty$-step interpolatory $\dm$-subdivision scheme must be $\mathscr{C}^1$-convergent.

Next we consider symmetric masks $a$ with $\fs(a)=[-7,8]$ and $\sr(a,\dm)=J$ with $J=3$.
We parameterize $a$ in (S1) by $\tilde{\pa}(z)=(1+z+z^2+z^3)^J \tilde{\pb}(z)$
such that $b=\{t_1, t_2, t_3, t_4, t_3, t_2, t_1\}_{[-7,-1]}$.
Solving the linear equations \eqref{a:lpm} in (S2) of Construction Procedure, we have $t_3=-9t_1-4t_2-\frac{15}{512}$ and $t_4=16t_1+6t_2+\frac{19}{256}$
with the free parameters $t_1, t_2\in \R$. Because $m_s=1>0$, we have to use (S3') in Construction Procedure. Noting that $l_w=-2$ and $h_w=1$ due to $\mbox{supp}(\phi)=[-\frac{7}{3},\frac{8}{3}]$, we solve the linear equations \eqref{w} of (S3') and we obtain
\be \label{w:s}
w=\{-\tfrac{1}{8}-s, \tfrac{3}{4}+3s, \tfrac{3}{8}-3s, s\}_{[-2,1]},
\ee
where $s\in \R$.
Then we use it to further solve the nonlinear equations \eqref{ms:ns:1} in (S3') and obtain a solution
$t_2=6t_1$ and $s=-\tfrac{1}{16}$ with the free parameter $t_1\in \R$.
That is, we now obtain
\[
b=\{t_1, 6t_1, -\tfrac{15}{512}-33t_1, \tfrac{19}{256}+52t_1, -\tfrac{15}{512}-33t_1, 6t_1, t_1\}_{[-7,-1]}.
\]
In fact, if we would use $J=4$ instead of $J=3$ in \eqref{w}, then \eqref{w} has a unique solution in \eqref{w:s} with $s=-\frac{1}{16}$ and then the nonlinear equations \eqref{ms:ns:1} become linear equations, yielding the same solution $t_2=6t_1$. Moreover, up to an integer shift and a multiplicative factor $4$, the above mask $a$ agrees with the mask $A$ reported in \cite[Proposition~3.5]{rom19}.

Optimizing $\sm_2(a,\dm)$ among values of $t_1$ as described in Subsection~\ref{subsec:sm}, we
take $t_1=-\frac{1}{832}$ and
obtain a symmetric mask $a\in \lp{0}$ with the symmetry center $1/2$ and $\sr(a,\dm)=3$ given by
\[
a=\{
-\tfrac{1}{832},
-\tfrac{9}{832},
-\tfrac{123}{6656},
-\tfrac{83}{6656},
\tfrac{141}{6656},
\tfrac{645}{6656},
\tfrac{607}{3328},
\tfrac{807}{3328},
\tfrac{807}{3328},
\tfrac{607}{3328},\tfrac{645}{6656},
\tfrac{141}{6656},-\tfrac{83}{6656},
-\tfrac{123}{6656},
-\tfrac{9}{832},-\tfrac{1}{832}
\}_{[-7,8]}.
\]
By calculation, we have $\sm_2(a,\dm)\approx 2.264759$ and $\sm_\infty(a,\dm)\ge 2.132628$ using \eqref{sminfty} with $n=2$.
By \cref{thm:int}, its symmetric
$\dm$-refinable function $\phi\in \CH{2}$ must be $\frac{1}{6}$-interpolating and its $\infty$-step interpolatory subdivision scheme must be $\mathscr{C}^2$-convergent.

Finally, we consider symmetric masks $a$ with $\fs(a)=[-12,13]$ and $\sr(a,\dm)=J$ with $J=5$.
We parameterize $a$ in (S1) by $\tilde{\pa}(z)=(1+z+z^2+z^3)^J \tilde{\pb}(z)$
such that $b=\{t_1, t_2, t_3, t_4, t_5, t_6, t_5, t_4, t_3, t_2, t_1\}_{[-12,-2]}$.
Solving the linear equations \eqref{a:lpm} in (S2) of Construction Procedure, we have
\[
t_4=\tfrac{715}{131072}-50t_1-20t_2-6t_3,\quad
t_5=-\tfrac{815}{32768}+175t_1+64t_2+15t_3,\quad
t_6=\tfrac{2609}{65536}-252t_1-90t_2-20t_3
\]
with the free parameters $t_1, t_2, t_3$. Because $m_s=1>0$, we have to use (S3') in Construction Procedure. Noting that $l_w=-4$ and $h_w=3$ due to $\mbox{supp}(\phi)=[-4,\frac{13}{3}]$, we solve the linear equations \eqref{w} of (S3') and we obtain
\begin{align*}
w=\{
&s_1, s_2, s_3, \tfrac{35}{128}-35s_1-15s_2-5s_3,
\tfrac{35}{32}+105s_1+40s_2+10s_3,
-\tfrac{35}{64}-126s_1-45s_2-10s_3,\\
&\qquad \tfrac{7}{32}+70s_1+24s_2+5s_3, -\tfrac{5}{128}-15s_1-5s_2-s_3\}_{[-4,3]},
\end{align*}
where $s_1,s_2,s_3\in \R$.
Then we use it to further solve the nonlinear equations \eqref{cond:ms} and \eqref{cond:ns} in (S3') and obtain three solution families. The solution with the simplest expresses is given by
\[
t_2=\tfrac{10}{3}t_1,\quad t_3=\tfrac{2145}{3670016}+\tfrac{55}{3}t_1,\quad s_1=0,\quad
s_2=\tfrac{3}{256},\quad
s_3=-\tfrac{25}{256}
\]
with $t_1\in \R$.
Optimizing $\sm_2(a,\dm)$ among values of $t_1$ as described in Subsection~\ref{subsec:sm}, we
take $t_1=\frac{103}{3670016}$ and
obtain a symmetric mask $a\in \lp{0}$ with the symmetry center $1/2$ and $\sr(a,\dm)=5$ such that $a|_{[1,13]}$ is given by
\begin{align*}
\{
&\tfrac{745}{3072},
\tfrac{343905}{1835008},
\tfrac{188627}{1835008},
\tfrac{284335}{11010048},
-\tfrac{183955}{11010048},
-\tfrac{101845}{3670016},
-\tfrac{65735}{3670016},
-\tfrac{10793}{2752512},
\tfrac{5585}{2752512},
\tfrac{12725}{3670016},
\tfrac{7295}{3670016},
\tfrac{2575}{11010048},
\tfrac{103}{3670016}\}_{[1,13]}
\end{align*}
with $\sm_2(a,\dm)\approx 3.109024$ and $\sm_\infty(a,\dm)\ge 2.873247$ using \eqref{sminfty} with $n=6$.
By \cref{thm:int}, its
$\dm$-refinable function $\phi\in \CH{2}$ must be $\frac{1}{6}$-interpolating and its $\infty$-step interpolatory $\dm$-subdivision scheme must be $\mathscr{C}^2$-convergent.
See \cref{fig:ex:M4} for graphs of the $s_a$-interpolating $\dm$-refinable functions $\phi$.
\end{example}

Finally, combining \cref{thm:int,thm:qs} (more precisely, see \cref{cor:qs}),
we present an example of $s_a$-interpolating $4$-refinable function using $2$-mask quasi-stationary subdivision schemes.

\begin{example}\label{ex:M2:M4}
\rm
Let $a_1, a_2\in \lp{0}$ be symmetric dyadic masks with $\sr(a_1,2)=\sr(a_2,2)=J$ with $J=3$ as follows:
\begin{align*}
&\widetilde{\pa_1}(z)=\tfrac{1}{8}z^{-1}(1+z)^3
(t_1 z^{-1}+1-2t_1+t_1z),\\
&\widetilde{\pa_2}(z)=\tfrac{1}{8}z^{-2}(1+z)^3
(t_3z^{-2}+t_2 z^{-1}+1-2t_2-2t_3+t_2z+t_3 z^2),
\end{align*}
Let $\dm:=4$ and define a new mask $a\in \lp{0}$ by $\tilde{\pa}(z):=\widetilde{\pa_1}(z^2)\widetilde{\pa_2}(z)$.
Then $\sr(a,\dm)=3$, $\fs(a)=[-8,9]$ and $a$ is symmetric about the point $1/2$.
Applying Construction Procedure and solving nonlinear equations, we obtain a solution given by
\[
t_2=t(1088t_1^2 + 510t_1 - 15),\quad
t_3=-t(64t_1^4 +62t_1^3 +271t_1^2 + 128t_1),
\]
where $t:=(64t_1^3 + 32t_1^2 - 16t_1 + 8)^{-1}$, and
\[
w=\tfrac{1}{16-32t_1}
\{2t_1^3+2t_1^2,
-6t_1^3 - 3t_1^2 + 2t_1 - 1,
4t_1^3 + 2t_1^2 - 18t_1 + 9,
4t_1^3 + 2t_1^2 - 18t_1 + 9,
-6t_1^3 - 3t_1^2 + 2t_1 - 1,
2t_1^3 + t_1^2\}_{[-3,2]},
\]
where  $t_1\in \R$ is a free parameter. Optimizing $\sm_2(a,\dm)$ and selecting $t_1=-\frac{65}{128}$, we have $\sm_2(a,\dm)\approx 2.380804$ and
$\sm_\infty(a,\dm)\ge 2.205219$ using \eqref{sminfty} with $n=3$. Explicitly, for $t_1=-\frac{65}{128}$, we have
\begin{align*}
&a_1=\tfrac{1}{1024}\{-65, 63, 514, 514, 63,-65\}_{[-2,3]},\\
&a_2=\tfrac{1}{536738816}
\{
-4280965,
14764145,
90418427,
167467801,
167467801,
90418427,
14764145,
-4280965\}_{[-4,3]}.
\end{align*}
By \cref{thm:int,thm:qs} or \cref{cor:qs}, the $\dm$-refinable function $\phi\in \CH{2}$ must be $\frac{1}{6}$-interpolating and its $\infty$-step interpolatory $2$-mask quasi-stationary $2$-subdivision scheme using masks $\{a_1,a_2\}$ must be $\mathscr{C}^2$-convergent.
See \cref{fig:ex:M4} for the graph of the $\frac{1}{6}$-interpolating $\dm$-refinable function $\phi$.
\end{example}

\begin{figure}[htbp]
	\centering
\begin{subfigure}[b]{0.3\textwidth} \includegraphics[width=\textwidth,height=0.4\textwidth]
{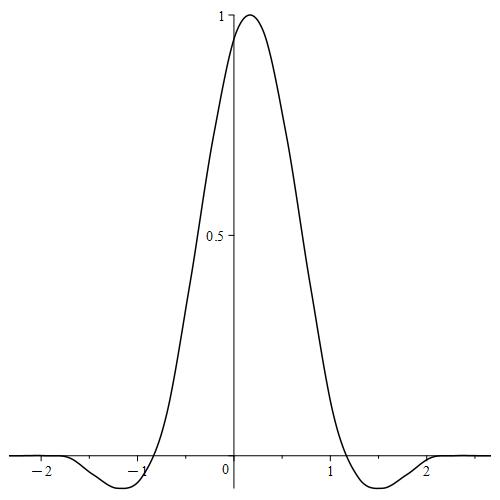} \caption{$\phi\in\CH{2}$}
	\end{subfigure}
\begin{subfigure}[b]{0.3\textwidth}	 \includegraphics[width=\textwidth,height=0.4\textwidth]
{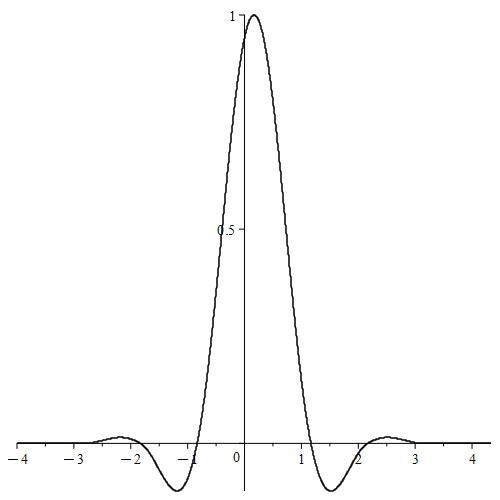}
		\caption{$\phi\in \CH{2}$}
\end{subfigure}
\begin{subfigure}[b]{0.3\textwidth}	 \includegraphics[width=\textwidth,height=0.4\textwidth]
{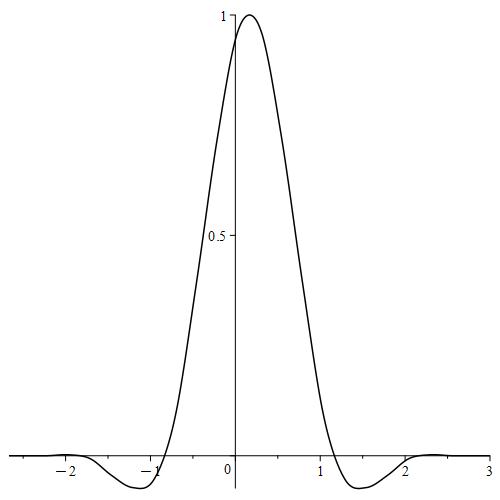}
		\caption{$\phi\in \CH{2}$}
\end{subfigure}
	 \begin{subfigure}[b]{0.3\textwidth} \includegraphics[width=\textwidth,height=0.4\textwidth]
{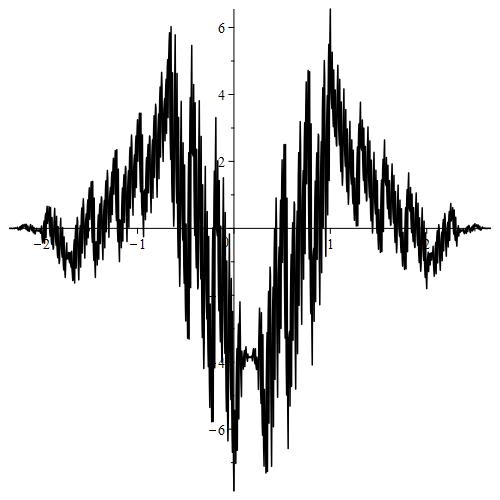} \caption{$\phi''$}
	\end{subfigure}
\begin{subfigure}[b]{0.3\textwidth}	 \includegraphics[width=\textwidth,height=0.4\textwidth]
{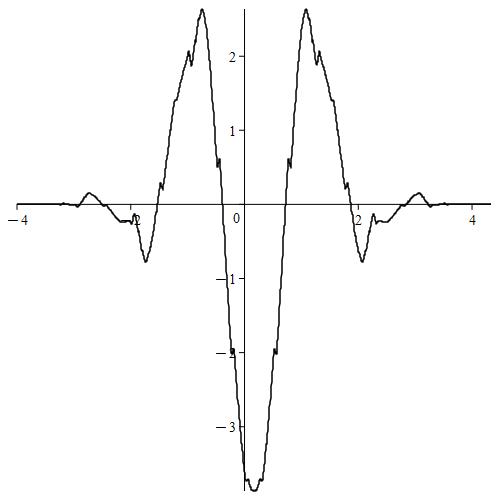}
		\caption{$\phi''$}
	\end{subfigure}
\begin{subfigure}[b]{0.3\textwidth}	 \includegraphics[width=\textwidth,height=0.4\textwidth]
{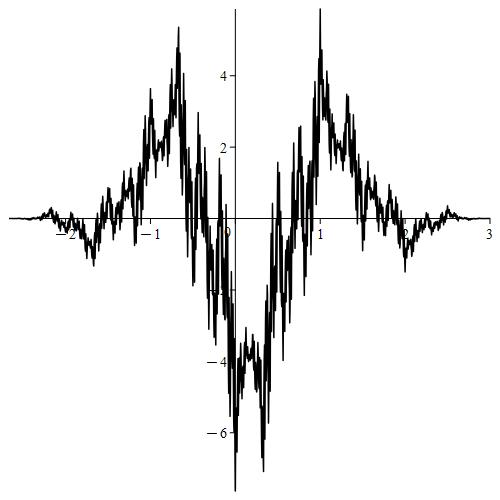}
		\caption{$\phi''$}
	\end{subfigure}
\caption{
(a) is the graph of the $\frac{1}{6}$-interpolating $4$-refinable function $\phi\in \CH{2}$ in \cref{ex:M4} with the mask $a$ satisfying $\sr(a,4)=3$.
(d) is the graph of $\phi''$ in (a).
(b) is the graph of the $\frac{1}{6}$-interpolating $4$-refinable function $\phi\in \CH{2}$ in \cref{ex:M4} with the mask $a$ satisfying $\sr(a,4)=5$.
(e) is the graph of $\phi''$ in (b).
(c) is the graph of the $\frac{1}{6}$-interpolating $4$-refinable function $\phi\in \CH{2}$ with $2$-mask quasi-stationary $2$-subdivision scheme
in \cref{ex:M2:M4} with $t_1=-\frac{65}{128}$. (f) is the graph of $\phi''$ in (c).
}\label{fig:ex:M4}
\end{figure}

\subsection{Application to subdivision curves}\label{subsec:app}

We first explain the rule of $s_a\in \R$ from the perspective of subdivision curves in CAGD. Let $\dm\in \N\bs\{1\}$ be a dilation factor and $a\in \lp{0}$ be a mask. Given an initial control polygonal $v=(v_x, v_y, v_z): \Z \rightarrow \R^3$ in the Euclidean space $\R^3$. That is, the initial control polygonal is given by the points $(v_x(k), v_y(k), v_z(k))$ in $\R^3$ for $k\in \Z$ which are connected in a natural way.
The subdivision scheme is applied componentwise to the vector sequence $v$ to produce finer and finer subdivided curves consisting of points $\{ (\sd_{a,\dm}^n v_x, \sd_{a,\dm}^n v_y, \sd_{a,\dm}^n v_z)\}_{n=1}^\infty$. As $n$ goes to $\infty$, one obtains a subdivision curve in $\R^3$. Obviously, no function expressions are explicitly given by the subdivision procedure to describe the limit subdivision curve in $\R^3$. To analyze the convergence and smoothness of the limit curve, it is necessary to find parametric expressions for the limit curve. The most natural way is to consider the subdivision scheme acting on a special initial control polygon in $\R^2$: $(v_0, v): \Z \rightarrow \R^2$ with $v_0(k):=k$ for $k\in \Z$ and $v$ being one of the component sequences $v_x, v_y, v_z$.
Then we obtain the following sequence of subdivision point data:
\be \label{v0:subdiv}
(\sd_{a,\dm}^n v_0, \sd_{a,\dm}^n v) \in \R^2,\qquad n\in \N.
\ee
Suppose now that the mask $a$ has at least order $2$ sum rules with respect to the dilation factor $\dm$, i.e., $\sr(a,\dm)\ge 2$. Define a special linear polynomial $\pp(x):=x$ for $x\in \R$. Then $v_0(k)=\pp(k)$ and $\pp'=1$ for all $k\in \Z$.
It is known in \cite[(2.20)]{han03} (also see \eqref{sr:poly} in this paper) that
\[
[\sd_{a,\dm} v_0](k)=[\sd_{a,\dm}\pp](k)
=\pp(\dm^{-1}k) \sum_{k\in \Z} a(k)-\dm^{-1} \pp'(\dm^{-1}k)\sum_{k\in \Z} k a(k)
=\dm^{-1} k-\dm^{-1} m_a,
\]
where we used $\sum_{k\in \Z} a(k)=1$ and $m_a:=\sum_{k\in \Z} k a(k)$.
Consequently, by induction we conclude from the above identity that
\be \label{sd:ma}
[\sd_{a,\dm}^n v_0](k)=\dm^{-n}k -\dm^{-n}m_a-\dm^{1-n}m_a-\cdots-\dm^{-1} m_a
=\dm^{-n} k-\frac{1-\dm^{-n}}{\dm-1} m_a.
\ee
Hence, the second component $[\sd_{a,\dm}^n v](k)$ for $k\in \Z$ in \eqref{v0:subdiv} is associated with the first component $[\sd_{a,\dm}^n v_0](k)$, which is just the value $\dm^{-n} k-\frac{1-\dm^{-n}}{\dm-1} m_a$.
Let $\eta_v$ be the limit function in \cref{def:sd} with $m=0$ for $\mathscr{C}^m$-convergence. Take $t:=\dm^{-n_0} k_0$ with $n_0\in \NN$ and $k_0\in \Z$. By $t=\dm^{-n} (\dm^{n-n_0} k_0)$ and $\dm^{n-n_0}k_0\in \Z$ for $n\ge n_0$,
we observe from \eqref{converg} with $j=0$ in \cref{def:sd} that
\[
\eta_v(t)=\lim_{n\to\infty} \eta_v(\dm^{-n}(\dm^{n-n_0} k_0))
=\lim_{n\to\infty} [\sd_{a,\dm}^n v](\dm^{n-n_0} k_0).
\]
Because the first component of the point
$([\sd_{a,\dm}^n v_0](\dm^{n-n_0}k_0), [\sd_{a,\dm}^n v](\dm^{n-n_0}k_0))$ in the subdivision data is $[\sd_{a,\dm}^n v_0](\dm^{n-n_0}k_0)$, which is equal to
$\dm^{-n}[ \dm^{n-n_0}k_0]-\frac{1-\dm^{-n}}{\dm-1} m_a=t-\frac{1-\dm^{-n}}{\dm-1} m_a$ by \eqref{sd:ma}, we conclude that its first component of the subdivided curve goes to
\[
\lim_{n\to\infty} [\sd_{a,\dm}^n v_0](\dm^{n-n_0} k_0)=
\lim_{n\to\infty} \left(t-\tfrac{1-\dm^{-n}}{\dm-1} m_a\right)=t-s_a \quad \mbox{with}\quad s_a:=\frac{m_a}{\dm-1}.
\]
In other words, we must associate the subdivision data $[\sd_{a,\dm}^n v](k)$ with the reference/parameter point $[\sd_{a,\dm}^n v_0](k)$ in \eqref{sd:ma} on the real line $\R$, i.e., the point $\dm^{-n}-\frac{1-\dm^{-n}}{\dm-1} m_a=\dm^{-n}(k-s_a)-s_a$.
Because $\cup_{n_0=0}^\infty \cup_{k_0\in \Z} \dm^{-n_0} k_0$ is dense in $\R$,
this naturally creates a parametric equation $x=t-s_a, y=\eta_v(t)$ in $\R^2$ for the subdivision curve with the initial polygon $\{(v_0(k), v(k))\}_{k\in \Z}$. By a change of variables, the limit two-dimensional subdivision curve is described by the parametric equation in $\R^2$:
\[
x=t,\quad y=\eta_v(s_a+t), \qquad t\in \R.
\]
Now for the subdivision curve in $\R^3$ generated from the initial control curve $\{(v_x(k), v_y(k), v_z(k))\}_{k\in \Z}$, a parametric equation for the limit subdivision curve is just given by $x=\eta_{v_x}(s_a+t), y=\eta_{v_y}(s_a+t), z=\eta_{v_z}(s_a+t)$ for $t\in \R$.
If the subdivision scheme interpolates the initial sequence $\{(v_0(k), v(k)) \}_{k\in \Z}$, then for $t=k\in \Z$, we must have $\eta(s_a+k)=v(k)$ for all $k\in \Z$. In particular, for $v=\td$, we have $\phi(s_a+k)=\eta_{\td}(k)=\td(k)$. This explains $s_a=\frac{m_a}{\dm-1}$ with $m_a=\sum_{k\in \Z} k a(k)$, which also agrees with \eqref{phi:sa} in \cref{prop:lpm}.
Moreover, if a mask $a$ has the symmetry $a(c_a-k)=a(k)$ for all $k\in \Z$ with $c_a\in \Z$, then we already explained after the proof of \cref{prop:lpm} that $s_a=\frac{m_a}{\dm-1}$ with $m_a=\sum_{k\in \Z} k a(k)$, which also agrees with our discussion of $s_a$ from the perspective in CAGD. Consequently, regardless whether $c_a$ is an even integer for primal subdivision schemes or $c_a$ is an odd integer for dual subdivision schemes, because we explained the role of $s_a$ above without requiring a symmetric mask $a$,
there are no essential differences between primal and dual subdivision schemes.
In summary, the subdivided data $[\sd_{a,\dm}^n v](k)$ for $k\in \Z$ must be naturally associated with the parameter point $[\sd_{a,\dm}^n v_0](k)$, i.e., the point $\dm^{-n}(k-s_a)-s_a$ with $s_a:=\frac{m_a}{\dm-1}$ and $m_a:=\sum_{k\in \Z} k a(k)$.

Finally, see \cref{fig:subcurve,fig:qssubcurve} for some examples of subdivision curves
by applying our constructed interpolatory (quasi)-stationary subdivision schemes to produce some simple subdivision curves.

\begin{figure}[bhtp]
	\centering
\begin{subfigure}[b]{0.24\textwidth} \includegraphics[width=0.9\textwidth,height=0.8\textwidth]
{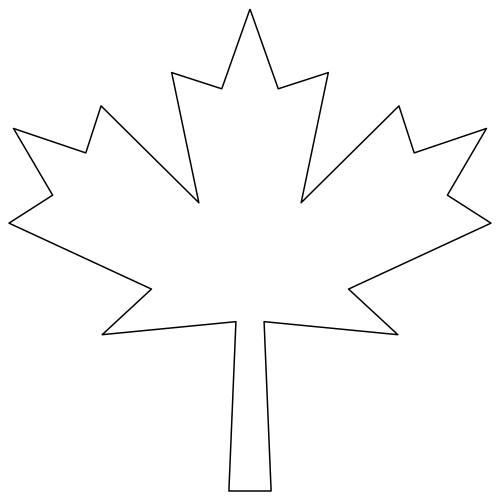} \caption{Initial Polygon}
	\end{subfigure}
	\begin{subfigure}[b]{0.24\textwidth} \includegraphics[width=0.9\textwidth,height=0.8\textwidth]
{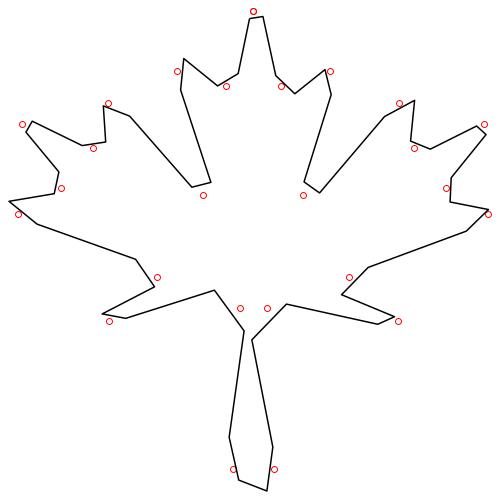} \caption{\cref{ex:M2}, L1}
	\end{subfigure} \begin{subfigure}[b]{0.24\textwidth}	 \includegraphics[width=0.9\textwidth,height=0.8\textwidth]
{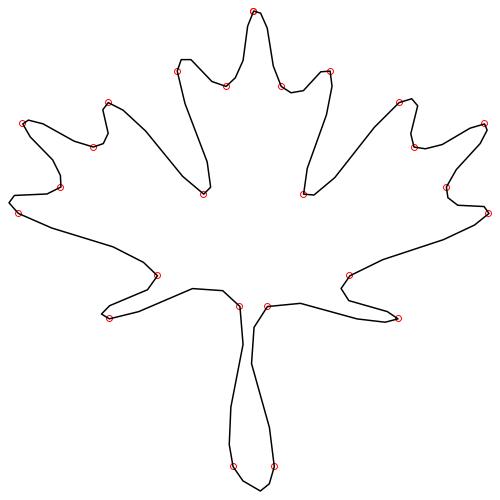}
		\caption{\cref{ex:M2}, L2}
\end{subfigure}
\begin{subfigure}[b]{0.24\textwidth}	 \includegraphics[width=0.9\textwidth,height=0.8\textwidth]
{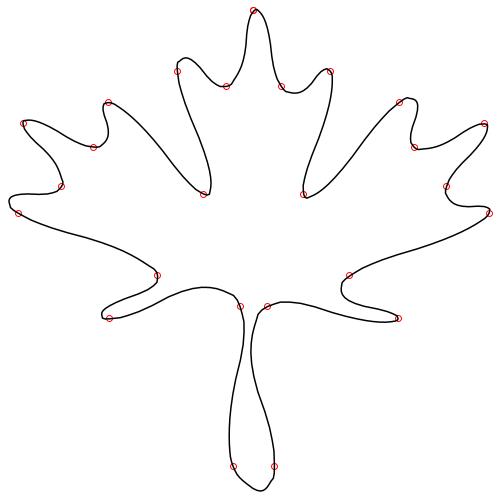}
		\caption{\cref{ex:M2}, L4}
	\end{subfigure}\\
\begin{subfigure}[b]{0.24\textwidth} \includegraphics[width=0.9\textwidth,height=0.8\textwidth]
{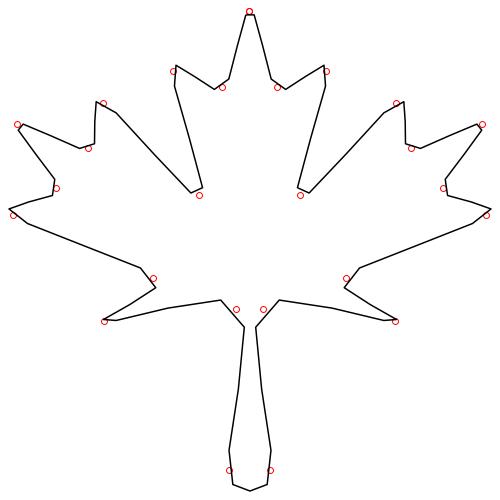} \caption{\cref{ex:M3}, L1}
	\end{subfigure}
	 \begin{subfigure}[b]{0.24\textwidth} \includegraphics[width=0.9\textwidth,height=0.8\textwidth]
{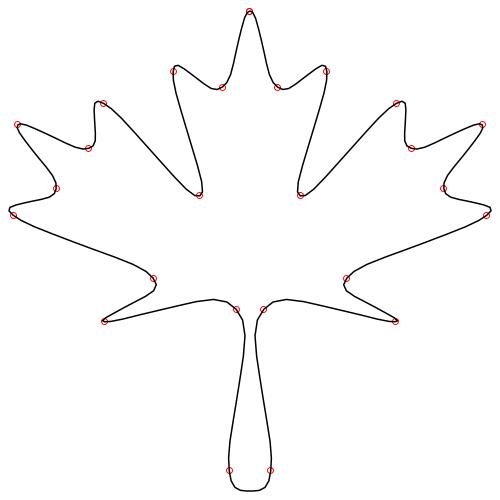} \caption{\cref{ex:M3}, L2}
	\end{subfigure} \begin{subfigure}[b]{0.24\textwidth}	 \includegraphics[width=0.9\textwidth,height=0.8\textwidth]
{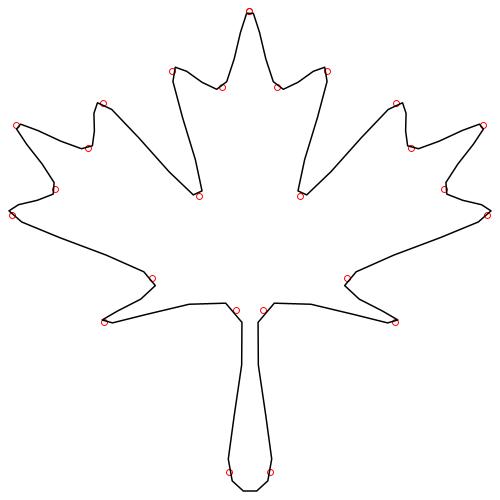}
		\caption{\cref{ex:M4}, L1}
\end{subfigure}
\begin{subfigure}[b]{0.24\textwidth}	 \includegraphics[width=0.9\textwidth,height=0.8\textwidth]
{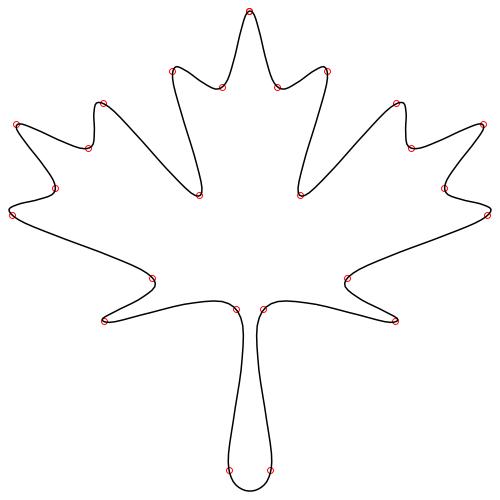}
		\caption{\cref{ex:M4}, L2}
	\end{subfigure}
\caption{
(a) is the initial control polygon in $\R^2$. The red dots in (b)-(h) indicate the vertices of the initial control polygon in (a) for illustrating the $n_s$-step interpolation property.
(b)--(d) are subdivision curves with levels $1,2,4$ using the $\mathscr{C}^1$-convergent $2$-step interpolatory $2$-subdivision scheme  with mask $a$  in \eqref{exM2C1} of \cref{ex:M2} with $t=-\frac{1}{18}$.
(e)--(f) are subdivision curves with levels $1,2$ using the $\mathscr{C}^2$-convergent $2$-step interpolatory $3$-subdivision scheme with mask $a$ in \cref{ex:M3} with $\sr(a,3)=3$ and $t_1=\frac{1}{432}$.
(g)--(h) are subdivision curves with levels $1,2$ using the $\mathscr{C}^2$-convergent $\infty$-step interpolatory $4$-subdivision scheme with mask $a$ in \cref{ex:M4} with $\sr(a,4)=3$ and $t_1=-\frac{1}{832}$.
}\label{fig:subcurve}
\end{figure}

\begin{figure}[htbp]
	\centering
\begin{subfigure}[b]{0.24\textwidth} \includegraphics[width=0.9\textwidth,height=0.8\textwidth]
{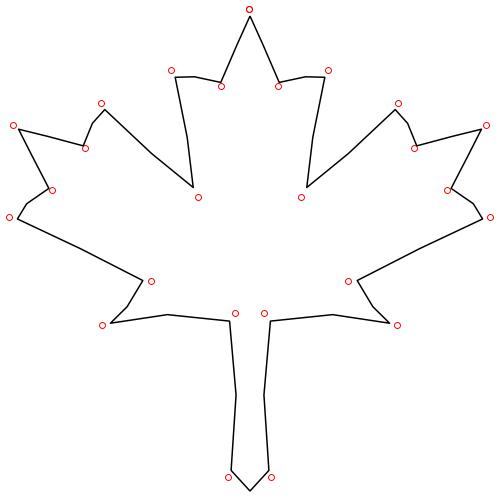} \caption{\cref{ex2}, L1}
	\end{subfigure}
	 \begin{subfigure}[b]{0.24\textwidth} \includegraphics[width=0.9\textwidth,height=0.8\textwidth]
{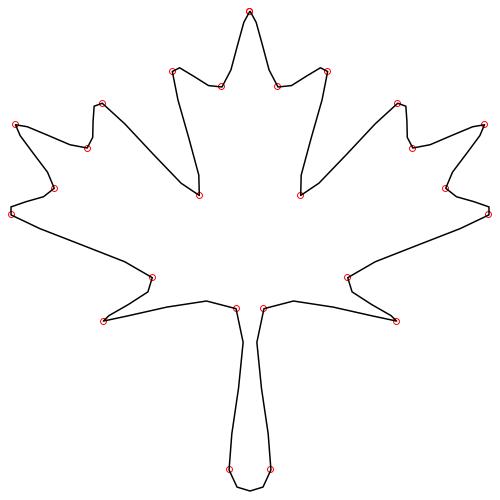} \caption{\cref{ex2}, L2}
	\end{subfigure} \begin{subfigure}[b]{0.24\textwidth}	 \includegraphics[width=0.9\textwidth,height=0.8\textwidth]
{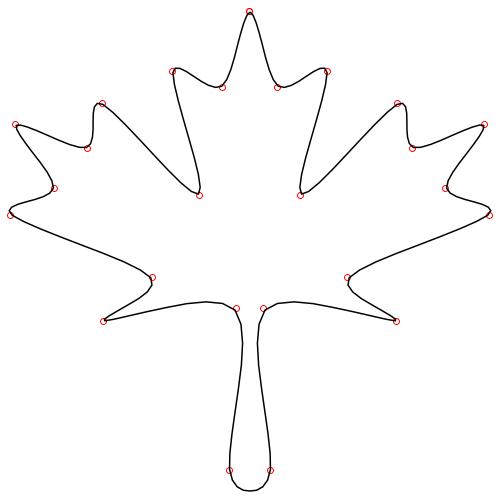}
		\caption{\cref{ex2}, L3}
\end{subfigure}
\begin{subfigure}[b]{0.24\textwidth}	 \includegraphics[width=0.9\textwidth,height=0.8\textwidth]
{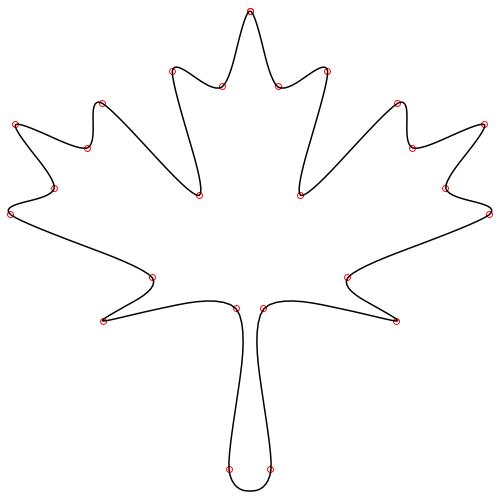}
		\caption{\cref{ex2}, L4}
\end{subfigure}\\
\begin{subfigure}[b]{0.24\textwidth} \includegraphics[width=0.9\textwidth,height=0.8\textwidth]
{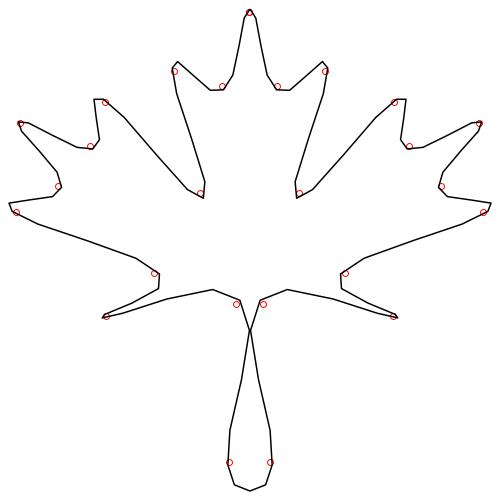} \caption{\cref{ex3}, L2}
	\end{subfigure}
	 \begin{subfigure}[b]{0.24\textwidth} \includegraphics[width=0.9\textwidth,height=0.8\textwidth]
{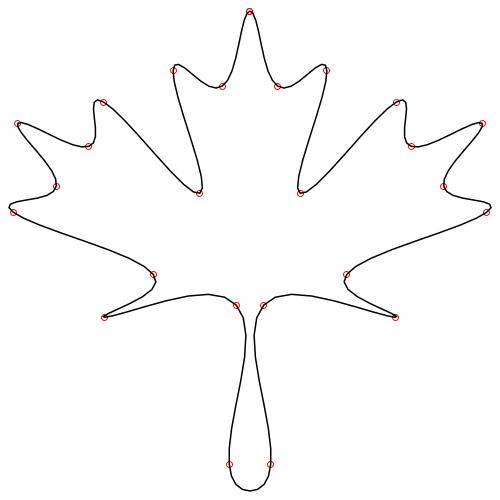} \caption{\cref{ex3}, L3}
	\end{subfigure} \begin{subfigure}[b]{0.24\textwidth}	 \includegraphics[width=0.9\textwidth,height=0.8\textwidth]
{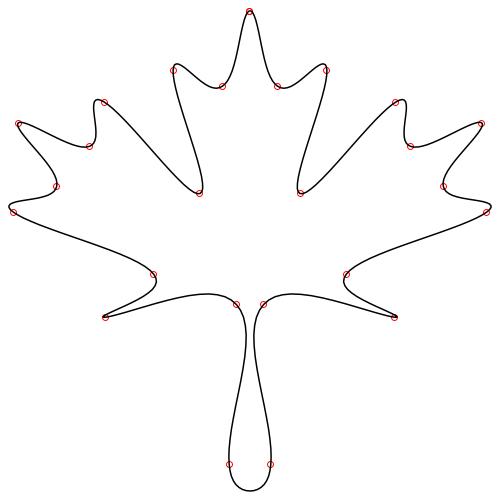}
		\caption{\cref{ex3}, L4}
\end{subfigure}
\begin{subfigure}[b]{0.24\textwidth}	 \includegraphics[width=0.9\textwidth,height=0.8\textwidth]
{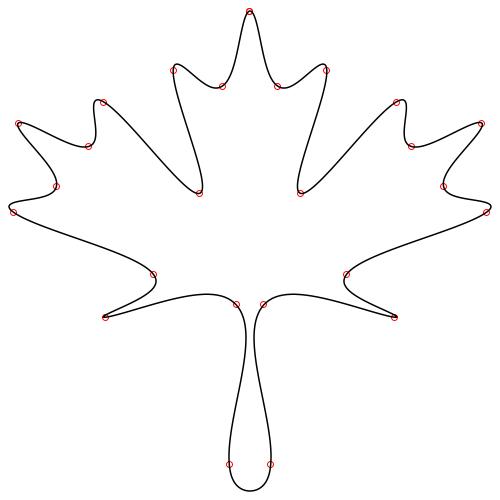}
		\caption{\cref{ex3}, L5}
	\end{subfigure}\\
\begin{subfigure}[b]{0.24\textwidth} \includegraphics[width=0.9\textwidth,height=0.8\textwidth]
{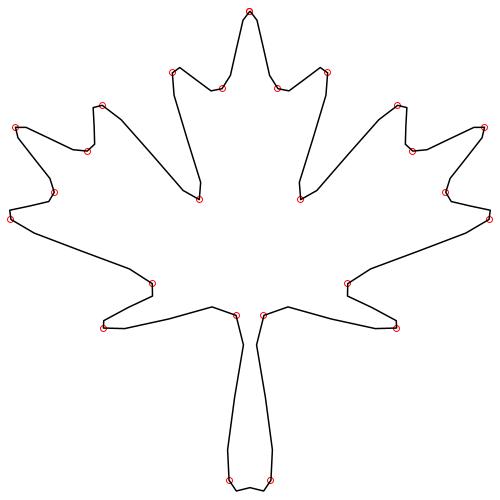} \caption{\cref{ex1}, L2}
	\end{subfigure}
	\begin{subfigure}[b]{0.24\textwidth} \includegraphics[width=0.9\textwidth,height=0.8\textwidth]
{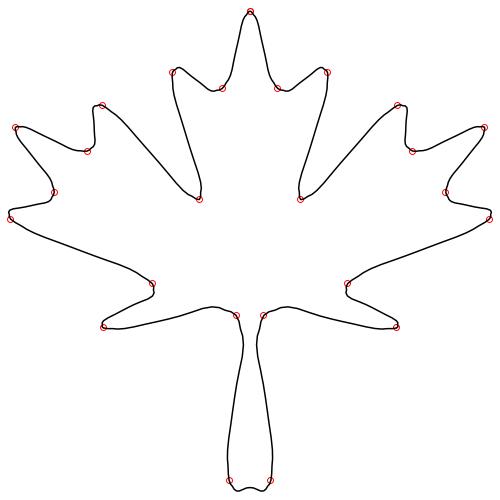} \caption{\cref{ex1}, L4}
	\end{subfigure} \begin{subfigure}[b]{0.24\textwidth}	 \includegraphics[width=0.9\textwidth,height=0.8\textwidth]
{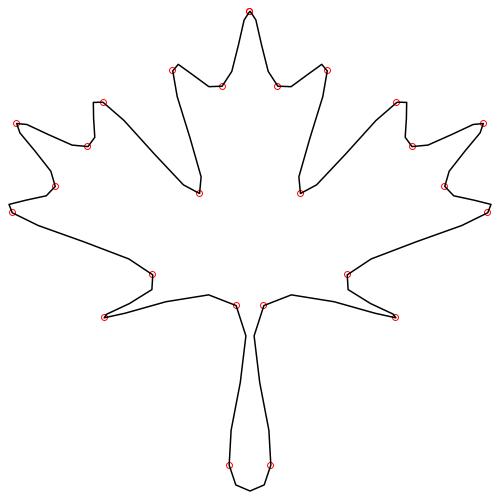}
		\caption{\cref{ex2}, L2}
\end{subfigure}
\begin{subfigure}[b]{0.24\textwidth}	 \includegraphics[width=0.9\textwidth,height=0.8\textwidth]
{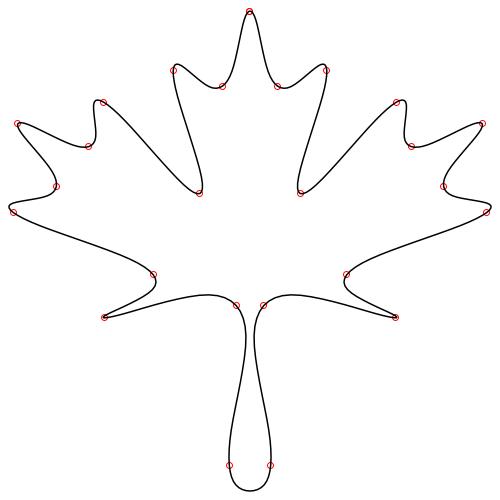}
		\caption{\cref{ex2}, L4}
	\end{subfigure}
\caption{
(a)--(d) are subdivision curves with levels $1,\ldots,4$ using the $\mathscr{C}^2$-convergent $2$-step interpolatory $2$-mask quasi-stationary $2$-subdivision scheme in \cref{ex2} with masks $\{a_1,a_2\}$ (having two-ring stencils) in \eqref{a1a2} using parameters in \eqref{tval1}.
(e)--(h) are subdivision curves with levels $2,\ldots,5$ using the $\mathscr{C}^3$-convergent $3$-step interpolatory $3$-mask quasi-stationary $2$-subdivision scheme in \cref{ex3} with masks $\{a_1,a_2, a_3\}$ having two-ring stencils.
(i)--(j) are subdivision curves with levels $2,4$ using the $\mathscr{C}^1$-convergent $2$-step interpolatory $2$-subdivision scheme in \cref{ex1} with mask $\{a_1, a_2\}$ having one-ring stencils.
(k)--(1) are subdivision curves with levels $2,4$ using the $\mathscr{C}^2$-convergent $2$-step interpolatory $2$-subdivision scheme with mask $a$ in \cref{ex2} using parameters in \eqref{tval2} and having two-ring stencils.
}\label{fig:qssubcurve}
\end{figure}

\section{Proof of \cref{thm:int}}
\label{sec:proof}

In this section we shall prove \cref{thm:int}, whose proof is critically built on two ingredients: the eigenvalues of shifted transition operators $\mtz_{a,\dm,\gamma}$ below and the structure of the number $s_a$ satisfying the condition \eqref{cond:sa}.
We already addressed the roles of $s_a$ satisfying the condition \eqref{cond:sa} in Subsection~\ref{subsec:sa}.
To prove \cref{thm:int}, we further need some auxiliary results about the structure of shifted transition operators $\mtz_{a,\dm,\gamma}$.

For $a \in \lp{0}$ and $\gamma\in \Z$, we define a shifted transition operator
$\mtz_{a,\dm,\gamma}: \lp{0}\rightarrow \lp{0}$ by
\be \label{tz}
[\mtz_{a,\dm,\gamma} v](n):=\dm \sum_{k\in\Z} a(k) v(\gamma+\dm n-k),\qquad n\in \Z, v\in \lp{0}.
\ee
We first study the eigenvalues of $\mtz_{a,\dm,\gamma}$ acting on the linear space $\lp{0}$. By $\mbox{spec}(\mtz_{a,\dm,\gamma})$ we denotes the multiset of all the eigenvalues of $\mtz_{a,\dm,\gamma}$ counting the multiplicity of nonzero eigenvalues of $\mtz_{a,\dm,\gamma}$.
We now study some properties of $\mtz_{a,\dm,\gamma}$ by generalizing the corresponding results in \cite{hanbook,hj98} for $\dm=2$.

\begin{lemma}\label{lem:tz}
Let $\dm\in \N\bs\{1\}$ be a dilation factor.
Let $a\in \lp{0}$ and $\gamma\in \Z$.
\begin{enumerate}
\item[(1)] $\mtz_{a,\dm,\gamma} \ell(K_{a,\gamma})\subseteq \ell(K_{a,\gamma})$, where $\ell(K_{a,\gamma})$ is the space of all sequences $v\in \lp{0}$ such that $v$ is supported inside $\Z\cap K_{a,\gamma}$ with $K_{a,\gamma}:=(\dm-1)^{-1}[\fs(a)-\gamma]$.
\item[(2)] If $\mtz_{a,\dm,\gamma} v=\gl v$ for some $\gl\in \C\bs \{0\}$ and $v\in \lp{0}$, then $v\in \ell(K_{a,\gamma})$.
\item[(3)] $\mtz_{a,\dm,\gamma+\dm j}(v(\cdot+m))=[\mtz_{a,\dm,\gamma+m} v](\cdot+j)$ for all $j,m\in \Z$. Hence, $\mbox{spec}(\mtz_{a,\dm,\gamma+(\dm-1)j})=
\mbox{spec}(\mtz_{a,\dm,\gamma})$ for all $j\in \Z$.
\end{enumerate}
\end{lemma}

\bp (1) Recall that $\fs(a)$ is the smallest interval such that $a(k)=0$ for all $k\in \Z\bs \fs(a)$.
Let $v\in \ell(K_{a,\gamma})$. Then $[\mtz_{a,\dm,\gamma} v](n)\ne 0$ only if $a(k)v(\gamma+Mn-k)\ne 0$ for some $k\in \Z$, from which we have $k\in \fs(a)$ and $\gamma+\dm n-k\in K_{a,\gamma}$, which implies $n\in \dm^{-1}[\fs(a)+K_{a,\gamma}-\gamma]$. By the definition of $K_{a,\gamma}$, we observe
\begin{align*}
\dm^{-1}[\fs(a)+K_{a,\gamma}-\gamma]
&=\dm^{-1}[\fs(a)+(\dm-1)^{-1}\fs(a)-(\dm-1)^{-1}\gamma-\gamma]\\
&=(\dm-1)^{-1} \fs(a)-(\dm-1)^{-1} \gamma=K_{a,\gamma}.
\end{align*}
This proves that $v$ is supported inside $K_{a,\gamma}$. Hence, we proved item (1).

(2) If $v$ is identically zero, then the claim is obviously true. So, we assume that $v$ is not identically zero. If $v(n)\ne 0$ for some $n\in \Z$, then $\gl^{-1}[\mtz_{a,\dm,\gamma} v](n)=v(n)\ne 0$. By the above same argument, we must have
$n\in \dm^{-1} \fs(a)-\dm^{-1}\gamma+\dm^{-1}\fs(v)$ from which we must have
\[
\fs(v)\subseteq \dm^{-1} \fs(v)+\dm^{-1}[\fs(a)-\gamma].
\]
Recursively applying the above relation (e.g., see \cite{hj98}), we conclude that
\[
\fs(v)\subseteq \sum_{j=1}^\infty \dm^{-j}[\fs(a)-\gamma]=(\dm-1)^{-1}[\fs(a)-\gamma]=K_{a,\gamma}.
\]
This proves item (2).

(3) Note that
\[
\mtz_{a,\dm,\gamma+\dm j} (v(\cdot+m))=
\dm \sum_{k\in \Z} a(k) v(\gamma+\dm j+\dm \cdot-k+m)=
\dm \sum_{k\in \Z} a(k) v(\gamma+m+\dm(\cdot+j)-k)=[\mtz_{a,\dm,\gamma+m} v](\cdot+j).
\]
Hence, considering $m=j$, we have $\mbox{spec}(\mtz_{a,\dm,\gamma+\dm j})=\mbox{spec}(\mtz_{a,\dm,\gamma+j})$ for all $j,\gamma\in \Z$, from which we further have  $\mbox{spec}(\mtz_{a,\dm,\gamma+(\dm-1)j})=
\mbox{spec}(\mtz_{a,\dm,\gamma})$.
This proves item (3).
\ep

We now study the eigenvalues of shifted transition operators $\mtz_{a,\dm,\gamma}$ and prove the identity \eqref{smabJ}.

\begin{lemma} \label{lem:eig}
Let $\dm\in \N\bs\{1\}$ be a dilation factor.
Let $a\in \lp{0}$ and $J\in \NN$ such that $\tilde{\pa}(z)=(1+z+\cdots+z^{\dm-1})^J \tilde{\pb}(z)$ for some $b\in \lp{0}$. Define $A_n:=\dm^{-n} \sd_{a,\dm}^n \td$ and $B_n:=\dm^{-n} \sd_{b,\dm}^n \td$, i.e.,
\be \label{ABn:mask}
\widetilde{\mathbf{A}_n}(z):=\tilde{\pa}(z^{\dm^{n-1}})\cdots \tilde{\pa}(z^\dm)\tilde{\pa}(z),\qquad
\widetilde{\mathbf{B}_n}(z):=\tilde{\pb}(z^{\dm^{n-1}})\cdots \tilde{\pb}(z^\dm)\tilde{\pb}(z),\qquad n\in \N.
\ee
Let $u\in \lp{0}$ and take $N\in \N$ such that all the sequences $a,b,u$ are supported inside $(-N,N)$. Then
\be \label{ABn}
2^{J(1/p-1)} N^{-J} \|B_n*u\|_{\lp{p}}\le \|\nabla^J (A_n *u)\|_{\lp{p}}\le 2^J \|B_n*u\|_{\lp{p}},\qquad \forall\; n\in \N, 1\le p\le \infty
\ee
and
\be \label{Bn:bound}
\liminf_{n\to\infty} \|B_n*u\|_{\lp{p}}^{1/n}\ge \dm^{\frac{1}{p}-1-J}|\tilde{\pa}(1)|,\qquad \forall\; u\in \lp{0} \mbox{ such that } \sum_{k\in \Z} u(k)\ne 0,
\ee
where  $\tilde{\pa}(1):=\sum_{k\in \Z} a(k)$. Moreover,
\be \label{spec:ab}
\mbox{spec}(\tz_{a,\dm,\gamma})=
\{\tilde{\pa}(1), \dm^{-1}\tilde{\pa}(1),\ldots,\dm^{1-J}\tilde{\pa}(1)\}\cup \mbox{spec}(\tz_{b,\dm,\gamma})
\ee
and it follows directly from \eqref{ABn} and \eqref{Bn:bound} that
\be \label{sm:ab}
\limsup_{n\to\infty}
\|\nabla^J \sd_{a,\dm}^n\td \|_{\lp{p}}^{1/n}
=\limsup_{n\to\infty} \|\sd_{b,\dm}^n \td \|_{\lp{p}}^{1/n}=\dm \limsup_{n\to\infty} \|B_n\|_{\lp{p}}^{1/n}\ge \dm^{\frac{1}{p}-J} |\tilde{\pa}(1)|.
\ee
\end{lemma}

\bp By $(1-z)(1+z+\cdots+z^{\dm-1})=1-z^\dm$, the symbol of $\nabla^J(A_n*u)$ is
\be \label{ABn:symbol}
(1-z)^J \widetilde{\mathbf{A}_n}(z)\tilde{\mathbf{u}}(z)
=(1-z^{\dm^n})^J\widetilde{\mathbf{B}_n}(z)\tilde{\mathbf{u}}(z)
=\sum_{j=0}^J \frac{J!}{j!(J-j)!} (-1)^j z^{\dm^n j}
\widetilde{\mathbf{B}_n}(z)\tilde{\mathbf{u}}(z).
\ee
Therefore, we deduce from the above identity in \eqref{ABn:symbol} that
\[
\|\nabla^J (A_n*u)\|_{\lp{p}}\le \sum_{j=0}^J \frac{J!}{j!(J-j)!} \|B_n*u\|_{\lp{p}}=2^J \|B_n*u\|_{\lp{p}}.
\]
This proves the upper bound in \eqref{ABn}.
Define a sequence $w_n\in \lp{0}$ by $\widetilde{\mathbf{w}_n}(z):=(1+z^{\dm^n}+(z^{\dm^n})^2+\cdots+(z^{\dm^n})^{2N-1})^J$. On the other hand, to prove the lower bound in \eqref{ABn}, we have
\begin{align*}
\widetilde{\mathbf{w}_n}(z)
(1-z)^J \widetilde{\mathbf{A}_n}(z)\tilde{\mathbf{u}}(z)
&=\widetilde{\mathbf{w}_n}(z)
(1-z^{\dm^n})^J \widetilde{\mathbf{B}_n}(z) \tilde{\mathbf{u}}(z)\\
&=(1-z^{2N \dm^n})^J \widetilde{\mathbf{B}_n}(z) \tilde{\mathbf{u}}(z)
=\sum_{j=0}^J \frac{J!}{j!(J-j)!} (-1)^j z^{2N j \dm^n} \widetilde{\mathbf{B}_n}(z) \tilde{\mathbf{u}}(z).
\end{align*}
Because all the sequences $a,b,u$ are supported inside $[1-N,N-1]$, we observe that the sequence $B_n*u$ is supported inside $[1-N\dm^n,N\dm^n-1]$. Hence, the sequences having the symbols $z^{2N j \dm^n} \widetilde{\mathbf{B}_n}(z) \tilde{\mathbf{u}}(z)$ must have mutually disjoint supports for $j=0,\ldots,J$. Then we deduce from the above identity that
\[
2^{J/p} \|B_n*u\|_{\lp{p}}
=\|w_n*(\nabla^J(A_n*u))\|_{\lp{p}}
\le \|\nabla^J(A_n*u)\|_{\lp{p}}\|w_n\|_{\lp{1}}
=(2N)^J \|\nabla^J(A_n*u)\|_{\lp{p}},
\]
from which we proved the lower bound in \eqref{ABn}.

We now prove \eqref{Bn:bound}. Let $1\le p'\le \infty$ such that $\frac{1}{p}+\frac{1}{p'}=1$.
Noting
\[
[\tilde{\pb}(1)]^n\tilde{\pu}(1)=
\widetilde{\mathbf{B_n}}(1)\tilde{\mathbf{u}}(1)=\widetilde{\mathbf{B_n*u}}(1)=
\sum_{k=1-N\dm^n}^{N\dm^n-1} [B_n*u](k)
\]
and using the H\"older's inequality, we have
\begin{align*}
|\tilde{\pb}(1)|^n|\tilde{\pu}(1)|
&\le \sum_{k=1-N\dm^n}^{N\dm^{n}-1}
|[B_n*u](k)|\le
\left(\sum_{k=1-N\dm^n}^{N\dm^{n}-1}
|[B_n*u](k)|^p\right)^{1/p}
\left(\sum_{k=1-N\dm^n}^{N\dm^{n}-1}
1\right)^{1/p'}\\
&=\|B_n*u\|_{\lp{p}} (2N \dm^n-1)^{1/p'}\le \|B_n*u\|_{\lp{p}} (2N \dm^n)^{1/p'}.
\end{align*}
Because $\tilde{\pu}(1)=\sum_{k\in \Z} u(k) \ne 0$, we deduce from the above identity that
\[
|\tilde{\pb}(1)| \le \left(\liminf_{n\to\infty}\|B_n*u\|_{\lp{p}}^{1/n}\right)
\left(\lim_{n\to\infty} (2N\dm^n)^{\frac{1}{n p'}}\right)=
\dm^{1/p'} \liminf_{n\to\infty}\|B_n*u\|_{\lp{p}}^{1/n},
\]
from which we have \eqref{Bn:bound} due to $1/p'=1-1/p$ and $\tilde{\pb}(1)=\dm^{-J}\tilde{\pa}(1)$.

For $j=0,\ldots,J$, we define $\mathscr{V}_j:=\{\nabla^j v \setsp v\in \lp{0}\}$ and define $b_j\in \lp{0}$ by $\widetilde{\pb_j}(z)=(1+z+\cdots+z^{\dm-1})^{J-j} \tilde{\pb}(z)$. Note that $\tilde{\pa}(z)=(1+z+\cdots+z^{\dm-1})^j \widetilde{\pb_j}(z)$ for all $j=0,\ldots, J$. In particular, $b_0=a$ and $b_J=b$. Note that the symbol of $a*(\nabla^j v)$ is
\[
(1-z)^j \tilde{\pa}(z)\tilde{\pv}(z)=(1-z^{\dm})^j \widetilde{\pb_j}(z)\tilde{\pv}(z)=\sum_{k=0}^j \frac{j!}{k!(j-k)!} (-1)^k z^{\dm k} \widetilde{\pb_j}(z)\tilde{\pv}(z).
\]
Therefore, by the definition of $\mtz_{a,\dm,\gamma}$ in \eqref{tz}, we have
\begin{align*}
\mtz_{a,\dm,\gamma}(\nabla^j v)
&=
\dm [a*(\nabla^j v)](\gamma+\dm\cdot)=
\sum_{k=0}^j \frac{j!}{k!(j-k)!} (-1)^k \dm [b_j*v](\gamma+\dm (\cdot-k))\\
&=\sum_{k=0}^j \frac{j!}{k!(j-k)!} (-1)^k  [\mtz_{b_j,\dm,\gamma} v](\cdot-k)
=\nabla^j (\mtz_{b_j,\dm,\gamma} v).
\end{align*}
That is, we proved
\be \label{tz:ab:j}
\mtz_{a,\dm,\gamma}(\nabla^j v)=
\nabla^j (\mtz_{b_j,\dm,\gamma}),\qquad \forall\; v\in \lp{0}\quad \mbox{and}\quad j=0,\ldots,J.
\ee
We conclude from \eqref{tz:ab:j} that $\mtz_{a,\dm,\gamma} \mathscr{V}_j\subseteq \mathscr{V}_j$
and $\mbox{spec}(\mtz_{a,\dm,\gamma}|_{\mathscr{V}_j})
=\mbox{spec}(\mtz_{b_j,\dm,\gamma})$ for all $j=0,\ldots,J$.

For $j=0,\ldots, J-1$, due to $\sr(a,\dm)\ge J$, note that $b_j$ must have at lease one sum rule, i.e., $\sum_{k\in \Z} b_j(\gamma+\dm k)=\frac{1}{\dm}\sum_{k\in \Z} b_j(k)=
\dm^{-1-j}\tilde{\pa}(1)$ for all $\gamma\in \Z$. Therefore,
\[
\sum_{k\in \Z} [\mtz_{b_j,\dm,\gamma}\td](k)
=\dm \sum_{k\in \Z} [b_j*\td](\gamma+\dm k)=
\dm \sum_{k\in \Z} b_j(\gamma+\dm k)=\dm^{-j}
\tilde{\pa}(1),
\]
from which we obtain $\mtz_{b_j,\dm,\gamma}\td-\dm^{-j}\tilde{\pa}(1) \td= \nabla w$ for some $w\in \lp{0}$. Hence, we deduce from \eqref{tz:ab:j} that
\begin{align*}
\mtz_{a,\dm,\gamma}(\nabla^j \td )-\dm^{-j}\tilde{\pa}(1)\nabla^{j}\td &=\nabla^j (\mtz_{b_j,\dm,\gamma}\td)-
\dm^{-j} \tilde{\pa}(1) \nabla^j \td\\
&=
\nabla^j\left( \mtz_{b_j,\dm,\gamma}\td-\dm^{-j} \tilde{\pa}(1)\td\right)
=
\nabla^j \nabla w=\nabla^{j+1} w\in \mathscr{V}_{j+1}.
\end{align*}
This proves $\mtz_{a,\dm,\gamma}(\nabla^j \td )-\dm^{-j}\tilde{\pa}(1)\nabla^{j}\td\in \mathscr{V}_{j+1}$ for all  $j=0,\ldots,J-1$. Note that $\mathscr{V}_j/\mathscr{V}_{j+1}$ is a one-dimensional space and is spanned by $\nabla^j \td$.
Consequently, we conclude that $\mbox{spec}(\mtz_{a,\dm,\gamma}|_{\mathscr{V}_j/\mathscr{V}_{j+1}})
=\{\dm^{-j}\tilde{\pa}(1)\}$ for all $j=0,\ldots,J-1$.
Since we proved $\mbox{spec}(\mtz_{a,\dm,\gamma}|_{\mathscr{V}_J})
=\mbox{spec}(\mtz_{b_J,\dm,\gamma})$ and $b_J=b$, we conclude that  \eqref{spec:ab} holds.
\ep

We now prove the major auxiliary result on the special eigenvalues of the shifted transition operators
$\mtz_{a,\dm,\gamma}$, which plays a key role in the proof of \cref{thm:int}.

\begin{theorem}\label{thm:eig}
Let $\dm\in \N\bs\{1\}$ be a dilation factor and $a\in \lp{0}$ be a finitely supported mask satisfying $\sum_{k \in \Z} a(k)=1$. Then
\be \label{srsm}
\sr(a,\dm)\ge \sm_p(a,\dm) \qquad \forall\; 1\le p\le \infty.
\ee
If $\sm_p(a,\dm)>\frac{1}{p}+m$ with $m\in \NN$ for some $1\le p\le \infty$, then
$|\lambda|<\dm^{-m}$ for all $\gl\in \mbox{spec}(\tz_{a,\dm,\gamma})$ but $\gl\not\in  \{1, \dm^{-1},\ldots, \dm^{-m}\}$, and
each $\dm^{-j}$ for $j=0,\ldots, m$ must be a simple eigenvalue of the transition operator $\mtz_{a,\dm,\gamma}$ in \eqref{tz} acting on $\lp{0}$ for all $\gamma\in \Z$.
\end{theorem}

\bp
Let $J:=\sr(a,\dm)$. Write
$\tilde{\pa}(z)=(1+z+\cdots+z^{\dm-1})^J \tilde{\pb}(z)$ for a unique sequence $b\in \lp{0}$.
By \eqref{sm:ab} in \cref{lem:eig} and $\tilde{\pa}(1)=1$, we deduce from the definition of $\sm_p(a,\dm)$ in \eqref{sm} that
\be \label{an:b:sm}
\dm^{\frac{1}{p}-J}\le \limsup_{n\to\infty}\|\sd_{b,\dm}^n \td\|_{\lp{p}}^{1/n}=
\limsup_{n\to\infty}\|\nabla^J\sd_{a,\dm}^n \td\|_{\lp{p}}^{1/n}=\dm^{\frac{1}{p}-\sm_p(a,\dm)},
\ee
from which we must have $\sm_p(a,\dm)\le J=\sr(a,\dm)$. This proves \eqref{srsm}.

We now prove the claims under the assumption $\sm_p(a,\dm)>\frac{1}{p}+m$. Using the identity \eqref{spec:ab} which links the eigenvalues of $\tz_{a,\dm,\gamma}$ with those of $\tz_{b,\dm,\gamma}$,
the key ingredient of the proof here is to show that the assumption $\sm_p(a,\dm)>\frac{1}{p}+m$ will force other non-special eigenvalues of $\tz_{a,\dm,\gamma}$ to have modulus less than $\dm^{-m}$.
Define $B_n$ as in \eqref{ABn:mask}.
We shall use induction to prove that
\be \label{tz:bn}
\mtz_{b,\dm,\gamma}^n v=\dm^n [B_n*v](\gamma+\dm \gamma+\cdots + \dm^{n-1} \gamma+\dm^n \cdot),\qquad n\in \N, v\in \lp{0}.
\ee
By the definition of $\mtz_{b,\dm,\gamma}$ in \eqref{tz}, it is easy to see that \eqref{tz:bn} holds with $n=1$ due to $B_1=b$. Suppose that the claim holds for $n-1$ with $n\ge 2$. By the induction hypothesis, we have
\begin{align*}
\mtz_{b,\dm,\gamma}^n v
&=\mtz_{b,\dm,\gamma} [\mtz_{b,\dm,\gamma}^{n-1} v]
=\dm^{n-1}\mtz_{b,\dm,\gamma} [(B_{n-1}*v)(\gamma+\dm \gamma+\cdots+\dm^{n-2}\gamma+\dm^{n-1}\cdot)]\\
&=
\dm^n \sum_{k\in \Z} b(k) (B_{n-1}*v)(\gamma+\dm \gamma+\cdots \dm^{n-2}\gamma+\dm^{n-1}(\gamma+\dm\cdot-k))\\
&=\dm^n \sum_{k\in \Z}\sum_{j\in \Z} b(k) B_{n-1}(j) v(\gamma+\dm \gamma+\cdots+\dm^{n-1}\gamma+\dm^{n}\cdot-(\dm^{n-1}k+j)).
\end{align*}
Therefore, using the above identity and the definition of $\sd_{b,\dm}$ in \eqref{sd}, we have
\begin{align*}
\mtz_{b,\dm,\gamma}^n v
&= \dm^n \sum_{j\in \Z} \sum_{k\in \Z} b(k) B_{n-1}(j-\dm^{n-1}k) v(\gamma+\dm \gamma+\cdots+\dm^{n-1}\gamma+\dm^{n}\cdot-j)\\
&=\dm \sum_{j\in \Z}[\sd_{B_{n-1},\dm^{n-1}}b](j)v(\gamma+\dm \gamma+\cdots+\dm^{n-1}\gamma+\dm^{n}\cdot-j)\\
&=\dm^n \sum_{j\in \Z} B_n(j)v(\gamma+\dm \gamma+\cdots+\dm^{n-1}\gamma+\dm^{n}\cdot-j)\\
&=\dm^n [B_n*v](\gamma+\dm \gamma+\cdots+\dm^{n-1}\gamma+\dm^{n}\cdot),
\end{align*}
where we used the fact that $\sd_{B_{n-1},\dm^{n-1}} b=\dm^{n-1} B_n$.
This proves \eqref{tz:bn} by induction on $n\in \N$.

Suppose $\mtz_{b,\dm,\gamma}v=\gl v$ for some $\gl\in \C\bs\{0\}$ and $v\in \lp{0}\bs\{0\}$. Then we deduce from \eqref{tz:bn} that
\[
\gl^n v=\mtz_{b,\dm,\gamma}^n v
=\dm^n [B_n*v](\gamma+\dm \gamma+\cdots +\dm^{n-1} \gamma+\dm^n \cdot).
\]
Consequently, we have
\[
|\gl|^n \|v\|_{\lp{p}}
\le \dm^n \| B_n*v\|_{\lp{p}}
\le \dm^n \|B_n\|_{\lp{p}} \|v\|_{\lp{1}}.
\]
Using \eqref{an:b:sm} and noting $B_n=\dm^{-n}\sd_{b,\dm}^n \td$, we conclude that
\[
|\gl|=\lim_{n\to\infty} |\gl| \|v\|_{\lp{p}}^{1/n}
\le \dm \limsup_{n\to\infty} \|B_n\|_{\lp{p}}^{1/n}=
 \limsup_{n\to\infty} \|\sd_{b,\dm}^n \td\|_{\lp{p}}^{1/n}=
\dm^{\frac{1}{p}-\sm_p(a,\dm)}<\dm^{-m},
\]
where we used our assumption $\sm_p(a,\dm)>\frac{1}{p}+m$ in the last inequality. This proves that any nonzero eigenvalue in $\mbox{spec}(\mtz_{b,\dm,\gamma})$ must be less than $\dm^{-m}$. Now all the claims follow directly
from \eqref{spec:ab} of \cref{lem:eig} and our assumption $\tilde{\pa}(1)=1$.
\ep

Before proving \cref{thm:int},
for a convergent subdivision scheme, we first show that the special limit function $\eta_{\td}$ for the particular initial data $v=\td$ must be the $\dm$-refinable function $\phi$.
Indeed, since $a\in \lp{0}$, the mask $a$ must support inside $[-N,N]$ for some $N\in \N$ and therefore, the function $\eta_{\td}$ must be supported inside $[-\frac{N}{\dm-1},\frac{N}{\dm-1}]\subseteq [-N,N]$. Then for any $x:=\dm^{-n_0} k_0$ with $n_0\in \N$ and $k_0\in \Z$, noting that $x=\dm^{-n}(\dm^{n-n_0} k_0)$ with $\dm^{n-n_0}k_0\in \Z$ for all $n\ge n_0$, we directly derive from \eqref{converg} and the definition of $\sd_{a,\dm}$ in \eqref{sd} that
\[
\eta_{\td}(x)=\lim_{n\to\infty} [\sd_{a,\dm}^{n+n_0} \td](\dm^{n-n_0} k_0)
=\dm \sum_{k=-N}^N a(k) \lim_{n\to\infty} [\sd_{a,\dm}^{n+n_0-1} \td](\dm^{n-n_0}k_0-\dm k)=
\dm\sum_{k\in \Z} a(k) \eta_{\td}(\dm x-k).
\]
Since $\{\dm^{-n_0}k_0 \setsp n_0 \in \N, k_0\in \Z\}$ is dense in $\R$ and $\eta_{\td}$ is continuous, the function $\eta_{\td}$ must satisfy \eqref{refeq}, i.e., $\wh{\eta_{\td}}(\dm \xi)=\tilde{\pa}(e^{-i\xi})\wh{\eta_{\td}}(\xi)$.
By $\sum_{k\in \Z} a(k)=1$, we observe that $\sum_{k\in \Z} [\sd_{a,\dm}^n \td](k)=\dm^n$ for all $n\in \N$.
Now using a Riemann sum for the continuous function $\eta_{\td}$, we deduce from \eqref{converg} that
\[
\int_{\R} \eta_{\td}(x)dx=\lim_{n\to\infty}
\dm^{-n} \sum_{k=-\dm^n N}^{\dm^n N}
\eta_{\td}(\dm^{-n} k)=\lim_{n\to\infty} \dm^{-n} \sum_{k\in \Z} [\sd_{a,\dm}^n \td](k)=1.
\]
That is, $\wh{\eta_{\td}}(0)=1$. Hence, the function $\eta_{\td}$ must agree with the $\dm$-refinable function $\phi$ with mask $a$.
For $0<\tau\le 1$ and a function $f\in \Lp{p}$, we say that $f$ belongs to the Lipschitz space $\mbox{Lip}(\tau, \Lp{p})$ if there exists a positive constant $C$ such that $\|f-f(\cdot-t)\|_{\Lp{p}}\le C\|t\|^\tau$ for all $t\in \R$. For convenience, we define $\mbox{Lip}(0, \Lp{p}):=\Lp{p}$.
The $L_p$ smoothness of a function $f\in \Lp{p}$ is measured by its \emph{$L_p$ critical exponent} $\sm_p(f)$ defined by
\be \label{sm:f}
\sm_p(f):=\sup\{ m+\tau \setsp  0\le \tau<1\; \mbox{ and }\; f^{(m)}\in \mbox{Lip}(\tau, \Lp{p}), m\in \NN\}.
\ee
For a compactly supported distribution $f$, we say that the integer shifts of $f$ are \emph{stable} if $\mbox{span}\{\wh{f}(\xi+2\pi k)\setsp k\in \Z\}=\C$ for every $\xi\in \R$.
For the $\dm$-refinable function $\phi$ with mask $a\in \lp{0}$,
by \cite[Theorem~3.3]{han99} and \cite[Theorem~4.3]{han03}, we have $\sm_p(\phi)\ge \sm_p(a,\dm)$; In addition, $\sm_p(\phi)= \sm_p(a,\dm)$ if the integer shifts of $\phi$ are stable.

Next, we discuss how the subdivision operator $\sd_{a,\dm}$ reproduces polynomials.
For any mask $a\in \lp{0}$, it is known in \cite[(2.20)]{han03} (also see \cite[Theorem~3.4]{han13})
that for $J:=\sr(a,\dm)$,
\be \label{sr:poly}
\sd_{a,\dm} \pp=\sum_{k\in \Z} \pp(\dm^{-1}(\cdot-k))a(k)
=\sum_{j=0}^\infty \frac{(-\dm^{-1})^j \pp^{(j)}(\dm^{-1}\cdot)}{j!} \sum_{k\in \Z} k^j a(k),
\qquad \pp\in \Pi_{J-1}.
\ee
We prove it here for the convenience of the reader. Using the Taylor expansion of $\pp\in \Pi_{J-1}$, we have
\[
\pp(k)=\pp(\dm^{-1}x-\dm^{-1}(x-\dm k))=
\sum_{j=0}^\infty \frac{\pp^{(j)}(\dm^{-1} x)}{j!} (-\dm^{-1}(x-\dm k))^j
=\sum_{j=0}^{\infty} \frac{(-\dm^{-1})^j \pp^{(j)}(\dm^{-1} x)}{j!}
(x-\dm k)^j.
\]
By the definition of $\sd_{a,\dm}$ in \eqref{sd}, using \eqref{sr:2} for sum rules and the above identity, for $\pp\in \Pi_{J-1}$ and $x\in \Z$, we have
\begin{align*}
[\sd_{a,\dm} \pp](x)
&=\dm \sum_{k\in \Z} \pp(k) a(x-\dm k)=\dm \sum_{j=0}^{\infty}
\frac{(-\dm^{-1})^j \pp^{(j)}(\dm^{-1}x)}{j!}
\sum_{k\in \Z}  (x-\dm k)^j a(x-\dm k)\\
&=\sum_{j=0}^{\infty}
\frac{(-\dm^{-1})^j \pp^{(j)}(\dm^{-1}x)}{j!}
\sum_{k\in \Z}  k^j a(k)=\sum_{k\in \Z} \left(\sum_{j=0}^{\infty} \frac{(-\dm^{-1} k)^j \pp^{(j)}(\dm^{-1}x)}{j!}\right)a(k),
\end{align*}
which proves \eqref{sr:poly} by noting
$\sum_{j=0}^{\infty} \frac{1}{j!}\pp^{(j)}(\dm^{-1} x) (-\dm^{-1} k)^j=\pp(\dm^{-1}(x-k))$.

We are now ready to prove \cref{thm:int}.

\begin{proof}[Proof of \cref{thm:int}]
(1)$\imply$(2). For $v\in \lp{}$, we define $f_v(x):=\sum_{k\in \Z} v(k) \phi(x-k)$.
Because $\phi(s_a+k)=\td(k)$ for all $k\in \Z$, we have $v(j)=f_v(s_a+j)$ for all $j\in \Z$. Hence, if $f_v=0$, then we must have $v(j)=0$ for all $j\in \Z$. Therefore, the integer shifts of $\phi$ are linearly independent and hence stable. Because $\phi\in \CH{m}$, we conclude from
\cite[Theorem~4.3 or Corollary~5.1]{han03} that $\sm_\infty(a,\dm)>m$.
Define a sequence $w\in \lp{0}$ by
\be \label{w:phi}
w(k):=\phi(\dm^{m_s}s_a+k),\qquad k\in \Z.
\ee
Note that $\phi$ is $\dm^n$-refinable with the mask $A_n$ for every $n\in \N$ by \eqref{phi:An}.
Then for all $k\in \Z$,
\[
[A_{m_s}*w](\dm^{m_s}k)=\sum_{j\in \Z} A_{m_s}(j) w(\dm^{m_s}k-j)=
\sum_{j\in \Z} A_{m_s}(j) \phi(\dm^{m_s}(s_a+k)-j)=\dm^{-m_s} \phi(s_a+k)=\dm^{-m_s} \td(k),
\]
which proves \eqref{cond:ms}, and
\begin{align*}
[A_{n_s}*w](\dm^{m_s}(\dm^{n_s}-1)s_a+\dm^{n_s}k)
&=\sum_{j\in \Z} A_{n_s}(j)w(\dm^{m_s}(\dm^{n_s}-1)s_a+\dm^{n_s}k-j)\\
&=\sum_{j\in \Z} A_{n_s}(j) \phi(\dm^{n_s}(\dm^{m_s} s_a+k)-j)\\
&=\dm^{-n_s} \phi(\dm^{m_s}s_a+k)=\dm^{-n_s} w(k),
\end{align*}
which proves \eqref{cond:ns}. This proves (1)$\imply$(2).

(2)$\imply$(1). Because $\sm_\infty(a,\dm)>m$, by \cite[Theorem~4.3 or Corollary~5.1]{han03}, we have $\phi\in \CH{m}$, which is also obtained from $\sm_\infty(\phi)\ge \sm_\infty(a,\dm)$.
Define
\[
v(k):=\phi(\dm^{m_s}s_a+k),\qquad k\in \Z.
\]
Since $\phi$ is $\dm^n$-refinable with the mask $A_n$ for every $n\in \N$,  by the same argument as in the proof of (1)$\imply$(2), we have
\[
[A_{m_s}*v](\dm^{m_s}k)=\sum_{j\in \Z} A_{m_s}(j) v(\dm^{m_s}k-j)=
\sum_{j\in \Z} A_{m_s}(j) \phi(\dm^{m_s}(s_a+k)-j)=\dm^{-m_s} \phi(s_a+k)
\]
and
\[
[A_{n_s}*v](\dm^{m_s}(\dm^{n_s}-1)s_a+\dm^{n_s}k)=
\sum_{j\in \Z} A_{n_s}(j) \phi(\dm^{n_s}(\dm^{m_s} s_a+k)-j)=
\dm^{-n_s}\phi(\dm^{m_s} s_a+k)=\dm^{-n_s} v(k).
\]
That is, we proved
\be \label{phi:w:v}
[A_{m_s}*v](\dm^{m_s}k)=\dm^{-m_s} \phi(s_a+k)
\quad \mbox{and}\quad [A_{n_s}*v](\dm^{m_s}(\dm^{n_s}-1)s_a+\dm^{n_s}k)=\dm^{-n_s} v(k).
\ee
Using the definition of a shifted transition operator $\tz_{a,\dm,\gamma}$ in \eqref{tz}, the second identity in \eqref{phi:w:v} can be equivalently expressed as $\tz_{A_{n_s},\dm^{n_s},\gamma} v=v$ with $\gamma:=\dm^{m_s}(\dm^{n_s}-1)s_a\in \Z$. That is, $v\in \lp{0}$ is an eigenvector of $\tz_{A_{n_s},\dm^{n_s},\gamma}$ corresponding to the eigenvalue $1$.
Similarly, the identity in \eqref{cond:ns} can be equivalently expressed as $\tz_{A_{n_s},\dm^{n_s},\gamma} w=w$. Note that $w$ cannot be the trivial zero sequence due to \eqref{cond:ms}.
Also it is easy to deduce from the definition of $\sm_\infty(a,\dm)$ in \eqref{sm} that $\sm_\infty(A_{n_s},\dm^{n_s})=\sm_\infty(a,\dm)>m\ge 0$.
By $\sm_\infty(A_{n_s},\dm^{n_s})>0$ and \cref{thm:eig} with $a$ and $\dm$ being replaced by $A_{n_s}$ and $\dm^{n_s}$, respectively, we conclude that $1$ must be a simple eigenvalue of $\tz_{A_{n_s},\dm^{n_s},\gamma}$ and hence, we conclude that $v=c w$ for some constant $c$.
Now by \eqref{cond:ns} and the first identity in \eqref{phi:w:v}, we have
\[
\dm^{-m_s} \phi(s_a+k)= [A_{m_s}*v](\dm^{m_s}k)=c [A_{m_s}*w](\dm^{m_s}k)
=c\dm^{-m_s}\td(k)
\]
for all $k\in \Z$. Hence, $\phi(s_a+k)=c\td(k)$ for all $k\in \Z$.
By \cref{thm:eig}, we have $\sr(a,\dm)\ge \sm_\infty(a,\dm)>m\ge 0$ and hence, \eqref{phi:poly} must hold with $J=1$ by
the proof of \cref{prop:lpm}. Hence, we conclude from \eqref{phi:poly} that $1=\sum_{k\in \Z} \phi(s_a+k)=\sum_{k\in \Z} c\td(k)=c$.
This proves (2)$\imply$(1).

(3)$\imply$(1).
Because the subdivision scheme is convergent, we already explained that $\eta_{\td}$ must be the $\dm$-refinable function $\phi$ with the mask $a$. By \cite[Theorem~2.1]{hj06} or \cite[Theorem~4.3]{han03}, we must have $\sm_\infty(a,\dm)>m$ and hence $\phi$ must belong to $\CH{m}$. Taking $v=\td$ in \eqref{cond:intsubdiv}, \eqref{intphi:sa} follows directly from \eqref{cond:intsubdiv}. Hence, $\phi$ must be an $s_a$-interpolating $\dm$-refinable function. This proves (3)$\imply$(1).

(1)$\imply$(3). We proved (1)$\imply$(2) and hence, we have $\sm_\infty(a,\dm)>m$. By \cite[Theorem~2.1]{hj06}, the $\dm$-subdivision scheme with mask $a$ is $\mathscr{C}^m$-convergent and hence, we already proved that $\eta_{\td}=\phi$ and $\eta_v=\sum_{k\in \Z} v(k)\phi(\cdot-k)$.
The identity \eqref{cond:intsubdiv} follows trivially from \eqref{intphi:sa}. This proves (1)$\imply$(3).

We now prove \eqref{poly:int}.
Since $J=\sr(a,\dm)$,
we conclude from \eqref{a:lpm} and \eqref{sr:poly} that
\[
\sd_{a,\dm} \pp
=\sum_{j=0}^\infty \frac{(-\dm^{-1})^j \pp^{(j)}(\dm^{-1}\cdot)}{j!} m_a^j=
\sum_{j=0}^\infty \frac{1}{j!}\pp^{(j)}(\dm^{-1} \cdot) (-\dm^{-1} m_a)^j=
\pp(\dm^{-1}(\cdot-m_a))=\pp(\dm^{-1}(s_a+\cdot)-s_a),
\]
where we used $s_a=\frac{m_a}{\dm-1}$ and hence $-\dm^{-1}m_a=\dm^{-1} s_a-s_a$. Now
by induction we have
\[
\sd_{a,\dm}^n \pp=\sd_{a,\dm}  [\pp(\dm^{1-n}(\cdot+s_a)-s_a)]=\pp(\dm^{1-n}(\dm^{-1}(\cdot+s_a)-s_a+s_a)-s_a)=
\pp(\dm^{-n}(\cdot+s_a)-s_a).
\]
This proves \eqref{poly:int}.
\end{proof}

Let $\phi$ be the $\dm$-refinable function with a mask $a\in \lp{0}$, i.e., $\wh{\phi}(\xi):=\prod_{j=1}^\infty \tilde{\pa}(e^{-i\dm^{-j}\xi})$. Under the condition $\sm_\infty(a,\dm)>0$, the function $\phi$ is continuous.
Note that $s_a\in \R$ satisfies \eqref{cond:sa} if and only if $s_a\in D_\dm$, where $D_\dm:=\cup_{m_s=0}^\infty \cup_{n_s=1}^\infty [\dm^{-m_s}(\dm^{n_s}-1)^{-1}\Z]$ is dense in $\R$. It is easy to observe that $\cup_{m=0}^\infty [\dm^{-m}\Z]\subseteq D_\dm$ and $\cup_{n=1}^\infty [(\dm^{n}-1)^{-1}\Z]\subseteq D_\dm$.
For every $s_a\in D_\dm$, we now discuss how to exactly compute $\phi(s_a)$ through \eqref{phi:w:v} within finite steps. Because $\sm_\infty(a,\dm)>0$, as we already know in the proof of \cref{thm:int}, $\tz_{A_{n_s},\dm^{n_s},\gamma}$ with $\gamma:=\dm^{m_s}(\dm^{n_s}-1)s_a\in \Z$ has a simple eigenvalue $1$, and the second identity in \eqref{phi:w:v} is equivalent to saying that  $v\in \lp{0}$ is an eigenvector of $\tz_{A_{n_s},\dm^{n_s},\gamma}$ corresponding to the eigenvalue $1$. Now the value $\phi(s_a)$ can be exactly computed within finite steps as follows:

\begin{enumerate}
\item[(S1)] Compute the unique eigenvector $v\in \lp{0}$ such that $\tz_{A_{n_s},\dm^{n_s},\gamma} v=v$ and $\sum_{k\in \Z} v(k)=1$.
\item[(S2)] Then $\phi(s_a)=\dm^{m_s} [A_{m_s}*v](0)$.
\end{enumerate}
For any $s_a\not\in D_\dm$,
the set $[0,1)\cap (\cup_{j=1}^\infty [\dm^j s_a+\Z])$ must be infinite and
it is unlikely that $\phi(s_a)$ can be computed within finite steps through the refinement equation using only the mask $a\in \lp{0}$.

\section{Proof of \cref{thm:qs} on
Quasi-stationary Subdivision Schemes}
\label{sec:qs}

In this section we shall first prove \cref{thm:qs}. Then we shall discuss how to combine \cref{thm:int,thm:qs} for $r n_s$-step interpolatory $r$-mask quasi-stationary  subdivision schemes.

\begin{proof}[Proof of \cref{thm:qs}]
The key ingredient of the proof is to show that $\sr(a_\ell,\dm)>m$ for all $\ell=1,\ldots,r$ play the critical role for proving the convergence of $r$-mask quasi-stationary subdivision schemes.
Since all involved masks $a_1,\ldots,a_r$ have finite supports, we observe that
\eqref{converg:qs} holds for every $K>0$ and $v\in \lp{}$ if and only if it holds for $K=\infty$ and $v\in \lp{0}$. Hence, for simplicity of presentation, we shall assume $v\in \lp{0}$ and use $K=\infty$ in \eqref{converg:qs}.

Sufficiency. Because $\sm_\infty(a,\dm^r)>m$, we conclude from \cite[Theorem~2.1]{hj06} that the $\dm^r$-subdivision scheme with mask $a$ is $\mathscr{C}^m$-convergent and its $\dm^r$-refinable function $\phi$ belongs to $\CH{m}$. Hence, for every initial sequence $v\in \lp{0}$, we conclude that
\be \label{converg:ell0}
\lim_{n\to\infty} \sup_{k\in \Z} \left|\dm^{j(rn+\ell)} [\nabla^j \sd_{a_\ell,\dm}\cdots \sd_{a_1,\dm}(\sd_{a_r,\dm}\cdots \sd_{a_1,\dm})^n v](k)-\eta_v^{(j)}(\dm^{-rn-\ell} k)\right|=0
\ee
holds for $\ell=0$ and $\ell=r$, where we used the convention that $\sd_{a_\ell,\dm}\cdots \sd_{a_1,\dm}$ is the identity mapping for $\ell=0$.
To prove \eqref{converg:ell0} for all $\ell=1,\ldots,r$, we have to prove \eqref{converg:ell0} for every $\ell \in \{1,\ldots,r-1\}$.
Note that $\sd_{a_r,\dm}\cdots \sd_{a_1,\dm}=\sd_{a,\dm^r}$.
By the assumption in \eqref{cond:qs}, for $j=0,\ldots,m$, there exists a unique finitely supported sequence $b_j\in \lp{0}$ such that
\be \label{aellbj}
\widetilde{\pa_1}(z^{\dm^{\ell-1}})\cdots \widetilde{\pa_\ell}(z)=(1+z+\cdots+z^{\dm^\ell-1})^j \widetilde{\pb_j}(z)
\ee
and $\sr(b_j, \dm^\ell)\ge m+1-j\ge 1$. Noting that $\nabla^j \sd_{a_\ell,\dm}\cdots\sd_{a_1,\dm}=
\sd_{b_j,\dm^\ell}\nabla^j$,
to prove \eqref{converg:ell0}, we have to prove the following equivalent form of \eqref{converg:ell0}: For $\ell\in \{1,\ldots,r-1\}$ and $j=0,\ldots,m$,
\be \label{converg:s}
\lim_{n\to\infty} \sup_{k\in \Z} \left|\dm^{j(rn+\ell)} [\sd_{b_j,\dm^\ell} \nabla^j \sd_{a,\dm^r}^n v](k)-\eta_v^{(j)}(\dm^{-rn-\ell} k)\right|=0.
\ee
From the following identity
\begin{align*}
&[\dm^{j(rn+\ell)}\sd_{b_j,\dm^\ell} \nabla^j \sd_{a,\dm^r}^n v](k)-\eta_v^{(j)}(\dm^{-rn-\ell} k)\\
&\qquad
=\dm^{j\ell} \left[ \sd_{b_j,\dm^\ell} \left(\dm^{jrn}[\nabla^j \sd_{a,\dm^r}^n v](\cdot)-\eta_v^{(j)}(\dm^{-rn} \cdot)\right)\right](k)+
\dm^{j\ell}\left[\sd_{b_j,\dm^\ell}(\eta_v(\dm^{-rn}\cdot))\right](k)-
\eta_v^{(j)}(\dm^{-rn-\ell} k),
\end{align*}
we conclude that
\be \label{est:I1I2}
\sup_{k\in \Z}\left|\dm^{j(rn+\ell)}[\sd_{b_j,\dm^\ell} \nabla^j \sd_{a,\dm^r}^n v](k)-\eta_v^{(j)}(\dm^{-rn-\ell} k)\right|
\le I_1+I_2
\ee
with
\[
I_1:=\sup_{k\in \Z}\left|\dm^{j\ell}\left[\sd_{b_j,\dm^\ell} \left(\dm^{jrn}[\nabla^j \sd_{a,\dm^r}^n v](\cdot)-\eta_v^{(j)}(\dm^{-rn} \cdot)\right)\right](k)\right|
\]
and
\[
I_2:= \sup_{k\in \Z}\left|\dm^{j\ell}[\sd_{b_j,\dm^\ell}(\eta_v(\dm^{-rn}\cdot))](k)-
\eta_v^{(j)}(\dm^{-rn-\ell} k)]\right|.
\]
Using the fact that $\|\sd_{b_j,\dm^\ell} w\|_{\lp{\infty}}\le \dm^\ell \|b_j\|_{\lp{1}} \|w\|_{\lp{\infty}}$, we conclude that
\[
I_1\le \dm^{(j+1)\ell} \|b_j\|_{\lp{1}} \left\| \dm^{jrn} [\nabla^j \sd_{a,\dm^r}^n v](\cdot)-\eta_v^{(j)}(\dm^{rn} \cdot)\right\|_{\lp{\infty}},
\]
which goes to $0$ as $n\to\infty$ by the proved fact that \eqref{converg:ell0} holds with $\ell=0$ and $\sd_{a_r,\dm}\cdots \sd_{a_1,\dm}=\sd_{a,\dm^r}$.

Note that $b_j\in \lp{0}$ is finitely supported and by $\sr(b_j,\dm^\ell)\ge m+1-j\ge 1$, $b_j$ must have at least order one sum rules with respect to $\dm^\ell$, that is, $\sum_{\gamma\in \Z} b_j(k+\dm^\ell \gamma)=\dm^{-\ell}\sum_{\gamma\in \Z} b_j(\gamma)=\dm^{-(j+1)\ell}$ for all $k \in \Z$, due to \eqref{aellbj} and our assumption $\sum_{\gamma\in \Z} a_q(\gamma)=1$ for all $q=1,\ldots,r$.
 Consequently,
\begin{align*}
&\dm^{j\ell}\left[\sd_{b_j,\dm^\ell} (\eta_v^{(j)}(\dm^{-rn}\cdot))\right](k)
-\eta_v^{(j)}(\dm^{-rn-\ell}k)\\
&\qquad =\dm^{(j+1)\ell} \sum_{\gamma\in \Z\cap [\dm^{-\ell}k-\dm^{-\ell}\fs(b_j)]} b_j(k-\dm^\ell \gamma)\left [\eta_v^{(j)}(\dm^{-rn}\gamma)-
\eta_v^{(j)}(\dm^{-rn} \dm^{-\ell}k)\right].
\end{align*}
Take $N\in \N$ such that $\fs(b_j)\subseteq [-N,N]$. Then for all $\gamma\in \Z\cap [\dm^{-\ell}k-\dm^{-\ell}\fs(b_j)]$, we have
\[
|\dm^{-rn}\gamma-\dm^{-rn}\dm^{-\ell}k|\le
\dm^{-rn}|\gamma-\dm^{-\ell}k|\le \dm^{-rn-\ell}N.
\]
Note that $\eta_v^{(j)}$ is a compactly supported uniformly continuous function on $\R$ because $v\in \lp{0}$ and $\phi$ has compact support.
Therefore, we conclude from the above inequalities that
\[
|I_2|\le \dm^{(j+1)\ell} \|b_j\|_{\lp{1}}\sup_{|x-y|\le \dm^{-rn-\ell}N} |\eta_v^{(j)}(x)-\eta_v^{(j)}(y)|,
\]
which goes to $0$ as $n\to\infty$. This proves \eqref{converg:s}. Hence, \eqref{converg:qs} must hold.

Necessity. Suppose that \eqref{converg:qs} holds. In particular, \eqref{converg:qs} holds with $n$ being replaced by $rn$. Hence, the $\dm^r$-subdivision scheme with mask $a$ must be $\mathscr{C}^m$-convergent. By \cite[Theorem~2.1]{hj06}, we conclude that $\sm_\infty(a,\dm^r)>m$. This proves the first part of \eqref{cond:qs}. Moreover, by the discussion immediately above the proof of \cref{thm:int}, we conclude that the function $\eta_v$ in \eqref{converg:qs} must satisfy $\eta_v=\sum_{k\in \Z} v(k) \phi(\cdot-k)$. In particular, we have $\eta_{\td}=\phi$ in \eqref{converg:qs}.
Now we use the proof by contradiction to prove $\sr(a_\ell,\dm)>m$ for all $\ell=1,\ldots,r$. Suppose not.
Then $j:=\sr(a_\ell,\dm) \le m$ for some $\ell=1,\ldots,r$. Since $j=\sr(a_\ell,\dm)$, we can write
\[
\widetilde{\pa_\ell}(z)=(1+z+\cdots+z^{\dm-1})^j \widetilde{\pb_\ell}(z)
\]
for some $b_\ell\in \lp{0}$ such that $\sr(b_\ell,\dm)=0$.
By \eqref{converg:qs} with $v=\td$, using $\nabla^j \sd_{a_\ell,\dm}=\sd_{b_\ell,\dm}\nabla^j$, we have
\be \label{aell:qs}
\lim_{n\to\infty} \sup_{k\in \Z}
\left|\dm^{j(rn+\ell)} \left[\sd_{b_\ell,\dm} \nabla^j \sd_{a_{\ell-1},\dm}\cdots \sd_{a_1,\dm} \sd_{a,\dm^r}^n\td\right](k)
-\phi^{(j)}(\dm^{-rn-\ell}k)\right|=0.
\ee
Now we can decompose the expression on the left-hand side of \eqref{aell:qs} into
\be \label{J1J2}
\dm^{j(rn+\ell)} \left[\sd_{b_\ell,\dm} \nabla^j \sd_{a_{\ell-1},\dm}\cdots \sd_{a_1,\dm} \sd_{a,\dm^r}^n\td\right](k)-\phi^{(j)}(\dm^{-rn-\ell}k)
=J_1(k)+J_2(k)
\ee
with
\[
J_1(k):=\dm^j \left[\sd_{b_\ell,\dm} \left(\left[\dm^{j(rn+\ell-1)}\nabla^j \sd_{a_{\ell-1},\dm}\cdots \sd_{a_1,\dm} \sd_{a,\dm^r}^n\td\right](\cdot)
-\phi^{(j)}(\dm^{-rn-\ell+1}\cdot)\right)\right](k)
\]
and
\[
J_2(k):=
\dm^j [\sd_{b_\ell,\dm} (\phi^{(j)}(\dm^{-rn-\ell+1}\cdot))](k)
-\phi^{(j)}(\dm^{-rn-\ell}k).
\]
Because \eqref{converg:qs} holds, we particularly have
\[
\lim_{n\to\infty} \sup_{k\in \Z}
|\dm^{j(n+\ell-1)} [\nabla^j \sd_{a_{\ell-1},\dm}\cdots \sd_{a_1,\dm} \sd_{a,\dm^r}^n\td](k)-\phi^{(j)}(\dm^{-rn-\ell+1}k)|=0.
\]
Hence, using the above identity and the same argument for $I_1$, we obtain $\lim_{n\to\infty} \sup_{k\in \Z}|J_1(k)|=0$. Consequently, we conclude from \eqref{aell:qs} and \eqref{J1J2} that
\be \label{J2}
\lim_{n\to\infty} \sup_{k\in \Z} |J_2(k)|=0.
\ee
Because $\sr(b_\ell,\dm)=0$, by \eqref{sr:2}, there must exist $\tilde{k} \in \Z$ such that $c:=\sum_{\gamma\in \Z} b(\tilde{k}+\dm \gamma)\ne \frac{1}{\dm} \sum_{\gamma\in \Z} b_\ell(\gamma)=\dm^{-j-1}$. That is, $\dm^{j+1}c\ne 1$. Then for every $k\in [\tilde{k}+\dm \Z]$, we have
{\small{\begin{align*}
J_2(k)&=\Big(\dm^{j+1} \sum_{\gamma\in \Z} b_\ell(k-\dm \gamma) \phi^{(j)}(\dm^{-rn-\ell+1}\gamma)\Big)
-\phi^{(j)}(\dm^{-rn-\ell}k)\\
&=\dm^{j+1} \sum_{\gamma\in \Z\cap [\dm^{-1}k-\dm^{-1}\fs(b_\ell)]} b_\ell(k-\dm \gamma) \left[\phi^{(j)}(\dm^{-rn-\ell+1}\gamma)
-\phi^{(j)}(\dm^{-rn-\ell}k)\right]
+(\dm^{j+1}c-1)\phi^{(j)}(\dm^{-rn-\ell}k).
\end{align*}
}}%
Take $N\in \N$ such that $\fs(b_\ell)\subseteq [-N,N]$. Then for all $\gamma \in \Z\cap [\dm^{-1}k-\dm^{-1}\fs(b_\ell)]$, as we discussed before, we have
\[
|\dm^{-rn-\ell+1}\gamma-\dm^{-rn-\ell}k|=\dm^{-rn-\ell+1}
|\gamma-\dm^{-1} k|\le \dm^{-rn-\ell}N.
\]
Hence, for every $k\in [\tilde{k}+\dm \Z]$, we have
\begin{align*}
&\sup_{k\in [\tilde{k}+\dm \Z]}\Big|
\sum_{\gamma\in \Z\cap [\dm^{-1}k-\dm^{-1}\fs(b_\ell)]} b_\ell(k-\dm \gamma) \left[\phi^{(j)}(\dm^{-rn-\ell+1}\gamma)
-\phi^{(j)}(\dm^{-rn-\ell}k)\right]\Big|\\
&\qquad\qquad\qquad
\le \|b_\ell\|_{\lp{1}} \sup_{|x-y|\le \dm^{-rn-\ell}N} |\phi^{(j)}(x)-\phi^{(j)}(y)|,
\end{align*}
which goes to $0$ by the uniform continuity of the compactly supported continuous function $\phi^{(j)}$.
From the above inequality,
we now conclude from \eqref{J2} and $\dm^{j+1}c\ne 1$ that
\be \label{phito0}
\lim_{n\to\infty} \sup_{k\in [\tilde{k}+\dm \Z]} |\phi^{(j)}(\dm^{-rn-\ell}k)|=0.
\ee
Take $x:=\dm^{-n_0} k_0$ with $n_0\in \NN$ and $k_0\in \Z$. Then
$x=\dm^{-rn-\ell}\dm k_1$ with $k_1:=\dm^{rn+\ell-1-n_0} k_0\in \Z$ for all $n\ge (1+n_0-\ell)/r$. Consequently, for all $n\ge (1+n_0-\ell)/r$, we have
\begin{align*}
|\phi^{(j)}(x)|
&=|\phi^{(j)}(\dm^{-rn-\ell} \dm k_1)|
\le |
\phi^{(j)}(\dm^{-rn-\ell}\dm k_1)
-\phi^{(j)}(\dm^{-rn-\ell}(\tilde{k}+\dm k_1))|+
|\phi^{(j)}(\dm^{-rn-\ell}(\tilde{k}+\dm k_1))|\\
&\le \sup_{|y-z|\le \dm^{-rn-\ell}|\tilde{k}|}
|\phi^{(j)}(y)-\phi^{(j)}(z)|+
\sup_{k\in [\tilde{k}+\dm \Z]} |\phi^{(j)}(\dm^{-rn-\ell} k)|,
\end{align*}
which goes to zero as $n\to\infty$ by using \eqref{phito0} and the uniform continuity of the compactly supported continuous function $\phi^{(j)}$. Hence, $\phi^{(j)}(x)=0$, that is, $\phi^{(j)}(\dm^{-n_0} k_0)=0$ for all $n_0\in \NN$ and $k_0\in \Z$.
Since $\{\dm^{-n_0} k_0\setsp n_0\in \NN, k_0\in \Z\}$ is dense in $\R$, we conclude that $\phi^{(j)}(x)=0$ for all $x\in \R$, which implies that $\phi$ must be a polynomial of degree less than $j$.
However, $\phi$ is compactly supported, which forces $\phi$ to be identically zero, a contradiction to $\wh{\phi}(0)=1$.
Consequently, we must have $\sr(a_\ell,\dm)>m$ for all $\ell=1,\ldots,r$. This proves \eqref{cond:qs}.
\end{proof}

We make some remarks here about \cref{thm:qs}. Generalizing refinable functions, we can define compactly supported functions $\phi_1,\ldots,\phi_r$ through a chain
of nested refinement equations as follows:
\[
\wh{\phi_\ell}(\dm\xi)=\widetilde{\pa_\ell}(e^{-i\xi})\wh{\phi_{\ell+1}}(\xi),\quad
\ell=1,\ldots,r-1\quad \mbox{and}\quad
\wh{\phi_r}(\dm \xi)=\widetilde{\pa_r}(e^{-i\xi})\wh{\phi_1}(\xi)
\]
under the normalization condition $\wh{\phi_1}(0)=\cdots=\wh{\phi_r}(0)=1$.
Then we must have
\[
\wh{\phi_1}(\dm^r\xi)=
\widetilde{\pa_1}
(e^{-i\dm^{r-1}\xi})\wh{\phi_2}(\dm^{r-1}\xi)=
\widetilde{\pa_1}
(e^{-i\dm^{r-1}\xi})
\cdots \widetilde{\pa_{r-1}}(e^{-i \dm\xi})
\widetilde{\pa_r}(e^{-i\xi})\wh{\phi_1}(\xi)=
\widetilde{\pa}(e^{-i\xi})\wh{\phi_1}(\xi).
\]
Hence we must have $\phi_1=\phi$, that is,
the $\dm^r$-refinable function $\phi$ with the mask $a$ in \eqref{mask:a1r} of \cref{thm:qs} can be obtained from the functions $\phi_1,\ldots,\phi_r$ satisfying the chain of nested refinement equations. In fact, this alternative interpretation of $\phi$ allows us to obtain non-traditional wavelets from the $r$ masks $a_1,\ldots, a_r$ in \cref{thm:qs}. We shall address this direction elsewhere.

Finally, we discuss how to combine \cref{thm:int,thm:qs} for interpolatory quasi-stationary subdivision schemes. It is very natural to
obtain a mask $a$ in \eqref{mask:a1r} of \cref{thm:qs} using masks $\{a_1,\ldots,a_r\}$ and then require that the mask $a$ defined in \eqref{mask:a1r} should satisfy the conditions in \cref{thm:int} for obtaining $s_a$-interpolating $\dm^r$-refinable function $\phi$. This leads to $rn_s$-interpolating $r$-mask quasi-stationary subdivision schemes as follows:

\begin{cor}\label{cor:qs}
Let $\dm\in \N\bs\{1\}$ be a dilation factor and $r\in \N$. Let $m\in \NN$ and $a_1,\ldots,a_r\in \lp{0}$ be finitely supported masks with $\sum_{k\in \Z} a_\ell(k)=1$ for $\ell=1,\ldots,r$.
Define a mask $a\in \lp{0}$ as in \eqref{mask:a1r}
and the $\dm^r$-refinable function $\phi$ via the Fourier transform $\wh{\phi}(\xi):=\prod_{j=1}^\infty \tilde{\pa}(e^{-i \dm^{-r j}\xi})$ for $\xi\in \R$.
For a real number $s_a\in \R$ satisfying the following condition
\[
\dm^{r m_s}(\dm^{r n_s}-1)s_a\in \Z \quad\mbox{for some $m_s\in \NN$ and $n_s\in \N$},
\]
the $\dm^r$-refinable function $\phi$ is $s_a$-interpolating and
the $r$-mask quasi-stationary $\dm$-subdivision scheme with masks $\{a_1,\ldots,a_r\}$ is $\mathscr{C}^m$-convergent $\infty$-step interpolatory
if and only if
\begin{enumerate}
\item[(i)] $\sm_\infty(a,\dm^r)>m$ and $\sr(a_\ell,\dm)>m$ for all $\ell=1,\ldots,r$;

\item[(ii)] there is a finitely supported sequence $w\in \lp{0}$ such that
\begin{align}
&[A_{m_s}*w](\dm^{r m_s} k)=\dm^{-rm_s} \td(k)\qquad \forall\; k\in \Z, \label{cond:ms:qs}\\
&[A_{n_s}*w](\dm^{rm_s}(\dm^{rn_s}-1)s_a+\dm^{rn_s}k)=\dm^{-rn_s} w(k),\qquad \forall\; k\in \Z,\label{cond:ns:ws}
\end{align}
where the finitely supported masks $A_n$ are defined to be $A_n:=\dm^{-rn} \sd_{a,\dm^r}^n \td$. For the particular case $m_s=0$, the conditions in \eqref{cond:ms:qs} and \eqref{cond:ns:ws} together are equivalent to
\be \label{cond:ms=0:ws}
A_{n_s}((\dm^{rn_s}-1)s_a+\dm^{rn_s}k)=\dm^{-rn_s} \td(k)\qquad \forall\; k\in \Z.
\ee
\end{enumerate}
For the particular case $m_s=0$,
the $\infty$-step interpolatory $r$-mask quasi-stationary $\dm$-subdivision scheme with masks $\{a_1,\ldots,a_r\}$ is $rn_s$-step interpolatory with the integer shift $(\dm^{rn_s}-1)s_a$, i.e.,
\[
[\sd_{a_1,\ldots,a_r,\dm}^{qrn_s,r} v]((I+\dm^{rn_s}+\cdots+\dm^{(q-1)rn_s})(\dm^{rn_s}-1)s_a
+\dm^{qr n_s}k)=v(k),\qquad \forall\; k\in \Z, q\in \N, v\in \lp{}.
\]
\end{cor}

\bp Sufficiency. By item (i), \eqref{cond:qs} is satisfied and hence by \cref{thm:qs},  the $r$-mask quasi-stationary subdivision scheme is $\mathscr{C}^m$-convergent and $\phi\in \mathscr{C}^m(\R)$. By item (ii), the conditions in item (2) of \cref{thm:int} are satisfied with $\dm$ being replaced by $\dm^r$. Consequently, $\phi$ must be $s_a$-interpolating and the subdivision scheme must be $\infty$-step interpolatory.
If $m_s=0$, then the $r$-mask quasi-stationary subdivision scheme is $rn_s$-interpolatory.

Necessity. If the $\dm^r$-refinable function $\phi$ is $s_a$-interpolating and its subdivision scheme is $\infty$-step interpolatory, we conclude from \cref{thm:int} that item (ii) must hold. On the other hand, because the $r$-mask quasi-stationary $\dm$-subdivision scheme with masks $\{a_1,\ldots, a_r\}$ is $\mathscr{C}^m$-convergent, we conclude from \cref{thm:qs} that item (i) must hold.
\ep

\section{Conclusions}

In this paper, we introduced in \cref{sec:intro} and characterized in \cref{thm:int} all $n_s$-step interpolatory $\dm$-subdivision schemes and their $s_a$-interpolating $\dm$-refinable functions with
$n_s\in \N \cup\{\infty\}$ and
any dilation factor $\dm\in \N\bs\{1\}$.
Furthermore, inspired by \cref{thm:int} and $n_s$-step interpolatory stationary subdivision schemes, we further introduced in \cref{def:sd:qs} the notion of $n_s$-step interpolatory $r$-mask quasi-stationary subdivision schemes with masks $\{a_1,\ldots,a_r\}$, and then we characterized their convergence and smoothness properties in \cref{thm:qs}.
The provided several examples of such $n_s$-step interpolatory $\dm$-subdivision schemes in \cref{sec:alg} demonstrate their potential usefulness and advantages in CAGD, numerical PDEs, and wavelet analysis.

\bigskip

\noindent \textbf{Conflict of interest statement} The author declares no conflict of interest.

\end{document}